\documentclass[10pt,noinfoline
]{article}

\RequirePackage[OT1]{fontenc}
\RequirePackage[
amsthm,amsmath
]{imsart}

\usepackage{graphicx}
\usepackage{latexsym,amsmath}
\usepackage{amsmath,amsthm,amscd}
\usepackage{amsfonts}
\usepackage[psamsfonts]{amssymb}

\usepackage{mathabx} 

\usepackage{enumerate}

\usepackage{url}

\usepackage{tocvsec2}

\usepackage{epstopdf}
\usepackage{wrapfig}
\usepackage{float}

\usepackage{calrsfs}

\usepackage{subfigure}
\usepackage[small]{caption}
\usepackage{hyperref}

\startlocaldefs
\numberwithin{equation}{section}

\theoremstyle{plain} 
\newtheorem{theorem}{Theorem}[section]
\newtheorem{corollary}[theorem]{Corollary}

\newtheorem{proposition}[theorem]{Proposition}

\theoremstyle{definition} 

\theoremstyle{definition} 
\newtheorem{ex}[theorem]{Example}
\newtheorem*{ex*}{Example}
\theoremstyle{remark} 

\theoremstyle{remark} 

\newtheorem*{remark*}{Remark}

\makeatletter
\def\subsubsubsection{\@startsection{subsubsubsection}{4}{\z@}{-3.25ex plus -1ex minus -.2ex}{1.5ex plus .2ex}{\normalsize}}
\makeatother                                   
 
\newcommand{\beqa}{\begin{eqnarray}}
\newcommand{\eeqa}{\end{eqnarray}}
 
\newcommand{\bseq}{\begin{subequations}}
 
\newcommand{\eseq}{\end{subequations}}

\newcommand{\dd}{\partial}

\newcommand{\var}{{\,\operatorname{VaR}}}
\newcommand{\cvar}{{\,\operatorname{CVaR}}}

\newcommand{\id}{{\,\operatorname{id}}}

\newcommand{\opt}{{\,\operatorname{opt}}}

\renewcommand{\dd}{{\,\operatorname{d}}}

\newcommand{\supp}{\operatorname{supp}}

\newcommand{\amin}{\operatorname{arg\,min}\limits_{t\in\R}B_\al(X;p)(t)}
\renewcommand{\amin}{\mathop{\operatorname{argmin}}\limits_{t\in\R}B_\al(X;p)(t)}
\newcommand{\ql}{{}_{\al-1}Q(X;p)}
\newcommand{\qlbe}{{}_{\be-1}Q(X;p)}
\newcommand{\qlz}{{}_0Q(X;p)}
\newcommand{\pl}[1]{{}_{#1-1}P(X;t)}

\newcommand{\al}{\alpha}

\newcommand{\Ga}{\Gamma}
\newcommand{\si}{\sigma}

\newcommand{\ka}{\kappa}
\newcommand{\la}{\lambda}
\newcommand{\ga}{\gamma}
\newcommand{\de}{\delta}
\newcommand{\be}{\beta}

\newcommand{\La}{\Lambda}

\renewcommand{\th}{\theta}
\renewcommand{\Psi}{\overline{\Phi}}

\newcommand{\alle}{\overset{\al+1}{\le}}

\newcommand{\st}{\mathrel{\overset{\mathrm{st}}{\le}}}
\newcommand{\sst}{\mathrel{\overset{\mathrm{st}}{<}}}

\newcommand{\LL}{\mathcal{L}}

\newcommand{\XX}{\mathcal{X}}

\newcommand{\ii}[1]{\,\mathsf{I}\{#1\}} 

\newcommand{\D}[1]{D^{(#1)}}  
\renewcommand{\D}{\mathrel{\overset{\mathrm{D}}\to}}
\renewcommand{\D}[1]{\mathrel{\underset{#1}{\overset{\mathrm{D}}\longrightarrow}}}

\renewcommand{\P}{\operatorname{\mathsf{P}}} 
 
\newcommand{\E}{\operatorname{\mathsf{E}}}

\newcommand{\R}{\mathbb{R}}
\newcommand{\N}{\mathbb{N}}

\renewcommand{\H}[1]{\mathcal{H}^{#1}}

\newcommand{\vp}{\varepsilon}

\newcommand{\tA}{{\tilde{A}}}

\newcommand{\tp}{{\tilde{p}}}

\newcommand{\tX}{{\tilde{X}}}
\newcommand{\tY}{{\tilde{Y}}}

\newcommand{\M}{{\mathrm{M}}}

\renewcommand{\le}{\leqslant}
\renewcommand{\ge}{\geqslant}
\renewcommand{\leq}{\leqslant}

\newcommand{\li}[1]{{{#1}^*}^{-1}}

  \newcommand{\fJ}{\operatorname{\raisebox{.8pt}{\fbox{\scriptsize \mathsf{H}}}}}
 \renewcommand{\fJ}{\mathrel{\raisebox{.0pt}{\fbox{\scriptsize H}}}}
 \renewcommand{\fJ}{\mathrel{\raisebox{.0pt}{\fbox{\scriptsize $\operatorname{H}$}}}}
 \renewcommand{\fJ}{\mathrel{\raisebox{.0pt}{\fbox{\scriptsize $\operatorname{\mathsf{H}}$}}}}

\endlocaldefs

\begin{document}

\begin{frontmatter}

\title{An optimal three-way stable and monotonic spectrum of bounds on 
quantiles: \\ 
a spectrum of coherent measures of financial risk and economic inequality  
}
\runtitle{Bounds on quantiles}

%

\begin{aug}
\author{\fnms{Iosif} \snm{Pinelis}
}
\runauthor{Iosif Pinelis}


\address{Department of Mathematical Sciences\\
Michigan Technological University\\
Houghton, Michigan 49931, USA\\
E-mail: \printead[ipinelis@mtu.edu]{e1}}
\end{aug}

\begin{abstract}
A certain spectrum $\big(P_\al(X;x)\big)_{\al\in[0,\infty]}$ of upper bounds on the tail probability $\P(X\ge x)$, with $P_0(X;x)=\P(X\ge x)$ and $P_\infty(X;x)$ being the best possible exponential upper bound on $\P(X\ge x)$, is 
shown to be stable and monotonic in $\al$, $x$, and $X$, where $x$ is a real number and $X$ is a random variable.  
The bounds $P_\al(X;x)$ are optimal values in certain minimization problems. 
The corresponding spectrum $\big(Q_\al(X;p)\big)_{\al\in[0,\infty]}$ of upper bounds on the $(1-p)$-quantile of $X$ is shown as well to be stable and monotonic in $\al$, $p$, and $X$, with $Q_0(X;p)$ equal the largest $(1-p)$-quantile of $X$. 
In fact, $P_\al(X;x)$ and $Q_\al(X;p)$ are nondecreasing in $X$ with respect to the stochastic dominance of any order $\ga\in[1,\al+1]$. 
It is shown that for small enough values of $p$ the quantile bounds $Q_\al(X;p)$ are close enough to the true quantiles $Q_0(X;p)$ provided that the right tail of the distribution of $X$ is light enough and regular enough, depending on $\al$. 
Moreover, it is shown that the quantile bounds $Q_\al(X;p)$ 
possess the crucial property of the subadditivity in $X$ if $\al\in[1,\infty]$, as well as the 
positive homogeneity and translation invariance properties, and thus constitute a continuous spectrum of so-called coherent measures of risk. 
A number of other useful properties of the bounds $P_\al(X;x)$ and $Q_\al(X;p)$ are established. In particular, it is shown that, quite similarly to the bounds $P_\al(X;x)$ on the tail probabilities, the quantile bounds $Q_\al(X;p)$ are the optimal values in certain minimization problems. 
This allows for a comparatively easy incorporation of the bounds $P_\al(X;x)$ and $Q_\al(X;p)$ into more specialized optimization problems, with additional restrictions, say on the distribution of the random variable $X$. 
It is shown that the mentioned minimization problems for which $P_\al(X;x)$ and $Q_\al(X;p)$ are the optimal values are in a certain sense dual to each other; in the special case $\al=\infty$ this corresponds to the bilinear Legendre--Fenchel duality. 
In finance, the $(1-p)$-quantile $Q_0(X;p)$ is known as the value-at-risk (VaR), whereas the value of $Q_1(X;p)$ is known as the conditional value-at-risk (CVaR) and also as  
the expected shortfall (ES), average value-at-risk (AVaR), and expected tail loss (ETL). 
Also in the present paper, a short proof of the well-known Rockafellar--Uryasev--Pflug theorem that VaR is a minimizer in the Rockafellar--Uryasev variational representation of CVaR is provided. More generally, the minimizers in the variational representation of $Q_\al(X;p)$ are described in detail for any $\al\in[1,\infty)$.   
A generalization of the Cillo--Delquie necessary and sufficient condition for the so-called mean-risk (M-R) 
to be nondecreasing with respect to the stochastic dominance of order $1$ is presented, with a short proof. 
Moreover, a necessary and sufficient condition for the M-R measure  
to be coherent is given. 
It is shown that the quantile bounds $Q_\al(X;p)$ can be used as measures of economic inequality. 
The spectrum parameter $\al$ may be considered an index of sensitivity: the greater is the value of $\al$, the greater is the sensitivity of the function $Q_\al(\cdot;p)$ to risk/inequality. 
The problems of effective computation of $P_\al(X;x)$ and $Q_\al(X;p)$ are considered. 
\end{abstract}

  
%

\setattribute{keyword}{AMS}{AMS 2010 subject classifications:}











\end{frontmatter}

\settocdepth{chapter}

\tableofcontents 

\settocdepth{subsubsection}

\theoremstyle{plain} 




\section{An optimal three-way stable and three-way monotonic spectrum of upper bounds on tail probabilities}\label{tails} 


Consider the family $(h_\al)_{\al\in[0,\infty]}$ of functions $h_\al\colon\R\to\R$ given by the formula 
\begin{equation}\label{eq:f}
	h_\al(u):=
	\left\{
	\begin{alignedat}{2}
	&\ii{u\ge0}&&\text{\ \ if }\al=0, \\ 
	&(1+u/\al)_+^\al&&\text{\ \ if }0<\al<\infty, \\ 
	&e^u&&\text{\ \ if }\al=\infty 
	\end{alignedat}
	\right.
\end{equation}
for all $u\in\R$. 
Here, as usual, $\ii{\cdot}$ denotes the indicator function, $u_+:=0\vee u$ and $u_+^\al:=(u_+)^\al$ for all real $u$. 

Obviously, the function $h_\al$ is nonnegative and nondecreasing for each $\al\in[0,\infty]$,  and it is also continuous for each $\al\in(0,\infty]$. 
Moreover, it is easy to see that, for each $u\in\R$, 
\begin{equation}\label{eq:f incr,cont}
	h_\al(u) \text{ is nondecreasing and continuous in }\al\in[0,\infty].  
\end{equation}

Next, let us use the functions $h_\al$ as generalized moment functions and thus introduce the generalized moments 
\begin{equation}\label{eq:A new}
	A_\al(X;x)(\la):=\E h_\al\big(\la(X-x)\big).   
\end{equation}
Here and in what follows, unless otherwise specified, $X$ is any random variable (r.v.), $x\in\R$, $\al\in[0,\infty]$, and $\la\in(0,\infty)$. 
Since $h_\al\ge0$, the expectation in \eqref{eq:A new} is always defined, but may take the value $\infty$. 
It may be noted that in the particular case $\al=0$ one has 
\begin{equation}\label{eq:A_0}
	A_0(X;x)(\la)=\P(X\ge x), 
\end{equation}
which does not actually depend on $\la\in(0,\infty)$. 

Now one can introduce the expressions 
\begin{equation}\label{eq:P new}
P_\al(X;x):=\inf_{\la\in(0,\infty)}A_\al(X;x)(\la)
=\left\{
	\begin{alignedat}{2}
	&\P(X\ge x)&&\text{\ \ if }\al=0, \\ 
	&\inf_{\la\in(0,\infty)}\E\big(1+\la(X-x)/\al)_+^\al&&\text{\ \ if }0<\al<\infty, \\ 
	&\inf_{\la\in(0,\infty)}\E e^{\la(X-x)}&&\text{\ \ if }\al=\infty.  
	\end{alignedat}
\right. 
\end{equation}
By \eqref{eq:f incr,cont}, $A_\al(X;x)(\la)$ and $P_\al(X;x)$ are nondecreasing in $\al\in[0,\infty]$.  
In particular, 
\begin{equation}\label{eq:0<al}
	P_0(X;x)=\P(X\ge x)\le P_\al(X;x).   
\end{equation}
It will be shown later (see Proposition~\ref{prop:P cont in al}) that $P_\al(X;x)$ also largely inherits the property of $h_\al(u)$ of being continuous in $\al\in[0,\infty]$. 

The definition \eqref{eq:P new} can be rewritten as  
\begin{equation}\label{eq:P old}
P_\al(X;x)=\inf_{t\in T_\al}\tA_\al(X;x)(t),
\end{equation}
where 
\begin{align}
T_\al&:=
\left\{
\begin{alignedat}{2}
&\R &&\text{\ \ if $\al\in[0,\infty)$,}\\ 
&(0,\infty) &&\text{\ \ if $\al=\infty$}  
\end{alignedat}
\right. \label{eq:T_al}\\
\intertext{and} 
\tA_\al(X;x)(t)&:=
\left\{
\begin{alignedat}{2}
&\frac{\E(X-t)_+^\al}{(x-t)_+^\al} && \text{ if } \al\in[0,\infty), \\ 
&\E e^{(X-x)/t} && \text{ if } \al=\infty;    
\end{alignedat}
\right. \label{eq:A old} 
\end{align}
here and subsequently, we also use the conventions $0^0:=0$ and $\frac a0:=\infty$ for all $a\in[0,\infty]$. 
The alternative representation \eqref{eq:P old} of $P_\al(X;x)$ follows because (i) $A_\al(X;x)(\la)=\tA_\al(X;x)(x-\al/\la)$ for $\al\in(0,\infty)$, 
(ii) $A_\infty(X;x)(\la)=\tA_\infty(X;x)(1/\la)$, and (iii) $P_0(X;x)=
\P(X\ge x) 
=\inf_{t\in(-\infty,x)}\P(X>t)=\inf_{t\in(-\infty,x)}\tA_0(X;x)(t)$. 

In view of \eqref{eq:P old}, one can see (cf.\ \cite[Corollary~2.3]{pin98}) that, for each $\al\in[0,\infty]$, $P_\al(X;x)$ is the optimal (that is, least possible) upper bound on the tail probability $\P(X\ge x)$ given the generalized moments $\E g_{\al;t}(X)$ for all $t\in T_\al$, where 
\begin{equation}\label{eq:g_al,t}
g_{\al;t}(u):=
\left\{
\begin{alignedat}{2}
&(u-t)_+^\al&&\text{ if }\al\in[0,\infty), \\ 
&e^{u/t}&&\text{ if }\al=\infty.  
\end{alignedat}
\right.
\end{equation}
In fact (cf.\ e.g.\ \cite[Proposition 3.3]{pin-hoeff}), the bound $P_\al(X;x)$ remains optimal given the larger class of generalized moments $\E g(X)$ for all functions $g\in\H\al$, where 
\begin{equation}\label{eq:H}
\H\al:=\big\{g\in\R^\R\colon
g(u)=\textstyle{\int\nolimits_\R} g_{\al;t}(u)\,\mu(dt)
\text{ for some $\mu\in\M_\al$ and all $u\in\R$}
\big\}, 
\end{equation}
$\M_\al$ denotes the set of all nonnegative Borel measures on $T_\al$, and, as usual, $\R^\R$ stands for the set of all real-valued functions on $\R$. 
By \cite[Proposition 1(ii)]{pin99} and \cite[Proposition~3.4]{pin-hoeff}, 
\begin{equation}\label{eq:F-al-beta}
0\le\al<\be\le\infty\quad\text{implies}\quad\H\al\supseteq\H\be.
\end{equation}
This provides the other way to come to the mentioned conclusion that  
\begin{equation}\label{eq:mono in al}
	P_\al(X;x)\text{ is nondecreasing in }\al\in[0,\infty]. 
\end{equation}

By \cite[Proposition~1.1]{binom}, the class $\H\al$ of generalized moment functions can be characterized as follows in the case when $\al$ is a natural number: for any $g\in\R^\R$, one has $g\in\H\al$ if and only if $g$ has finite derivatives $g^{(0)}:=g,g^{(1)}:=g',\dots,g^{(\al-1)}$ on $\R$ such that $g^{(\al-1)}$ is convex on $\R$ and $\lim_{x\to-\infty}g^{(j)}(x)=0$ for $j=0,1,\dots,\al-1$. 
Also, by \cite[Proposition~3.4]{pin-hoeff}, $g\in\H\infty$ if and only if $g$ is infinitely differentiable 
on $\R$, and $g^{(j)}\ge0$ on $\R$ and $\lim_{x\to-\infty}g^{(j)}(x)=0$ for all $j = 0,1,\dots$. 

Thus, the greater the value of $\al$, the narrower and easier to deal with is the class $\H\al$ and the smoother are the functions comprising $\H\al$. However, the greater the value of $\al$, the farther away is the bound $P_\al(X;x)$ from the true tail probability $\P(X\ge x)$. 

Of the bounds $P_\al(X;x)$, the loosest and easiest one to get is $P_\infty(X;x)$, the so-called exponential upper bound on the tail probability $\P(X\ge x)$. It is used very widely, in particular when $X$ is the sum of independent r.v.'s $X_i$, in which case one can rely on the factorization 
$A_\al(X;x)(\la)=e^{-\la x}\prod_i\E e^{\la X_i}$. 
A bound very similar to $P_3(X;x)$ was introduced in \cite{eaton2} in the case when $X$ the sum of independent bounded r.v.'s; see also \cite{T2original,dufour-hallin,T2}. 
For any $\al\in(0,\infty)$, the bound $P_\al(X;x)$ is a special case of a more general bound given in \cite[Corollary~2.3]{pin98}; see also \cite[Theorem~2.5]{pin98}. For some of the further developments in this direction see \cite{pin99,bent-liet02,bent-ap,normal,asymm,bent-64pp,pin-hoeff}. 
The papers mentioned in this paragraph used the representation \eqref{eq:P old} of $P_\al(X;x)$, rather than the new representation \eqref{eq:P new}. 
The new representation appears, not only of more unifying form, but also more convenient as far as such properties of $P_\al(X;x)$ as the monotonicity in $\al$ and the continuity in $\al$ and in $X$ are concerned; cf.\ \eqref{eq:f incr,cont} and the proofs of Propositions~\ref{prop:P cont in al} and \ref{prop:P cont in X}; those proofs, as well as proofs of most of the other statements in this paper, are given in Appendix~\ref{proofs}. 
Yet another advantage of the representation \eqref{eq:P new} is that, for $\al\in[1,\infty)$, the function $A_\al(X;x)(\cdot)$ inherits the convexity property of $h_\al$, which facilitates the minimization of $A_\al(X;x)(\la)$ in $\la$, as needed to find $P_\al(X;x)$ by \eqref{eq:P new}; relevant details on the remaining ``difficult case'' $\al\in(0,1)$ can be found in Section~\ref{P comput}. 

On the other hand, the ``old'' representation \eqref{eq:P old} of $P_\al(X;x)$ is more instrumental in establishing the mentioned connection with the classes 
$\H\al$ of generalized moment functions; in proving part~\eqref{it:conc} of Proposition~\ref{prop:P,+}; and in discovering and 
proving Theorem~\ref{prop:Q=}.

\begin{center}***\end{center}

Some of the more elementary properties of $P_\al(X;x)$ are presented in 

\begin{proposition}\label{prop:P} \ 
\begin{enumerate}[(i)]
	\item $P_\al(X;x)$ is nonincreasing in $x\in\R$. 
	\item If $\al\in(0,\infty)$ and $\E X_+^\al=\infty$, then $P_\al(X;x)=\infty$ for all $x\in\R$. 
	\item If $\al=\infty$ and $\E e^{\la X}=\infty$ for all real $\la>0$, then $P_\infty(X;x)=\infty$ for all $x\in\R$. 
	\item If $\al\in(0,\infty)$ and $\E X_+^\al<\infty$, then $P_\al(X;x)\to1$ as $x\to-\infty$ and $P_\al(X;x)\to0$ as $x\to\infty$, so that $0\le P_\al(X;x)\le1$ for all $x\in\R$. 
	\item If $\al=\infty$ and $\E e^{\la_0 X}<\infty$ for some real $\la_0>0$, then $P_\al(X;x)\to1$ as $x\to-\infty$ and $P_\al(X;x)\to0$ as $x\to\infty$, so that $0\le P_\al(X;x)\le1$ for all $x\in\R$. 
\end{enumerate}
\end{proposition}


In view of Proposition~\ref{prop:P}, it will be henceforth assumed by default that the tail bounds $P_\al(X;x)$ -- as well as the quantile bounds $Q_\al(X;p)$, to be introduced in Section~\ref{quantiles}, and also the corresponding expressions $A_\al(X;x)(\la)$, $\tA_\al(X;x)(t)$, $B_\al(X;p)(t)$, and $\ql$ as in \eqref{eq:A new}, \eqref{eq:A old}, \eqref{eq:B}, and \eqref{eq:ql} -- are defined and considered only for r.v.'s $X\in\XX_\al$  (unless indicated otherwise), where 
\begin{equation*}
	\XX_\al:=
	\left\{
	\begin{alignedat}{2}
	&\XX&&\text{\ \ if }\al=0, \\ 
	&\big\{X\in\XX\colon\E X_+^\al<\infty\big\}&&\text{\ \ if }\al\in(0,\infty), \\ 
	&\big\{X\in\XX\colon\La_X\ne\emptyset\big\}&&\text{\ \ if }\al=\infty,  
	\end{alignedat}
	\right.
\end{equation*}
$\XX$ is the set of all real-valued r.v.'s on a given probability space (implicit in this paper), 
and  
\begin{equation}\label{eq:La_X}
	\La_X:=\big\{\la\in(0,\infty)\colon\E e^{\la X}<\infty\big\}.  
\end{equation}
Observe that the set $\XX_\al$ is a convex cone: for any $\th\in[0,\infty)$ and any $X$ and $Y$ in $\XX_\al$, the r.v.'s $\th X$ and $X+Y$ are in $\XX_\al$. Indeed, the conclusion that $\th X\in\XX_\al$ for any $\th\in[0,\infty)$ and $X\in\XX_\al$ is obvious. Concerning the conclusion that $X+Y\in\XX_\al$ for any $X$ and $Y$ in $\XX_\al$, use the inequalities $\E(X+Y)_+^\al\le\E(X_++Y_+)^\al\le\E X_+^\al+\E Y_+^\al$ if $\al\in(0,1]$, Minkowski's inequality $\|X_++Y_+\|_\al\le\|X_+\|_\al+\|Y_+\|_\al$ if $\al\in[1,\infty)$, and H\"older's inequality $\E e^{\la(X+Y)}\le\big(\E e^{p\la X}\big)^{1/p} 
\big(\E e^{q\la X}\big)^{1/q}$ for any positive $p$ and $q$ such that $\frac1p+\frac1q=1$. 
Here, as usual, $\|Z\|_\al:=(\E|Z|^\al)^{1/\al}$, the $\LL^\al$-norm of a r.v.\ $Z$ -- which is actually a norm if and only if $\al\ge1$. 
Also, it is obvious that the cone $\XX_\al$ contains all real constants. 

It follows 
from Proposition~\ref{prop:P} and \eqref{eq:0<al} that 
\begin{equation}\label{eq:P decr in x}
		\text{$P_\al(X;x)$ is nonincreasing in $x\in\R$, 
		with $P_\al(X;(-\infty)+)=1$ and $P_\al(X;\infty-)=0$.}
\end{equation}
Here, as usual, $f(a+)$ an $f(a-)$ denote the right and left limits of $f$ at $a$. 

One can say more in this respect. 
To do that, introduce 
\begin{equation}\label{eq:x_*,p_*}
	x_*:=x_{*,X}:=\sup\supp X\quad\text{and}\quad p_*:=p_{*,X}:=\P(X=x_*). 
\end{equation}
Here, as usual, $\supp X$ denotes the support set of (the distribution of the r.v.) $X$; speaking somewhat loosely, $x_*$ is the maximum value taken by the r.v.\ $X$, and $p_*$ is the probability with which this value is taken. It is of course possible that $x_*=\infty$, in which case necessarily $p_*=0$, since the r.v.\ $X$ was assumed to be 
real-valued. 


Introduce also 
\begin{equation}\label{eq:x_al}
	x_\al:=x_{\al,X}:=\inf E_\al(1),
\end{equation}
where 
\begin{equation}\label{eq:J_al}
	E_\al(p):=E_{\al,X}(p):=\{x\in\R\colon P_\al(X;x)<p\}. 
\end{equation}
Recall that, according to the standard convention, for any subset $E$ of $\R$, $\inf E=\infty$ if and only if $E=\emptyset$. 

Now one can state 


\begin{proposition}\label{prop:P,+} \ 
\begin{enumerate}[(i)]
	\item\label{it:x ge x_*} 
	For all $x\in[x_*,\infty)$ one has 
$
	P_\al(X;x)=P_0(X;x)=\P(X\ge x)=\P(X=x)=p_*\ii{x=x_*}. 
$
	\item\label{it:x<x_*} For all $x\in(-\infty,x_*)$ one has 
$P_\al(X;x)>0$. 
	\item\label{it:conc} The function $(-\infty,x_*]\cap\R\ni x\mapsto P_\al(X;x)^{-1/\al}$ is continuous and convex if $\al\in(0,\infty)$; 
we use the conventions $0^{-a}:=\infty$  and $\infty^{-a}:=0$ for all real $a>0$; 
concerning the continuity of functions with values in the set $[0,\infty]$, 
we use the natural topology on this set.  	
Also, the function $(-\infty,x_*]\cap\R\ni x\mapsto-\ln P_\infty(X;x)$ is continuous and convex, with the convention $\ln0:=-\infty$. 
	\item\label{it:cont} 
If $\al\in(0,\infty]$ then the function 
$(-\infty,x_*]\cap\R\ni x\mapsto P_\al(X;x)$ is continuous. 
	\item\label{it:left-cont} The function 
$\R\ni x\mapsto P_\al(X;x)$ is left-continuous. 
	\item\label{it:x_al incr} $x_\al$ is nondecreasing in $\al\in[0,\infty]$, and $x_\al<\infty$ for all $\al\in[0,\infty]$. 
	\item\label{it:EX}
If $\al\in[1,\infty]$ then $x_\al=\E X$; even 
for $X\in\XX_\al$, it is of course possible that $\E X=-\infty$, in which case $P_\al(X;x)<1$ for all real $x$. 
	\item\label{it:x_al<x_*} $x_\al\le x_*$, and $x_\al=x_*$ if and only if $p_*=1$. 
	\item\label{it:J_al} $E_\al(1)=(x_\al,\infty)\ne\emptyset$.  
	\item\label{it:P_al,x<x_al} $P_\al(X;x)=1$ for all $x\in(-\infty,x_\al]$.  
	\item\label{it:str decr} 
If $\al\in(0,\infty]$ then $P_\al(X;x)$ is strictly decreasing in $x\in[x_\al,x_*]\cap\R$.   	
\end{enumerate}  
\end{proposition}
 
This proposition will be useful when establishing continuity properties of the quantile bounds considered in Section~\ref{quantiles}. 
For $\al\in(1,\infty)$, parts \eqref{it:x ge x_*}, \eqref{it:cont}, \eqref{it:EX}, \eqref{it:P_al,x<x_al}, and \eqref{it:str decr} of Proposition~\ref{prop:P,+} are contained in \cite[Proposition~3.2]{pin-hoeff}.  

One may also note here that, by \eqref{eq:P decr in x} and part~\eqref{it:left-cont} of Proposition~\ref{prop:P,+}, the function 
$P_\al(X;\cdot)$ may be regarded as the tail function of some r.v.\ $Z_\al$: $P_\al(X;u)=\P(Z_\al\ge u)$ for all real $u$. 

\begin{ex}\label{ex:P,X_ab} 
Some parts of Propositions~\ref{prop:P} and \ref{prop:P,+} are illustrated in the following picture with graphs of the function $P_\al(X;\cdot)$ for various values of $\al$ in the important case when the r.v.\ $X$ takes only two values. Then, by the translation invariance property stated below in Theorem~\ref{th:P}, without loss of generality (w.l.o.g.) $\E X=0$. 
Thus, $X=X_{a,b}$, where $a$ and $b$ are positive real numbers and $X_{a,b}$ is a r.v.\ with the uniquely determined zero-mean distribution on the set $\{-a,b\}$. 

\vspace*{6pt}
\noindent
\begin{parbox}{.38\textwidth}
	{ \includegraphics[width=.36\textwidth]{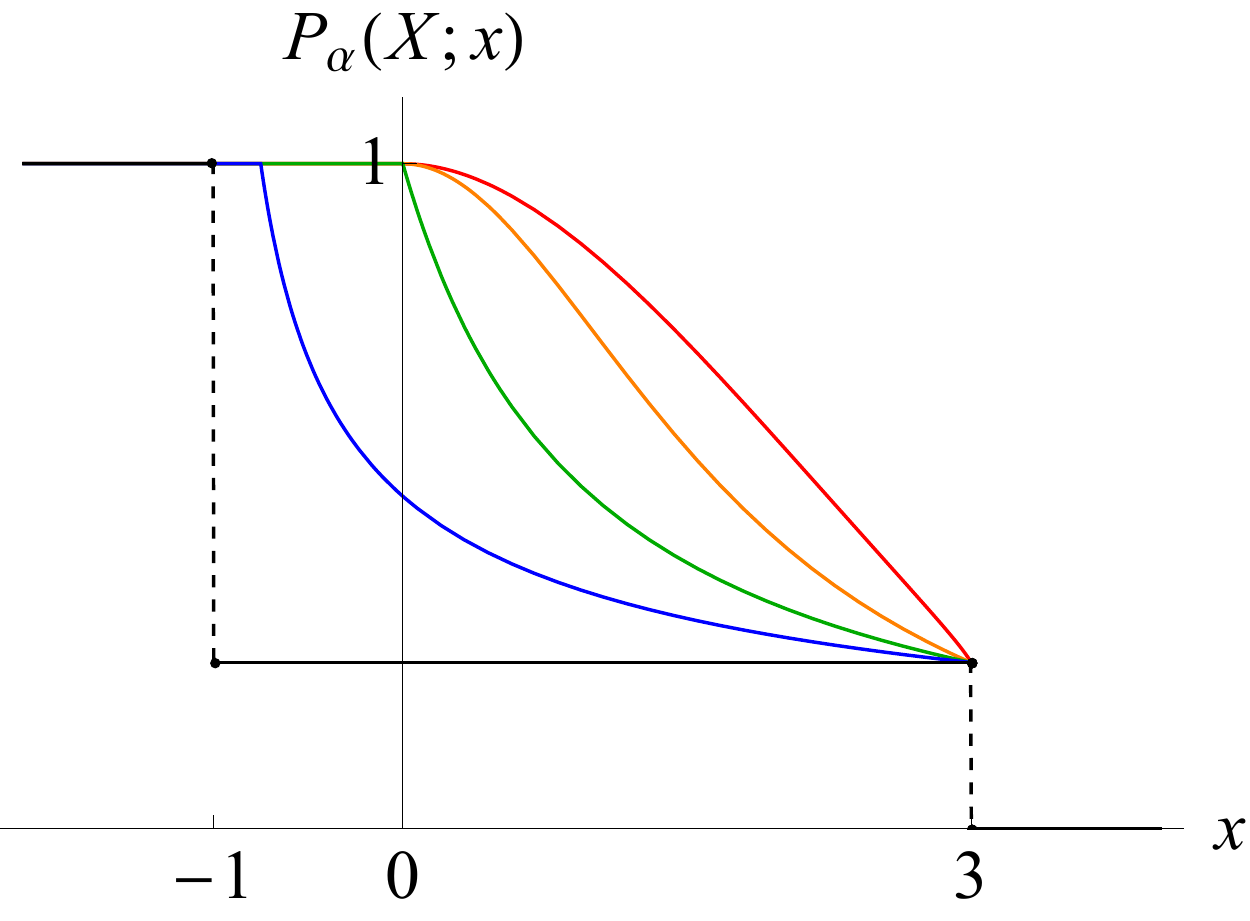} }
\end{parbox}
\begin{parbox}{.6\textwidth}
{
Explicit expressions for $P_\al(X_{a,b};x)$ were given in \cite[Section~3, Example]{pin-hoeff} in the cases $\al\in(1,\infty)$ and $\al=\infty$. 
In the case $\al\in(0,1]$, one can see that 
$P_\al(X_{a,b};x)$ equals $\frac a{a+b}\, (\frac{a+b}{a+x})^{\al }$ 
for $x\in[x_\al,x_*]$, $1$ for $x\in(-\infty,x_\al]$, and $0$ for $x\in(x_*,\infty)$, with $x_\al=a \big((\frac{a}{a+b})^{\frac{1}{\al }-1}-1\big)$ and $x_*=b$. 
Graphs $\{\big(x,P_\al(X_{a,b};x)\big)\colon -2\le x\le 4\}$ are shown here on the left, with 
$a=1$ and $b=3$, for values of $\al$ equal $0$ (black), $\frac12$ (blue), $1$ (green), $2$ (orange), and $\infty$ (red). 
In particular, here $x_{1/2}=-\frac34$. 
}
\end{parbox}
\end{ex}

\begin{proposition}\label{prop:P cont in al} \ 
$P_\al(X;x)$ is continuous in $\al\in[0,\infty]$ in the following sense: 
Suppose that 
$(\al_n)$ is any sequence in $[0,\infty)$ converging to $\al\in[0,\infty]$, with $\be:=\sup_n\al_n$ 
and $X\in\XX_\be$; 
then $P_{\al_n}(X;x)\to P_\al(X;x)$.   
\end{proposition}


In view of parts (ii) and (iii) of Proposition~\ref{prop:P}, the 
condition 
$X\in\XX_\be$ in Proposition~\ref{prop:P cont in al} is essential. 


Let us now turn to the question of stability of $P_\al(X;x)$ with respect to (the distribution of) $X$. 
First here, recall that one of a number of mutually equivalent definitions of the convergence in distribution, $X_n\D{n\to\infty} X$, of a sequence of r.v.'s $X_n$ to a r.v.\ $X$ is the following:  
$\P(X_n\ge x)\underset{n\to\infty}\longrightarrow\P(X\ge x)$ for all real $x$ such that $\P(X=x)=0$. 

We shall also need the following uniform integrability condition: 
\begin{align}								&\text{$\sup_n\E(X_n)_+^\al\ii{X_n>N}\underset{N\to\infty}\longrightarrow0$ if $\al\in(0,\infty)$,} 	\label{it:int,al<infty} \\  
	&\text{$\sup_n\E e^{\la X_n}\ii{X_n>N}\underset{N\to\infty}\longrightarrow0$ for each $\la\in\La_X$ if $\al=\infty$}.  	\label{it:int,al=infty} 
\end{align}
 
\begin{proposition}\label{prop:P cont in X} \ 
Suppose that $\al\in(0,\infty]$. 
Then
$P_\al(X;x)$ is continuous in $X$ in the following sense. 
Take any sequence $(X_n)_{n\in\N}$ of real-valued r.v.'s such that $X_n\D{n\to\infty} X$ and 
the uniform integrability condition \eqref{it:int,al<infty}--\eqref{it:int,al=infty} is satisfied. 
Then one has the following. 
\begin{enumerate}[(i)]
	\item\label{it:P, n->infty} 
The convergence 	
\begin{equation}\label{eq:P, n->infty}
	P_\al(X_n;x)\underset{n\to\infty}\longrightarrow P_\al(X;x)  
\end{equation}  
takes place for all real $x\ne x_*$, where $x_*=x_{*,X}$ as in \eqref{eq:x_*,p_*}; thus, by parts \eqref{it:x ge x_*} and \eqref{it:cont} of Proposition~\ref{prop:P,+}, \eqref{eq:P, n->infty} holds for all real $x$ that are points of continuity of the function $P_\al(X;\,\cdot)$. 
	\item\label{it:P,x_*} 
The convergence \eqref{eq:P, n->infty} holds for $x=x_*$ as well provided that $\P(X_n=x_*)\underset{n\to\infty}\longrightarrow\P(X=x_*)$. 
In particular, \eqref{eq:P, n->infty} holds for $x=x_*$ if $\P(X=x_*)=0$.   
\end{enumerate} 
\end{proposition}

Note that in the case $\al=0$ the convergence \eqref{eq:P, n->infty} may fail to hold, not only for $x=x_*$, but for all real $x$ such that $\P(X=x)>0$. 

\begin{center}***\end{center}

Let us now discuss matters of monotonicity of $P_\al(X;x)$ in $X$, with respect to various orders on the mentioned set $\XX$ of all real-valued r.v.'s $X$. 
Using the family of function classes $\H\al$, defined by \eqref{eq:H}, one can introduce a family of stochastic orders, say $\alle$, on the set $\XX$ 
by the formula 
\begin{equation*}
	X\alle Y\overset{\text{def}}\iff \E g(X)\le\E g(Y)\text{ for all }g\in\H\al, 
\end{equation*}
where $\al\in[0,\infty]$ and $X$ and $Y$ are in $\XX$. 
To avoid using the term ``order'' with two different meanings in one phrase, let us refer to the relation $\alle$ as the \emph{stochastic dominance of order $\al+1$}, rather than the stochastic order of order $\al+1$. 
In view of \eqref{eq:H}, it is clear that 
\begin{equation}\label{eq:alle}
		X\alle Y\iff \E g_{\al;t}(X)\le\E g_{\al;t}(Y)\text{ for all }t\in T_\al,   
\end{equation}
so that, in the case when $\al=m-1$ for some natural number $m$, the order $\alle$ coincides with the ``$m$-increasing-convex'' order $\leq_{m-\text{icx}}$ as defined e.g.\ in \cite[page~206]{shaked-shanti}. 
In particular, 
\begin{equation}\label{eq:st}
\begin{aligned}
	X\overset1\le Y&\iff \E g(X)\le\E g(Y)\text{ for all nonnegative nondecreasing functions }   g\colon\R\to\R \\ 
	&\iff \P(X>t)\le\P(Y>t)\text{ for all }t\in\R \\ 
		&\iff \P(X\ge t)\le\P(Y\ge t)\text{ for all }t\in\R
	\iff X\st Y, 
\end{aligned}	
\end{equation}
where $\st$ denotes the usual stochastic dominance of order $1$, and 
\begin{equation}\label{eq:le2}
\begin{aligned}
	X\overset2\le Y&\iff \E g(X)\le\E g(Y)\text{ for all nonnegative nondecreasing convex functions } g\colon\R\to\R \\ 
	&\iff \E(X-t)_+\le\E(Y-t)_+\text{ for all }t\in\R, 
\end{aligned}	
\end{equation}
so that $\overset2\le$ coincides with the usual stochastic dominance of order $2$. 
Also, 
\begin{equation}\label{eq:st iff}
	\text{$X\st Y$ iff $X_1\le Y_1$ for some r.v.'s $X_1$ and $Y_1$ which are copies in distribution of $X$ and $Y$,}
\end{equation}
respectively.  

By \eqref{eq:F-al-beta}, the orders $\alle$ are graded in the sense that 
\begin{equation}\label{eq:le-al-beta}
\text{if $X\alle Y$ for some $\al\in[0,\infty]$, then $X\overset{\be+1}\le Y$ for all $\be\in[\al,\infty]$.} 
\end{equation}

A stochastic order, which is a ``mirror image'' of the order $\alle$, but only for nonnegative r.v.'s, 
was presented by Fishburn in \cite{fishburn80}; note \cite[Theorem~2]{fishburn80} on the relation with a ``bounded'' version of this order, previously introduced and studied in \cite{fishburn76}. 
Denoting the corresponding 
Fishburn \cite{fishburn80} order by $\le_{\al+1}$, one has 
\begin{equation}\label{eq:fishburn}
X\le_{\al+1}Y\iff (-Y)\alle(-X),    
\end{equation}
for nonnegative r.v.'s $X$ and $Y$. 
However, as shown in this paper (recall Proposition~\ref{prop:P}) the condition of the nonnegativity of the r.v.'s is not essential; without it, one can either deal with infinite expected values or, alternatively, require that they be finite. 
The case when $\al$ is an integer was considered, in a different form, in \cite{atkinson08}. 

One may also consider 
the order $\le_\al^{-1}$ 
defined by the condition that $X\le_\al^{-1} Y$ if and only if $X$ and $Y$ are nonnegative r.v.'s and 
$F_X^{(-\al)}(p)\le F_Y^{(-\al)}(p)$ for all $p\in(0,1)$, where $\al\in(0,\infty)$, 
\begin{align}
	F_X^{(-\al)}(p)&:=\frac1{\Ga(\al)}\int_{[0,p)}(p-u)^{\al-1}\dd F_X^{-1}(u), 
	\label{eq:check Q} \\
	F_X^{-1}(p)&:=\inf\{x\in[0,\infty)\colon\P(X\le x)\ge p\}=-Q(-X;p)    
	\label{eq:F^{-1}}
\end{align}
with $Q(\cdot;\cdot)$ as in \eqref{eq:Q_0}, and the integral in \eqref{eq:check Q} is understood as the Lebesgue integral with respect to the nonnegative Borel measure $\mu_X^{-1}$ on $[0,1)$ defined by the condition that $\mu_X^{-1}\big([0,p)\big)=F_X^{-1}(p)$ for all $p\in(0,1)$; cf.\ \cite{muliere-scarsini,ortobelli_etal06}. 
Note that $F_X^{(-1)}(p)=F_X^{-1}(p)$. 
For nonnegative r.v.'s, the order $\le_{\al+1}^{-1}$  
coincides with the order $\le_{\al+1}$ if $\al\in\{0,1\}$; again see  \cite{muliere-scarsini,ortobelli_etal06}. 
%
Even for nonnegative r.v.'s, it seems unclear how the orders $\le_{\al+1}$ and $\le_{\al+1}^{-1}$ relate to each other for positive real $\al\ne1$; see e.g.\ the discussion following Proposition~1 in \cite{muliere-scarsini} and Note\,${}^1$ on page~100 in \cite{ortobelli_etal09}. 
%


The following theorem summarizes some of the properties of the tail probability bounds $P_\al(X;x)$ established above and also adds a few simple properties of these bounds. 

\begin{theorem}\label{th:P}
The following properties of the tail probability bounds $P_\al(X;x)$ are valid. 
\begin{description}
	\item[Model-independence: ] $P_\al(X;x)$ depends on the r.v.\ $X$ only through the distribution of $X$. 	 
	\item[Monotonicity in $X$: ] $P_\al(\cdot\,;x)$ is nondecreasing with respect to the stochastic dominance of order $\al+1$: for any r.v.\ $Y$ such that $X\alle Y$, one has $P_\al(X;x)\le P_\al(Y;x)$. Therefore, $P_\al(\cdot\,;x)$ is nondecreasing with respect to the stochastic dominance of any order $\ga\in[1,\al+1]$; in particular,   
for any r.v.\ $Y$ such that $X\le Y$, one has $P_\al(X;x)\le P_\al(Y;x)$. 
	\item[Monotonicity in $\al$: ] $P_\al(X;x)$ is nondecreasing in $\al\in[0,\infty]$. 
	\item[Monotonicity in $x$: ] $P_\al(X;x)$ is nonincreasing in $x\in\R$. 	
	\item[Values: ] $P_\al(X;x)$ takes only values in the interval $[0,1]$. 
	\item[$\al$-concavity in $x$: ] $P_\al(X;x)^{-1/\al}$ is convex in $x$ if $\al\in(0,\infty)$, and $\ln P_\al(X;x)$ is concave in $x$ if $\al=\infty$. 
	\item[Stability in $x$: ] $P_\al(X;x)$ is continuous in $x$ at any point $x\in\R$ -- except the point $x=x_*$ when $p_*>0$. 	
	\item[Stability in $\al$: ] Suppose that 
	a sequence $(\al_n)$ is as in 
	Proposition~\ref{prop:P cont in al}. 
	Then $P_{\al_n}(X;x)\to P_\al(X;x)$. 	 	
	\item[Stability in $X$: ] Suppose that $\al\in(0,\infty]$ and a sequence $(X_n)$ is as in Proposition~\ref{prop:P cont in X}. Then $P_\al(X_n;x)\to P_\al(X;x)$. 
	\item[Translation invariance: ] $P_\al(X+c;x+c)=P_\al(X;x)$ for all real $c$. 
	\item[Consistency: ] $P_\al(c;x)=P_0(c;x)=\ii{c\ge x}$ for all real $c$; that is, if the r.v.\ $X$ is the constant $c$, 
	then all the tail probability bounds $P_\al(X;x)$ precisely equal the true tail probability $\P(X\ge x)$. 
	\item[Positive homogeneity: ] $P_\al(\ka X;\ka x)=P_\al(X;x)$ for all real $\ka>0$.  
\end{description}
\end{theorem}

A property similar to the model-independence was called ``neutrality'' in \cite[page~97]{yaari87}. 


\section{An optimal three-way stable and three-way monotonic spectrum of upper bounds on  quantiles}\label{quantiles} 
Take any 
\begin{equation}\label{eq:p}
p\in(0,1)	
\end{equation}
and introduce the generalized inverse (with respect to $x$) of the bound $P_\al(X;x)$ by the formula 
\begin{equation}\label{eq:Q}
	Q_\al(X;p):=\inf E_{\al,X}(p)=\inf\big\{x\in\R\colon P_\al(X;x)<p\big\};  
\end{equation}
where $E_{\al,X}(p)$ is as in \eqref{eq:J_al}. 
In particular, in view of the equality in \eqref{eq:0<al}, 
\begin{equation}\label{eq:Q_0}
	Q(X;p):=Q_0(X;p)=\inf\big\{x\in\R\colon\P(X\ge x)<p\big\}=\inf\big\{x\in\R\colon\P(X>x)<p\big\}, 
\end{equation}
which is a $(1-p)$-quantile of (the distribution of) the r.v.\ $X$; actually, $Q(X;p)$ is the largest one in the set of all $(1-p)$-quantiles of $X$. 

It follows immediately from \eqref{eq:Q}, 
\eqref{eq:mono in al}, and \eqref{eq:Q_0} that 
\begin{equation}\label{eq:Q incr}
\begin{aligned}
	&\text{$Q_\al(X;p)$ is an upper bound on the quantile $Q(X;p)$, and} \\
	&\text{$Q_\al(X;p)$ is 
	nondecreasing in $\al\in[0,\infty]$.}
\end{aligned}	
\end{equation}
Thus, one has a monotonic spectrum of upper bounds, $Q_\al(X;p)$, on the quantile $Q(X;p)$, ranging from the tightest bound, $Q_0(X;p)=Q(X;p)$, to the loosest one, $Q_\infty(X;p)$, which latter is based on the exponential bound $P_\infty(X;x)=\inf_{\la>0}\E e^{\la(X-x)}$ on $\P(X\ge x)$. 
Also, it is obvious from \eqref{eq:Q} 
that 
\begin{equation}\label{eq:Q decr in p}
\text{$Q_\al(X;p)$ is nonincreasing in $p\in(0,1)$. }	
\end{equation}



\begin{proposition}\label{prop:inverse} \ 
Recall the definitions of $x_*$ and $x_\al$ in \eqref{eq:x_*,p_*} and \eqref{eq:x_al}. The following statements are true. 
\begin{enumerate}[(i)]
	\item\label{it:Q in R} $Q_\al(X;p)\in\R$.  
	\item\label{it:p le p_*} If $p\in(0,p_*]\cap(0,1)$ then $Q_\al(X;p)=x_*$.  
	\item\label{it:Q le x_*} $Q_\al(X;p)\le x_*$.  
	\item\label{it:Q to x_*} $Q_\al(X;p)\underset{p\downarrow0}\longrightarrow x_*$.  
	\item\label{it:p>p_*} If $\al\in(0,\infty]$, then the function 
\begin{equation}\label{eq:inv}
(p_*,1)\ni p\mapsto Q_\al(X;p)\in(x_\al,x_*)	
\end{equation}
is the unique inverse to the continuous strictly decreasing function 
\begin{equation}\label{eq:dir}
(x_\al,x_*)\ni x\mapsto P_\al(X;x)\in(p_*,1). 
\end{equation}
So, the function \eqref{eq:inv} too is continuous and strictly decreasing. 
	\item\label{it:y} If $\al\in(0,\infty]$, then for any $y\in\big(-\infty,Q_\al(X;p)\big)$ one has 
$P_\al(X;y)>p$. 
	\item\label{it:>EX} If $\al\in[1,\infty]$, then $Q_\al(X;p)>\E X$. 	
\end{enumerate}
\end{proposition}

\begin{ex}\label{ex:Q,X_ab}\ 

\vspace*{6pt}
\noindent
\begin{parbox}{.38\textwidth}
	{ \includegraphics[width=.36\textwidth]{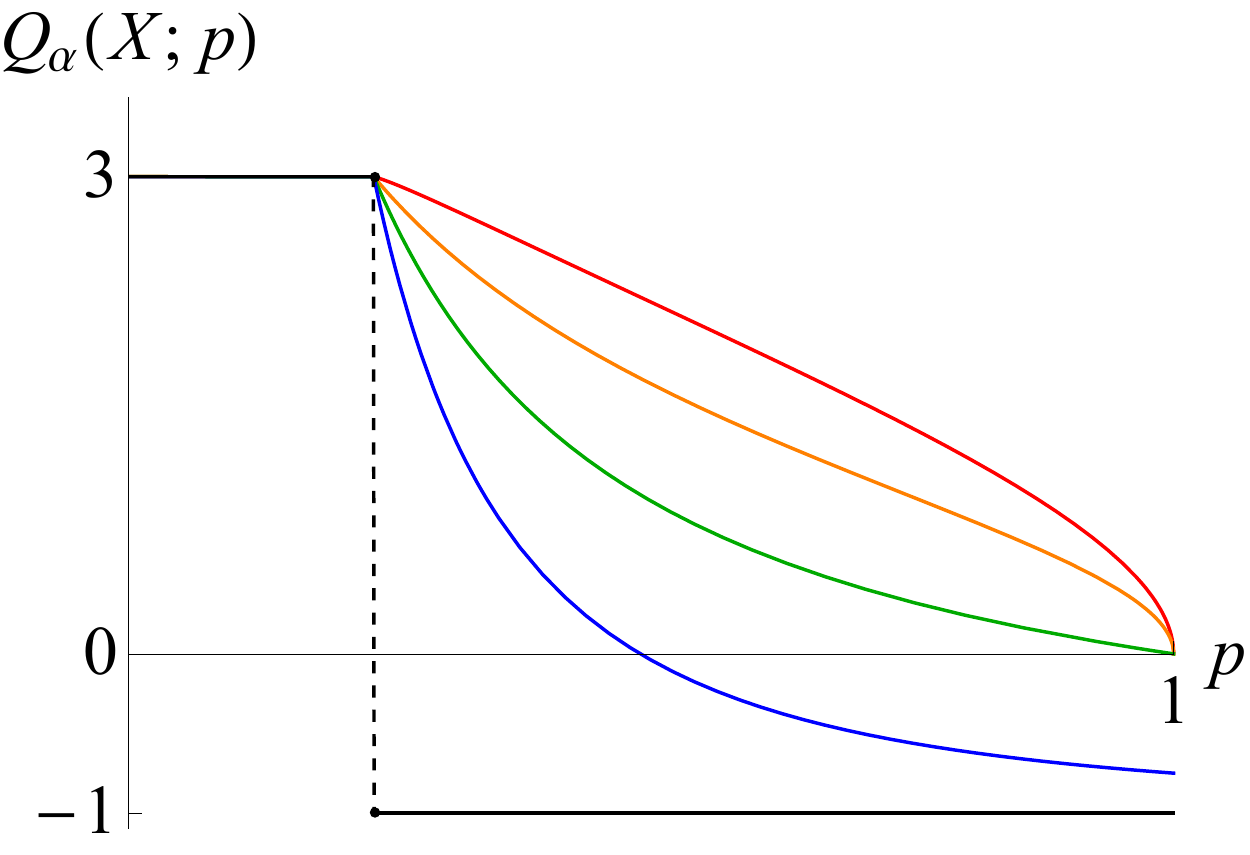} }
\end{parbox}
\begin{parbox}{.6\textwidth}
{Some parts of Proposition~\ref{prop:inverse} are illustrated in the picture here on the left, with graphs $\{\big(p,Q_\al(X;p)\big)\colon 0<p<1\}$ for a r.v.\ $X=X_{a,b}$ as in Example~\ref{ex:P,X_ab}, with the same $a=1$ and $b=3$, and the same values of $\al$, equal $0$ (black), $\frac12$ (blue), $1$ (green), $2$ (orange), and $\infty$ (red). 
One may compare this picture with the one in Example~\ref{ex:P,X_ab}, having in mind that the function $Q_\al(X;\cdot)$ is a generalized inverse to the function $P_\al(X;\cdot)$. 
}
\end{parbox}
\end{ex}


\vspace*{6pt}

The definition \eqref{eq:Q} of $Q_\al(X;p)$ is rather complicated, in view of the definition \eqref{eq:P new} of $P_\al(X;x)$. 
So, the following theorem will be useful, as it provides a more direct expression of $Q_\al(X;p)$; at that, one may again recall \eqref{eq:Q_0}, concerning the case $\al=0$. 

\begin{theorem}\label{prop:Q=}
For all $\al\in(0,\infty]$
\begin{equation}\label{eq:Q=inf}
Q_\al(X;p)=\inf_{t\in T_\al}B_\al(X;p)(t),
\end{equation} 
where $T_\al$ is as in \eqref{eq:T_al} and 
\begin{equation}\label{eq:B}
	B_\al(X;p)(t):=
\left\{	
\begin{alignedat}{2}
	&t+\frac{\|(X-t)_+\|_\al}{p^{1/\al}}&&\text{\ \ for $\al\in(0,\infty)$, }\\
	&t\,\ln\frac{\E e^{X/t}}p &&\text{\ \ for $\al=\infty$.}
\end{alignedat}	
\right.
\end{equation} 
\end{theorem}


\begin{proof}[Proof of Theorem~\ref{prop:Q=}]
The proof is based on the simple observation, following immediately from the definitions \eqref{eq:A old} and \eqref{eq:B}, that the dual level sets for the functions $\tA_\al(X;x)$ and $B_\al(X;p)$ are the same: 
\begin{equation}\label{eq:T=T}
	T_{\tA_\al(X;x)}(p)=T_{B_\al(X;p)}(x)  
\end{equation}
for all $\al\in(0,\infty]$, $x\in\R$, and $p\in(0,1)$, where 
\begin{equation*}
	T_{\tA_\al(X;x)}(p):=\{t\in T_\al\colon\tA_\al(X;x)(t)<p\}\quad\text{and}\quad
	T_{B_\al(X;p)}(x):=\{t\in T_\al\colon B_\al(X;p)(t)<x\}. 
\end{equation*}
Indeed, by \eqref{eq:P old} and \eqref{eq:T=T},  
\begin{equation*}
\begin{aligned}
	P_\al(X;x)<p &\iff \inf_{t\in T_\al}\tA_\al(X;x)(t)<p \\ 
	&\iff T_{\tA_\al(X;x)}(p)\ne\emptyset	\iff T_{B_\al(X;p)}(x)\ne\emptyset	
&\iff x>\inf_{t\in T_\al}\,B_\al(X;p)(t).   	
\end{aligned}	
\end{equation*}
Now \eqref{eq:Q=inf} follows immediately by \eqref{eq:Q}. 
\end{proof}

Note that the case $\al=\infty$ of Theorem~\ref{prop:Q=} is a special case of \cite[Proposition 1.5]{pin-yt}, and the above proof of Theorem~\ref{prop:Q=} is similar to that of \cite[Proposition 1.5]{pin-yt}. 
Correspondingly, the duality presented in the above proof of Theorem~\ref{prop:Q=} is a generalization of the bilinear Legendre--Fenchel duality considered in \cite{pin-yt}. 
	
\begin{theorem}\label{th:coher}
The following properties of the quantile bounds $Q_\al(X;p)$ are valid. 
\begin{description}
	\item[Model-independence: ] $Q_\al(X;p)$ depends on the r.v.\ $X$ only through the distribution of $X$. 	
	\item[Monotonicity in $X$: ] $Q_\al(\cdot\,;p)$ is nondecreasing with respect to the stochastic dominance of order $\al+1$: for any r.v.\ $Y$ such that $X\alle Y$, one has $Q_\al(X;p)\le Q_\al(Y;p)$. Therefore, $Q_\al(\cdot\,;p)$ is nondecreasing with respect to the stochastic dominance of any order $\ga\in[1,\al+1]$; in particular,   
for any r.v.\ $Y$ such that $X\le Y$, one has $Q_\al(X;p)\le Q_\al(Y;p)$. 	 
	\item[Monotonicity in $\al$: ] $Q_\al(X;p)$ is nondecreasing in $\al\in[0,\infty]$. 
	\item[Monotonicity in $p$: ] $Q_\al(X;p)$ is nonincreasing in $p\in(0,1)$, and $Q_\al(X;p)$ is strictly decreasing in $p\in[p_*,1)\cap(0,1)$ if $\al\in(0,\infty]$. 	
	\item[Finiteness: ] $Q_\al(X;p)$ takes only (finite) real values. 
	\item[Concavity in $p^{-1/\al}$ or in $\ln\frac1p$: ] $Q_\al(X;p)$ is concave in $p^{-1/\al}$ if $\al\in(0,\infty)$, and $Q_\infty(X;p)$ is concave in $\ln\frac1p$. 
	\item[Stability in $p$: ] $Q_\al(X;p)$ is continuous in $p\in(0,1)$ if $\al\in(0,\infty]$. 	
	\item[Stability in $X$: ] Suppose that $\al\in(0,\infty]$ and a sequence $(X_n)$ is as in Proposition~\ref{prop:P cont in X}. Then $Q_\al(X_n;p)\to Q_\al(X;p)$. 
	\item[Stability in $\al$: ] Suppose that $\al\in(0,\infty]$ and a sequence $(\al_n)$ is as in 
	Proposition~\ref{prop:P cont in al}. 
	 Then $Q_{\al_n}(X;p)\to Q_\al(X;p)$. 	 	
	\item[Translation invariance: ] $Q_\al(X+c;p)=Q_\al(X;p)+c$ for all real $c$. 
	\item[Consistency: ] $Q_\al(c;p)=c$ for all real $c$; that is, if the r.v.\ $X$ is the constant $c$, 
	then all the quantile bounds $Q_\al(X;p)$ equal $c$. 
	\item[Sensitivity: ] Suppose here that $X\ge0$. If at that $\P(X>0)>0$, then $Q_\al(X;p)>0$ for all $\al\in(0,\infty]$; if, moreover, $\P(X>0)>p$, then $Q_0(X;p)>0$. 
	\item[Positive homogeneity: ] $Q_\al(\ka X;p)=\ka Q_\al(X;p)$ for all real $\ka\ge0$.  
	\item[Subadditivity: ] $Q_\al(X;p)$ is subadditive in $X$ if $\al\in[1,\infty]$; that is, for any other r.v.\ $Y$ (defined on the same probability space as $X$) one has 
	$$Q_\al(X+Y;p)\le Q_\al(X;p)+Q_\al(Y;p).$$ 
		\item[Convexity: ] $Q_\al(X;p)$ is convex in $X$ if $\al\in[1,\infty]$; that is, for any other r.v.\ $Y$ (defined on the same probability space as $X$) and any $t\in(0,1)$ one has $$Q_\al\big((1-t)X+tY;p\big)\le(1-t)Q_\al(X;p)+tQ_\al(Y;p).$$ 
\end{description}
\end{theorem}

The inequality $Q_1(X;p)\le Q_\infty(X;p)$, in other notations, was mentioned (without proof) in \cite{rio06.17.13}; of course, this inequality is a particular, and important, case of the monotonicity of $Q_\al(X;p)$ in $\al\in[0,\infty]$. 
That $Q_\al(\cdot\,;p)$ is nondecreasing with respect to the stochastic dominance of order $\al+1$ was shown (using other notations) in  
\cite{degiorgi} in the case $\al=1$. 

The following strict 
monotonicity property 
complements the monotonicity property 
of $Q_\al(X;p)$ in $X$ stated in Theorem~\ref{th:coher}.  

\begin{proposition}\label{cor:str mono X} 
Suppose that a r.v.\ $Y$ is stochastically strictly greater than $X$ \big(which may be written as $X\sst Y$; cf.\ 
\eqref{eq:st}\big) in the sense that $X\st Y$ and for any $v\in\R$ there is some $u\in(v,\infty)$ such that $\P(X\ge u)<\P(Y\ge u)$. Then 
$Q_\al(X;p)<Q_\al(Y;p)$ if $\al\in(0,\infty]$. 
\end{proposition}

This proposition will 
be useful in the proof of Proposition~\ref{prop:al<1} below. 

Given the positive homogeneity, it is clear that the subadditivity and convexity properties of $Q_\al(X;p)$ easily follow from each other. 
In the statements in Theorem~\ref{th:coher} on these two mutually equivalent properties, it was assumed that $\al\in[1,\infty]$. One may ask whether this restriction is essential. The answer to this question is ``yes'': 

\begin{proposition}\label{prop:al<1}
There are r.v.'s $X$ and $Y$ such that 
for all $\al\in
[0,1)$ and all $p\in(0,1)$ 
one has $Q_\al(X+Y;p)>Q_\al(X;p)+Q_\al(Y;p)$,  
so that $Q_\al(X;p)$ is not subadditive (and, equivalently, not convex) in $X$. 
\end{proposition}

It is well known (see e.g.\ \cite{artzner.etal.99,pflug00,rocka-ur02}) that $Q(X;p)=Q_0(X;p)$ is not subadditive in $X$; it could therefore have been expected that 
$Q_\al(X;p)$ will not be subadditive in $X$ if $\al$ is close enough to $0$. 
In a quite strong and specific sense, Proposition~\ref{prop:al<1} justifies such expectations.

\begin{center}***\end{center}

In Figure~\ref{fig:}, 
the graphs of the quantile bounds $Q_\al(X;p)$ as functions of $\al\in(0,20]$ are given for (a) $p=0.05$ (left panel) and (b) $p=0.01$ (right panel), for the case when $X$ has the Gamma distribution with the scale parameter equal $1$ and values $0.5$, $1$, $2$, and $5$ of the shape parameter (say $a$) -- shown respectively in colors red, green, blue, and black. 
\begin{figure}[ht]
\centering
\subfigure[
$p=0.05$]{
    \includegraphics
    [width=2.5in]{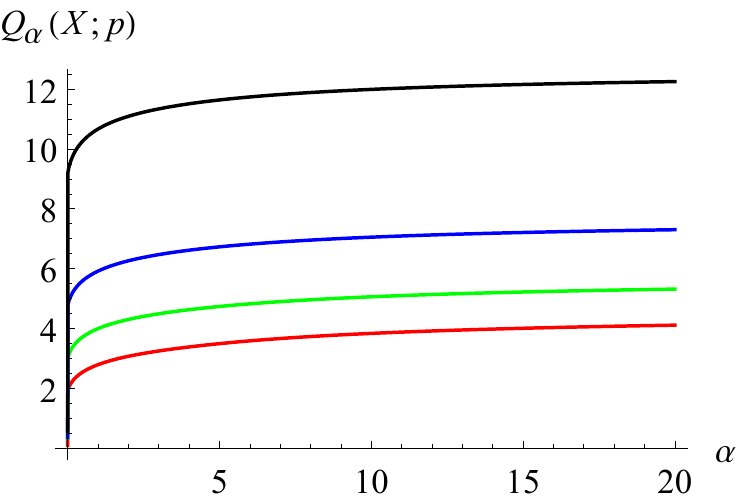}
    \label{fig:subfig1}
\rule[-10pt]{0pt}{0pt}    
}
\hspace*{1cm}
\subfigure[
$p=0.01$]{
    \includegraphics
    [width=2.5in]{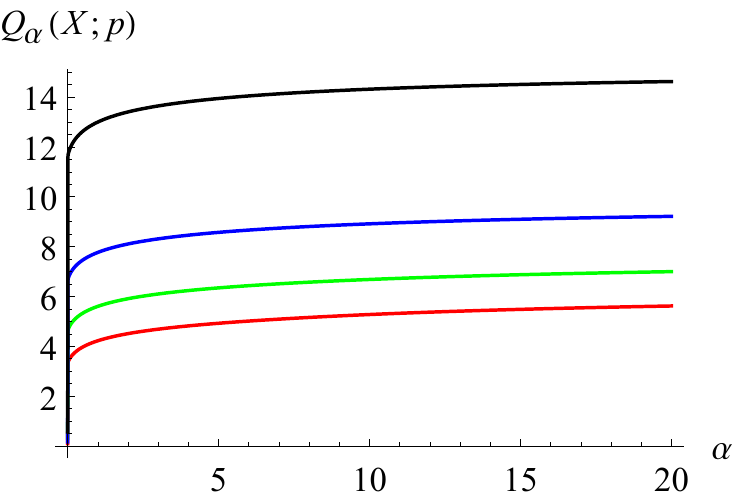}
    \label{fig:subfig2}
\rule[-10pt]{0pt}{0pt}     
}
\caption[Optional caption for list of figures]{
Graphs of the quantile bounds $Q_\al(X;p)$ as functions of $\al$. 
}
\label{fig:}
\end{figure}

These graphs illustrate the general 
monotonicity properties of $Q_\al(X;p)$ in $\al$, $X$, and $p$ stated in Theorem~\ref{th:coher}; recall here that the Gamma distribution is (i) stochastically increasing with respect to the shape parameter $a$ and (ii) close to normality for large values of $a$. 
It is also seen that $Q_\al(X;p)$ varies rather little in $\al$, so that the quantile bounds $Q_\al(X;p)$ are not too far from the corresponding true quantiles $Q(X;p)=Q_0(X;p)$; cf.\ the somewhat similar observation made in \cite[Theorem~2.8.]{T2}. 

In fact, it can be shown that for small values of $p$ the quantile bounds $Q_\al(X;p)$ are relatively close to the true quantiles $Q_0(X;p)$ whenever the right tail of the distribution of $X$ is light enough (depending on $\al$) and regular enough. 
One possible formalization of this general thesis is provided by Proposition~\ref{prop:Q close} below, which is based on \cite[Theorem~4.2 and Remark 4.3]{pin98}. We shall need pertinent definitions from that paper. 

Take any $r\in(0,\infty]$. 
Given a positive function $q$ on $\R$, let us say that $q(x)$ is {\em like} $x^{-r}$ if there is a positive twice differentiable function $q_0$ on $\R$ such that 
\begin{equation}\label{like-x-r-eq}
	q(x)\underset{x\to\infty}\sim q_0(x)\quad\text{and}\quad
	\lim_{x\to\infty} {q_0(x)q_0''(x)\over{q_0'}(x)^2}=1+{1\over r};  
\end{equation}
as usual, we write $f\sim g$ if $f/g\to1$.
For any real $r\ne0$, the second relation in \eqref{like-x-r-eq} can be rewritten as 
\begin{equation*}
	\lim_{x\to\infty} \Big(\frac1{(\ln q_0)'(x)}\Big)'=-\frac1r, \tag{\ref{like-x-r-eq}a}
\end{equation*}
which successively implies $\tfrac1{(\ln q_0)'(x)}\sim-\tfrac xr$, $(\ln q_0)'(x)\sim-\tfrac rx$, $\ln q_0(x)\sim-r\ln x$, and hence 
\begin{equation*}
q(x)=x^{-r+o(1)} \quad\text{as}\quad x\to\infty. 
\tag{\ref{like-x-r-eq}b}	
\end{equation*}
In particular, given any $r\in(0,\infty)$, $s\in\R$, and $C\in(0,\infty)$, 
\begin{equation*}
\text{if 
$q(x)\underset{x\to\infty}\sim Cx^{-r}\ln^s x$,  
then $q(x)$ is like $x^{-r}$.} 	
\end{equation*}
Also, given any $C$ and $C_1$ in $(0,\infty)$, 
$\ga\in(0,\infty)$, and $s\in(1,\infty)$, 
\begin{equation*}
\text{if\ \  
$q(x)\underset{x\to\infty}\sim C\exp\{-C_1 x^\ga\}$\ \ or 
$q(x)\underset{x\to\infty}\sim C\exp\{-C_1 \ln^s x\}$,\ \ 
then $q(x)$ is like $x^{-\infty}$.}
\end{equation*}
Moreover, if $q(x)$ is like $x^{-\infty}$, then for all $r\in(0,\infty)$ one has $q(x)=o(x^{-r})$ as $x\to\infty$.



\begin{proposition}\label{prop:Q close} 
\ 

\begin{enumerate}[(i)]
	\item If the r.v.\ $X$ is bounded from above -- that is, $x_*<\infty$, then 
\begin{equation}\label{eq:x_*<infty}
	\lim_{p\downarrow0}Q_\al(X;p)=\lim_{p\downarrow0}Q_0(X;p)=x_*\in\R.  
\end{equation}
	\item If $\P(X\ge x)$ is like  $x^{-r}$ for some $r\in(\al,\infty]$, then $x_*=\infty$, $Q_\al(X;p)\underset{p\downarrow0}\longrightarrow\infty$, and 
\begin{equation}\label{eq:Q close}
	Q_\al(X;p)\underset{p\downarrow0}\sim K(r,\al)\,Q_0(X;p), 
\end{equation} 
where
\begin{equation}\label{eq:K}
	K(r,\al):=
	\left\{
	\begin{alignedat}{2}
	&1&&\text{\; if\; $r=\infty$}, \\
	&c_{r,\al}^{1/r}&&\text{\; if\; $r<\infty$,}
	\end{alignedat}
	\right.
\end{equation}  
\begin{equation*}
	c_{r,\al}:=\frac{\Ga(\al+1)\Ga (r-\al)}{\Ga (r)} \frac{r^r}{\al^\al  (r-\al)^{r-\al} }, 
\end{equation*}
and $\Ga$ is the Gamma function, given by the formula $\Ga(\be)=\int_0^\infty u^{\be-1}e^{-u}\dd u$ for $\be>0$.   
\end{enumerate}
\end{proposition}

\vspace*{6pt}
\noindent
\begin{parbox}{.38\textwidth}
	{ \includegraphics[width=.36\textwidth]{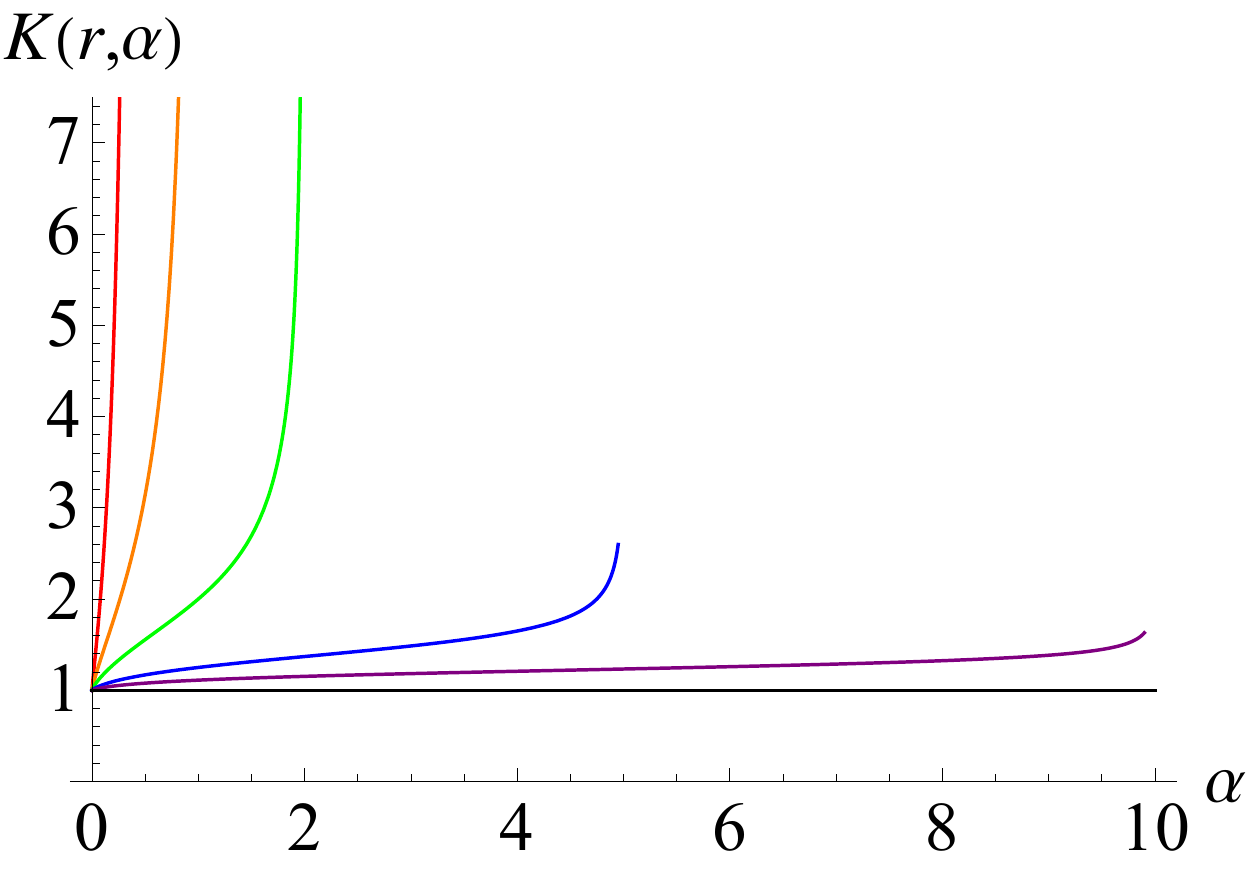} }
\end{parbox}
\begin{parbox}{.6\textwidth}
{
Graphs $\big\{\big(\al,K(r,\al)\big)\colon0<\al\le\min(10,0.99r),\,K(r,\al)\le7.5\big\}$ are shown here on the left: for $r=\frac12$ (red), $r=1$ (orange), $r=2$ (green), $r=5$ (blue), $r=10$ (purple), and $r=\infty$ (black). 
It is seen that, if the (right) tail of the distribution of $X$ is very heavy -- that is, if $r$ is comparatively small, then even for $\al=1$ the quantile bound $Q_\al(X;p)$ is much greater (for small $p$) than the true quantile $Q_0(X;p)$. However, if the tail is not very heavy  -- that is, if $r$ is not so small, then the graph of $K(r,\al)$ is very flat -- that is, $Q_\al(X;p)$ varies little with $\al$ (if $p$ is small). 
}
\end{parbox}

\vspace*{6pt}
Concerning the relevance of the condition $p\downarrow0$ in the asymptotic relations \eqref{eq:x_*<infty} and \eqref{eq:Q close} in Proposition~\ref{prop:Q close}, note that small values of $p$ are of particular importance in financial practice. Indeed, values of $p$ commonly used with the so-called value-at-risk (VaR) measure -- equal to the quantile $Q_0(X;p)$, where $X$ is the amount of financial loss --    
are 1\% and 5\% for one-day and two-week horizons, respectively \cite{pearson02}.  

As for the condition that $X$ be bounded from above in part (i) of Proposition~\ref{prop:Q close}, it is obviously fulfilled, in particular, whenever the r.v.\ $X$ takes only finitely many values, as is assumed e.g.\ in \cite{artzner.etal.99}. 

By (\ref{like-x-r-eq}b), if $\P(X\ge x)$ is like  $x^{-r}$ for some r.v.\ $X$ and some real $r\ne0$, 
%
then necessarily $r>0$ -- which is in accordance with the assumption $r\in(0,\infty]$ made above concerning \eqref{like-x-r-eq}. 

Usual statistical families of continuous distributions, including the normal, log-normal, gamma, beta, Student, and Pareto families, 
are covered by 
Proposition~\ref{prop:Q close}. 
More specifically, part (i) of Proposition~\ref{prop:Q close} applies to the beta family of distributions, including the uniform distribution. Part (ii) of Proposition~\ref{prop:Q close} applies (with $r=\infty$) to the normal, log-normal, and gamma families, including the exponential family; the Student family (with $r=d$, where $d$ is the number of degrees of freedom); and Pareto family \big(with $r=\be$, assuming that $\P(X\ge x)=(a/x)^\be$ for some $\be$ and $a$ in $(0,\infty)$ and all real $x\ge a$\big). 
However, note that under the condition that $\P(X\ge x)$ is like  $x^{-r}$, as in part~(ii) of Proposition~\ref{prop:Q close}, \eqref{eq:Q close} is guaranteed to hold only for $\al\in[0,r)$. 
Also, in the case when $\P(X\ge x)$ is like  $x^{-r}$ for some $r\in(0,\infty)$, 
%
\eqref{eq:Q close} cannot possibly hold for any $\al\in(r,\infty]$ -- because then, by part (ii) of Proposition~\ref{prop:P}, \eqref{eq:mono in al}, and \eqref{eq:Q}, $Q_\al(X;p)=\infty$. 
So, the general tendency is that, the lighter the right tail of the distribution of $X$, the wider is the range of values of $\al$ for which \eqref{eq:Q close} holds. 
Moreover, it appears that, the lighter the right tail of the distribution of $X$, the closer is the constant $K(r,\al)$ in \eqref{eq:Q close} to $1$ and, more generally, the closer is the quantile bound $Q_\al(X;p)$ to the true quantile $Q_0(X;p)$. 

One can also show, using \cite[Remark~3.13]{pin98} or \cite[Remark~1.4]{binom}, that \eqref{eq:Q close} will hold -- with $r=\infty$ and hence $K(r,\al)=1$ -- for all $\al\in[0,\infty]$ and usual statistical families of discrete distributions, including the Poisson, geometric, and, more generally negative binomial families -- because the right tails of those distributions are light enough. 
On the other hand, \eqref{eq:x_*<infty} will hold for any distributions with a bounded support, including the binomial and hypergeometric distributions. 

More examples of distributions to which Proposition~\ref{prop:Q close} is applicable can be found in \cite{pin99}. 

\begin{center}***\end{center}

\section{Computation of the tail probability and quantile bounds} \label{comput}

\subsection{Computation of $P_\al(X;x)$}  \label{P comput}
The computation of $P_\al(X;x)$ in the case $\al=0$ is straightforward, in view of the equality in \eqref{eq:0<al}. 
If $x\in[x_*,\infty)$, then the value of $P_\al(X;x)$ is easily found by part \eqref{it:x ge x_*} of Proposition~\ref{prop:P,+}. So, in the rest of this subsection it may be assumed that  $\al\in(0,\infty]$ and $x\in(-\infty,x_*)$.  

In the case when $\al\in(0,\infty)$, using \eqref{eq:P new}, 
the inequality 
\begin{equation}\label{eq:domin}
	\big(1+\la(X-x)/\al\big)_+^\al
	\le2^{(\al-1)_+}\big(\la^\al X_+^\al+(\al-\la x)_+^\al\big)/\al^\al, 
\end{equation}
the condition $X\in\XX_\al$, 
and dominated convergence, one sees that $A_\al(X;x)(\la)$ is continuous in $\la\in(0,\infty)$ and right-continuous in $\la$ at $\la=0$ (assuming the definition \eqref{eq:A new} for $\la=0$ as well), and hence  
\begin{equation}\label{eq:P,tA}
P_\al(X;x)=\inf_{\la\in[0,\infty)}A_\al(X;x)(\la).    
\end{equation} 
Similarly, using in place of \eqref{eq:domin} the inequality $e^{\la X}\le1+e^{\la_0 X}$ whenever $0\le\la\le\la_0$, one can show that $A_\infty(X;x)(\la)$ is continuous in $\la\in\La_X$ \big(recall \eqref{eq:La_X}\big) and right-continuous in $\la$ at $\la=0$, so that \eqref{eq:P,tA} holds for $\al=\infty$ as well -- provided that $X\in\XX_\infty$. Moreover, by the Fatou lemma for the convergence in distribution \cite[Theorem~5.3]{billingsley}, $A_\infty(X;x)(\la)$ is lower-semicontinuous in $\la$ at $\la=\la_*:=\sup\La_X$ even if $\la_*\in\R\setminus\La_X$. It then follows by the convexity of $A_\infty(X;x)(\la)$ in $\la$ that $A_\infty(X;x)(\la)$ is left-continuous in $\la$ at $\la=\la_*$ whenever $\la_*\in\R$; at that, the natural topology on the set $[0,\infty]$ is used, as it is of course possible that $A_\infty(X;x)(\la_*)=\infty$.

Since $x\in(-\infty,x_*)$, one can find some $y\in(x,\infty)$ such that $\P(X\ge y)>0$ (of course, necessarily $y\in(x,x_*]$); so, one can introduce 
\begin{equation}\label{eq:lamax}
\la_{\max}:=\la_{\max,\al}:=\la_{\max,\al,X}:=
\left\{
\begin{alignedat}{2}
&\frac\al{y-x}\,\Big(\frac1{\P(X\ge y)^{1/\al}}-1\Big) && \text{ if } \al\in(0,\infty), \\ 
&\frac1{y-x}\,\ln\frac1{\P(X\ge y)} && \text{ if } \al=\infty.      
\end{alignedat}
\right. 
\end{equation}
Then, by \eqref{eq:A new}, 
$A_\al(X;x)(\la)\ge\E\big(1+\la(X-x)/\al\big)_+^\al\ii{X\ge y}
\ge\big(1+\la(y-x)/\al\big)^\al\P(X\ge y)>1$ if $\al\in(0,\infty)$ and $\la\in(\la_{\max,\al},\infty)$, and $A_\infty(X;x)(\la)\ge\E e^{\la(X-x)}\ii{X\ge y}
	\ge e^{\la(y-x)}\P(X\ge y)>1$ if $\la\in(\la_{\max,\infty},\infty)$.
So, for all $\al\in(0,\infty]$ one has $A_\al(X;x)(\la)>1\ge P_\al(X;x)=\inf_{\la\in(0,\infty)}A_\al(X;x)(\la)$ provided that $\la\in(\la_{\max,\infty},\infty)$, and hence 
\begin{equation}\label{eq:P,tA,lamax}
P_\al(X;x)=\inf_{\la\in[0,\la_{\max,\al}]}A_\al(X;x)(\la),\quad\text{if}\quad
\text{$\al\in(0,\infty]$ and $x\in(-\infty,x_*)$}. 
\end{equation} 
Therefore and because $\la_{\max,\al}<\infty$, the minimization of $A_\al(X;x)(\la)$ in $\la$ in \eqref{eq:P,tA,lamax} in order to compute the value of $P_\al(X;x)$ can be done effectively if $\al\in[1,\infty]$, because in this case $A_\al(X;x)(\la)$ is convex in $\la$. 
At that, the positive-part moments $\E\big(1+\la(X-x)/\al)_+^\al$, which express $A_\al(X;x)(\la)$ for $\al\in(0,\infty)$ in accordance with \eqref{eq:A new}, 
can be efficiently computed using formulas in \cite{positive}; cf.\ e.g.\ \cite[Section~3.2.3]{pin-hoeff}. 
Of course, for specific kinds of distributions of the r.v.\ $X$, 
more explicit expressions for the positive-part moments 
can be used. 

In the remaining case, when $\al\in(0,1)$, the function $\la\mapsto A_\al(X;x)(\la)$ cannot in general be ``convexified'' by any monotonic transformations in the domain and/or range of this function, and the set of minimizing values of $\la$ does not even have to be connected, in the following rather strong sense:  
 
\begin{proposition}\label{prop:two min,P} 
For any $\al\in(0,1)$, $p\in(0,1)$, and $x\in\R$, there is a r.v.\ $X$ (taking three distinct values) such that $P_\al(X;x)=p$ and the infimum $\inf_{\la\in(0,\infty)}$ in \eqref{eq:P new} is attained at precisely two distinct values of $\la\in(0,\infty)$.  
\end{proposition}

Proposition~\ref{prop:two min,P} is illustrated by 

\begin{ex}\label{ex:two min,P}\  

\vspace*{4pt} 
\noindent
\begin{parbox}{.38\textwidth}
	{ \includegraphics[width=.36\textwidth]{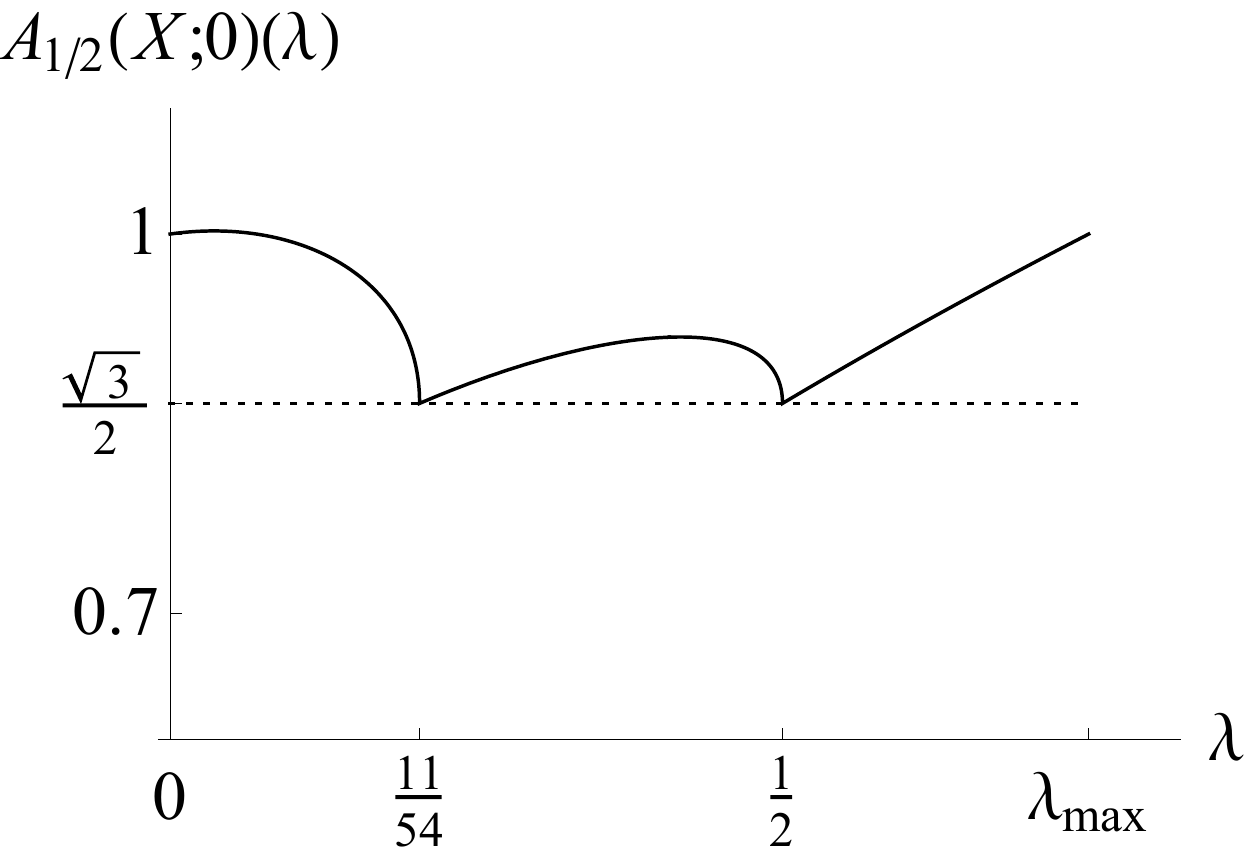} }
\end{parbox}
\begin{parbox}{.6\textwidth}
{
Let $X$ be a r.v.\ taking values $-\frac{27}{11}, -1, 2$ with probabilities $\frac14,\frac14,\frac12$; then $x_*=2$. Also let $\al=\frac12$ and $x=0$, so that $x\in(-\infty,x_*)$, and then let $\la_{\max}$ be as in \eqref{eq:lamax} with $y=x_*=2$, so that here $\la_{\max}=\frac34$. 
Then the minimum of $A_\al(X;0)(\la)$ over all real $\la\ge0$ equals $\frac{\sqrt3}2$ and is attained at each of the two points, $\la=\frac{11}{54}$ and $\la=\frac12$, and only at these two points. 
The graph $\big\{\big(\la,A_{1/2}(X;0)(\la)\big)\colon0\le\la\le\la_{\max}\big\}$ is shown here on the left.  
}
\end{parbox}
\end{ex}

Nonetheless, effective minimization of $A_\al(X;x)(\la)$ in $\la$ in \eqref{eq:P,tA,lamax} is possible even in the case $\al\in(0,1)$, say by the interval method. Indeed, take any $\al\in(0,1)$ and write 
\begin{equation*}
	A_\al(X;x)(\la)=A_\al^+(X;x)(\la)+A_\al^-(X;x)(\la), 
\end{equation*}
where (cf.\ \eqref{eq:A new}) 
\begin{equation*}
	A_\al^+(X;x)(\la):=\E\big(1+\la(X-x)/\al)_+^\al\ii{X\ge x}\quad\text{and}\quad
	A_\al^-(X;x)(\la):=\E\big(1+\la(X-x)/\al)_+^\al\ii{X<x}.  
\end{equation*}
Just as $A_\al(X;x)(\la)$ is continuous in $\la\in[0,\infty)$, so are $A_\al^+(X;x)(\la)$ and $A_\al^-(X;x)(\la)$. 
It is also clear that $A_\al^+(X;x)(\la)$ is nondecreasing and $A_\al^-(X;x)(\la)$ is nonincreasing in $\la\in[0,\infty)$. 

So, as soon as the minimizing values of $\la$ are bracketed as in \eqref{eq:P,tA,lamax}, one 
can partition the finite interval $[0,\la_{\max,\al}]$ into a large number of small subintervals $[a,b]$ with $0\le a<b\le\la_{\max,\al}$. For each such subinterval, 
\begin{align*}
	M_{a,b}:=\max_{\la\in[a,b]}A_\al(X;x)(\la)\le A_\al^+(X;x)(b)+A_\al^-(X;x)(a), \\ 
	m_{a,b}:=\min_{\la\in[a,b]}A_\al(X;x)(\la)\ge A_\al^+(X;x)(a)+A_\al^-(X;x)(b),  
\end{align*}
so that, by the continuity of $A_\al^\pm(X;x)(\la)$ in $\la$, 
\begin{equation*}
M_{a,b}-m_{a,b}\le A_\al^+(X;x)(b)-A_\al^+(X;x)(a)+A_\al^-(X;x)(a)-A_\al^-(X;x)(b)\longrightarrow0	
\end{equation*}
as $b-a\to0$, uniformly over all subintervals $[a,b]$ of the interval $[0,\la_{\max,\al}]$. 
Thus, one can effectively bracket the value $P_\al(X;x)=\inf_{\la\in[0,\la_{\max,\al}]}A_\al(X;x)(\la)$ with any degree of accuracy; this same approach will work, and perhaps may be sometimes useful, for $\al\in[1,\infty)$ as well.

\subsection{Computation of $Q_\al(X;p)$}\label{Q comput}

\begin{proposition}\label{prop:att} 
\textbf{\emph{(Quantile bounds: Attainment and bracketing).}} 
\begin{enumerate}[(i)]
	\item If $\al\in(0,\infty)$ then $\inf_{t\in T_\al}=\inf_{t\in\R}$ in \eqref{eq:Q=inf} is attained at some $t_\opt\in\R$ and hence 	
\begin{equation}\label{eq:att}
	Q_\al(X;p)=\min_{t\in\R}B_\al(X;p)(t)=B_\al(X;p)(t_\opt); 
\end{equation}
moreover, for any 
\begin{equation*}
	s\in\R\quad\text{and}\quad \tp\in(p,1), 
\end{equation*}
necessarily 
\begin{equation}\label{eq:bracket}
	t_\opt\in[t_{\min},t_{\max}],
\end{equation}
where 
\begin{equation}\label{eq:t_max}
t_{\max}:=B_\al(X;p)(s),	\quad
t_{\min}:=t_{0,\min}\wedge t_{1,\min},  
\end{equation}
\begin{equation}\label{eq:t_0,t_1}
	t_{0,\min}:=Q_0(X;\tp), \quad 
	t_{1,\min}:=\frac{(\tp/p)^{1/\al}\,t_{0,\min}-t_{\max}}{(\tp/p)^{1/\al}-1}. 
\end{equation}
	\item Suppose now that $\al=\infty$. Then $\inf_{t\in T_\al}=\inf_{t\in(0,\infty)}$ in \eqref{eq:Q=inf} is attained and hence 
\begin{equation*}
	Q_\infty(X;p)=\min_{t\in(0,\infty)}B_\infty(X;p)(t)  
\end{equation*}	
unless
\begin{equation}\label{eq:x_*,p_*conds}
	x_*<\infty \quad\text{and}\quad p\le p_*,     
\end{equation}
where $x_*$ and $p_*$ are as in \eqref{eq:x_*,p_*}. 
On the other hand, if conditions \eqref{eq:x_*,p_*conds} hold then $B_\infty(X;p)(t)$ is  strictly increasing in $t>0$ and hence  
$\inf_{t\in T_\al}=\inf_{t\in(0,\infty)}$ in \eqref{eq:Q=inf} is not attained; rather, 
\begin{equation*}
	Q_\infty(X;p)=\inf_{t>0}B_\infty(X;p)(t)=
	B_\infty(X;p)(0+)=x_*.  
\end{equation*}
\end{enumerate}
\end{proposition}

For instance, 
in the case when $\al=0.5$, $p=0.05$, and $X$ has the Gamma distribution with the shape and scale parameters equal to $2.5$ and $1$, respectively, Proposition~\ref{prop:att} yields  $t_{\min}>4.01$ (using $\tp=0.095$) and $t_{\max}<6.45$. 

When $\al=0$, the quantile bound $Q_\al(X;p)$ is simply the quantile $Q(X;p)$, which can be effectively computed by formula \eqref{eq:Q_0}, since the tail probability $\P(X>x)$ is monotone in $x$. 
Next, as was noted in the proof of Theorem~\ref{th:coher}, $B_\al(X;p)(t)$ is convex in $t$ when $\al\in[1,\infty]$, which provides for an effective computation of $Q_\al(X;p)$ by formula \eqref{eq:Q=inf}. 

Therefore, it remains to consider the computation -- again by formula \eqref{eq:Q=inf} -- of $Q_\al(X;p)$ 
for $\al\in(0,1)$. In such a case, as in Subsection~\ref{P comput}, one can use an interval method. 
As soon as the minimizing values of $t$ are bracketed as in \eqref{eq:bracket}, one 
can partition the finite interval $[t_{\min},t_{\max}]$ into a large number of small subintervals $[a,b]$ with $t_{\min}\le a<b\le t_{\max}$. For each such subinterval, 
\begin{align*}
	M_{a,b}:=\max_{t\in[a,b]}B_\al(X;p)(t)\le b+p^{-1/\al}\,\|(X-a)_+\|_\al, \\ 
	m_{a,b}:=\min_{t\in[a,b]}B_\al(X;p)(t)\ge a+p^{-1/\al}\,\|(X-b)_+\|_\al,  
\end{align*}
so that, by the continuity of $\|(X-t)_+\|_\al$ in $t$, 
\begin{equation*}
M_{a,b}-m_{a,b}\le b-a+p^{-1/\al}\,(\|(X-a)_+\|_\al-\|(X-b)_+\|_\al)\longrightarrow0	
\end{equation*}
as $b-a\to0$, uniformly over all subintervals $[a,b]$ of the interval $[t_{\min},t_{\max}]$. 
Thus, one can effectively bracket the value $Q_\al(X;p)=\inf_{t\in\R}B_\al(X;p)(t)$; this same approach will work, and perhaps may be useful, for $\al\in[1,\infty)$ as well.

%
%

\begin{center}***\end{center}

If $\al\in(1,\infty)$ then, by part~\eqref{B str conv} 
of Proposition~\ref{prop:B str conv} and part~(i) of Proposition~\ref{prop:att}, 
the set $\amin$ is a singleton one; that is, there is exactly one minimizer $t\in\R$ of  
$B_\al(X;p)(t)$.
If $\al=1$ then $B_\al(X;p)(t)=B_1(X;p)(t)$ is convex, but not strictly convex, in $t$, and the set $\amin$ of all minimizers of $B_\al(X;p)(t)$ in $t$ coincides with the set of all $(1-p)$-quantiles of $X$, as mentioned at the conclusion of the derivation of the identity \eqref{eq:al=1}. Thus, if $\al=1$, then the set $\amin$ 
may in general 
be, depending on $p$ and the distribution of $X$, a nonzero-length closed interval. 
Finally, if $\al\in(0,1)$ then, in general, the set $\amin$ 
does not have to be connected:  

\begin{proposition}\label{prop:two min,Q} 
For any $\al\in(0,1)$, $p\in(0,1)$, and $x\in\R$, there is a r.v.\ $X$ (taking three distinct values) such that $Q_\al(X;p)=x$ and the infimum $\inf_{t\in T_\al}=\inf_{t\in\R}$ in \eqref{eq:Q=inf} is attained at precisely two distinct values of $t$.  
\end{proposition}

Proposition~\ref{prop:two min,Q} follows immediately from Proposition~\ref{prop:two min,P}, by the duality \eqref{eq:T=T} and the change-of-variables identity  $A_\al(X;x)(\la)=\tA_\al(X;x)(x-\al/\la)$ for $\al\in(0,\infty)$, used to establish \eqref{eq:P old}--\eqref{eq:A old}. 
At that, $\la\in(0,\infty)$ is one of the two minimizers of $A_\al(X;x)(\la)$ in Proposition~\ref{prop:two min,P} if and only if $t:=x-\al/\la$ is one of the two minimizers of $B_\al(X;p)(t)$ in Proposition~\ref{prop:two min,Q}. 

Proposition~\ref{prop:two min,P} is illustrated by the following example, which is obtained from Example~\ref{ex:two min,P} by the same duality \eqref{eq:T=T}. 

\begin{ex}\label{ex:two min,Q}\  

\vspace*{4pt} 
\noindent
\begin{parbox}{.38\textwidth}
	{ \includegraphics[width=.36\textwidth]{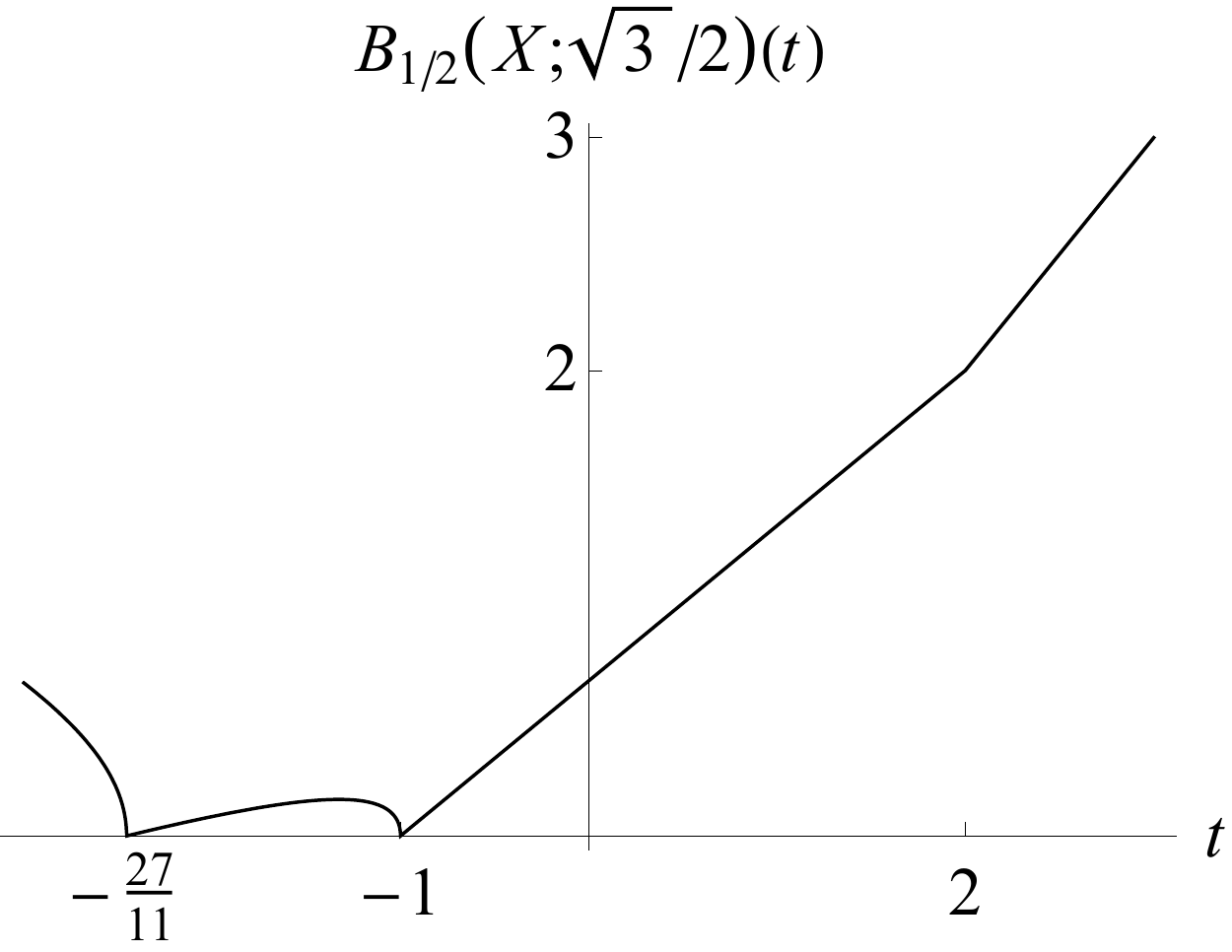} }
\end{parbox}
\begin{parbox}{.6\textwidth}
{
As in Example~\ref{ex:two min,P}, let $\al=\frac12$ and 
let $X$ be a r.v.\ 
taking values $-\frac{27}{11}, -1, 2$ with probabilities $\frac14,\frac14,\frac12$. Also let $p=\frac{\sqrt3}2$. 
Then the minimum of $B_\al(X;p)(t)$ over all real $t$ equals $0$ and is attained at each of the two points, $t=-\frac{27}{11}$ and $t=-1$, and only at these two points. 
The graph $\big\{\big(t,B_{1/2}\big(X;\frac{\sqrt3}2\big)(t)\big)\colon\break
-3\le t\le3\big\}$ is shown here on the left. 
The minimizing values of $t$ here, $-\frac{27}{11}$ and $-1$, are related with the minimizing values of $\la$ in Example~\ref{ex:two min,P}, $\frac{11}{54}$ and $\frac12$, by the mentioned formula $t=x-\al/\la$ (here with $x=0$ and $\al=\frac12$).  
}
\end{parbox}
\end{ex}

\begin{center}***\end{center}

In the case $\al=1$, an expression of $Q_\al(X;p)$ can be given 
in terms of the true $(1-p)$-quantile $Q(X;p)$: 
\begin{equation}\label{eq:al=1}
	Q_1(X;p)=Q(X;p)+\tfrac1p\,\E\big(X-Q(X;p)\big)_+. 
\end{equation}
That the expression for $Q_1(X;p)$ in \eqref{eq:Q=inf} coincides with the one in \eqref{eq:al=1} was proved in \cite[Theorem~1]{rocka-ur00} for absolutely continous r.v.'s $X$ and in \cite[page~273]{pflug00} and \cite[Theorem~10]{rocka-ur02} in general. 
For the readers' convenience, let us present here the following brief proof of \eqref{eq:al=1}. For all real $h>0$ and $t\in\R$ one has 
\begin{equation*}
	(X-t)_+ - (X-t-h)_+ = h \ii{X>t}-(t+h-X)\ii{t<X<t+h}.  
\end{equation*}
It follows that the right derivative of the convex function $t\mapsto t+\|(X-t)_+\|_1/p$ at any point $t\in\R$ is $1-\P(X>t)/p$, which, by \eqref{eq:Q_0}, is $\le0$ if $t<Q(X;p)$ and $>0$ if $t>Q(X;p)$. 
Hence, $Q(X;p)$ is a minimizer in  $t\in\R$ of $t+\|(X-t)_+\|_1/p$, and thus \eqref{eq:al=1} follows by \eqref{eq:Q=inf}. 
It is also seen now that any $(1-p)$-quantile of $X$ is a minimizer in  $t\in\R$ of $t+\|(X-t)_+\|_1/p$ as well, and $Q(X;p)$ is the largest of these minimizers. 

%

As was shown in \cite{rocka-ur02}, the expression for $Q_1(X;p)$ in \eqref{eq:al=1} can be rewritten as a conditional expectation: 
\begin{equation}\label{eq:Q_1=}
	Q_1(X;p)=Q(X;p)+\E\big(X-Q(X;p)\big|X\ge Q(X;p),\,U\ge\de\big)
	=\E\big(X|X\ge Q(X;p),\,U\ge\de\big), 
\end{equation}
where $U$ is any r.v.\ which independent of $X$ and uniformly distributed on the interval $[0,1]$, $\de:=\de(X;p):=d\ii{X=Q(X;p)}$, and $d$ is any real number in the interval $[0,1]$ such that 
$$\P\!\big(X\ge Q(X;p)\big)-p=\P\big(X=Q(X;p)\big)\,d;$$ 
such a number $d$ always exists. 
Thus, the r.v.\ $U$ is used to split the possible atom of the distribution of $X$ at the quantile point $Q(X;p)$ in order to make the randomized tail probability 
$\P\!\big(X\ge Q(X;p),\,U\ge\de\big)$ exactly equal to $p$. 
Of course, in the absence of such an atom, one can simply write
\begin{equation}\label{eq:Q_1,simple}
Q_1(X;p)=Q(X;p)+\E\big(X-Q(X;p)\big|X\ge Q(X;p)\big)=\E\big(X|X\ge Q(X;p)\big). 	
\end{equation}

However, as pointed out in \cite{rocka-ur00,rocka-ur02}, a variational formula such as \eqref{eq:Q=inf} has a distinct advantage over such ostensibly explicit formulas as \eqref{eq:al=1} and \eqref{eq:Q_1=}, since \eqref{eq:Q=inf} allows for incorporation into  specialized optimization problems, with additional restrictions, say on the distribution of $X$; cf.\ e.g.\ \cite[Theorem~2]{rocka-ur00}. 


Nonetheless, let us obtain an extension of the representation \eqref{eq:al=1}, valid for all $\al\in[1,\infty)$. 
In accordance with \cite[Proposition~3.2]{pin-hoeff}, consider 
\begin{equation}\label{eq:x_**}
	x_{**}:=x_{**,X}:=\sup\big((\supp X)\setminus\{x_*\}\big)
	\in[-\infty,x_*]\subseteq[-\infty,\infty].  
\end{equation}

The following proposition will be useful.  

\begin{proposition}\label{prop:B str conv}\ 
\begin{enumerate}[(i)]
	\item\label{B conv} If $\al\in[1,\infty]$ then $B_\al(X;p)(t)$ is convex in $t\in T_\al$.
	\item\label{B str conv} If $\al\in(1,\infty)$ then $B_\al(X;p)(t)$ is strictly convex in $t\in(-\infty,x_{**}]\cap\R
	$.
	\item\label{B str conv,infty} $B_\infty(X;p)(t)$ is strictly convex in 
	$t\in
	\{s\in(0,\infty)\colon \E e^{X/s}<\infty\}$ unless $\P(X=c)=1$ for some $c\in\R$.
\end{enumerate} 
\end{proposition}


Suppose at this point that $\al\in[1,\infty)$. By part~(i) of Proposition~\ref{prop:att} (stated in Section~\ref{comput}), the minimum-attainment set 
\begin{equation}\label{eq:amin}
	\amin:=\{t\in\R\colon B_\al(X;p)(t)=Q_\al(X;p)\} 
\end{equation}
is nonempty and bounded. Also, this set is closed, by the continuity of $ B_\al(X;p)(t)$ in $t\in\R$. Therefore, the definition 
\begin{equation}\label{eq:ql}
	\ql:=\max\amin 
\end{equation}
makes sense, and 
\begin{equation}\label{eq:ql in R}
\ql\in\R; 	
\end{equation}
thus, $\ql$ is the largest value of $t$ minimizing $B_\al(X;p)(t)$. 
Moreover, in the case when $\al=1$ this largest minimizer is  
\begin{equation}\label{eq:qlz=Q0}
	\qlz=Q_0(X;p),  
\end{equation}
the largest $(1-p)$-quantile of $X$, 
as stated at the end of the paragraph containing \eqref{eq:al=1} and its proof. 
Thus, indeed the representation \eqref{eq:al=1} is extended to all $\al\in[1,\infty)$: 
\begin{equation}\label{eq:al>1}
	Q_\al(X;p)=\ql+p^{-1/\al}\,\big\|\big(X-\ql\big)_+\big\|_\al. 
\end{equation}
Further properties of $\ql$ are presented in 


\begin{proposition}\label{prop:ql} 
Suppose that $\al\in[1,\infty)$. 
Then the following statements are true. 
\begin{enumerate}[(i)]	
	\item\label{ql comp} Computation of $\ql$:
\begin{enumerate}[(a)]
	\item \label{ql,p<p*} If 
	$p\in(0,p_*]\cap(0,1)$ then $\ql=x_*=Q_0(X;p)
	$, 
	where $x_*$ is as in \eqref{eq:x_*,p_*}. 
	\item \label{ql,p>p*} Suppose here that $\al\in(1,\infty)$ and $p\in(p_*,1)$. 
	Then $\ql=t_{\al,p}$, where $t_{\al,p}$ is the only root $t\in(-\infty,x_*)$ of the equation $\pl\al=p$ and, again for $t\in(-\infty,x_*)$,  
\begin{equation}\label{eq:pl}
\kern-1cm
	\pl\al:=
	\Big(\frac{\|(X-t)_+\|_{\al-1}}{\|(X-t)_+\|_\al\ \ \;}\Big)^{(\al-1)\al}
=	\frac{\E^\al(X-t)_+^{\al-1}}{\E^{\al-1}(X-t)_+^\al}
=\P(X>t)\,
\frac{\E^\al\big((X-t)^{\al-1}|X>t\big)}{\E^{\al-1}\big((X-t)^\al|X>t\big)}.  
\end{equation}
Moreover, 
\begin{equation}\label{eq:ql<x**}
	\ql=t_{\al,p}\in(-\infty,x_{**}). 
\end{equation}
Furthermore, 
for all $t\in(-\infty,x_*)$ one has $\pl\al\le\P(X>t)\le P_0(X;t)$, and also 
$\pl\al\to\P(X>t)$ as $\al\downarrow1$ \big(assuming that $X\in\XX_\al$ for some $\al\in(1,\infty)$\big);
in addition, $\pl\al<\P(X>t)$ for all $t\in(-\infty,x_{**})$.  
\end{enumerate}
	\item \label{ql decr in p} $\ql$ is nonincreasing in $p\in(0,1)$ and $\ql=x_*=\mbox{\emph{const}}$ for $p\in(0,p_*]\cap(0,1)$. 
Moreover, if $\al\in(1,\infty)$ then $\ql$ is strictly decreasing in $p\in[p_*,1)$ and  continuous in $p\in(p_*,1)$; at that, $\ql\underset{p\downarrow p_*}\longrightarrow x_{**}$ and $\ql\underset{p\uparrow1}\longrightarrow-\infty$. 
	\item \label{ql decr in al} $\ql$ is nonincreasing in $\al\in[1,\infty)$ in the sense that $\qlbe\le\ql$ if $1\le\al<\be<\infty$ and $X\in\XX_\be$; in particular, 	
\begin{equation}\label{eq:ql<Q0}
	\ql\le\qlz=Q_0(X;p)=Q(X;p),  
\end{equation}
if $X\in\XX_\al$. 
More specifically, if $p\in(0,p_*]\cap(0,1)$, $1\le\al<\infty$, and $X\in\XX_\al$, then $\ql=x_*$; 
if $p\in(p_*,1)$, $1\le\al<\be<\infty$, and $X\in\XX_\be$, then $\qlbe<\ql$. 
Moreover, if $p\in(p_*,1)$ and $X\in\XX_\infty$, then $\ql\underset{\al\uparrow\infty}\longrightarrow-\infty$.  
	\item\label{ql consist} $\ql$ is consistent: ${}_{\al-1}Q(c;p)=c$ for all $c\in\R$. 
	\item\label{ql pos-hom} $\ql$ is positive-homogeneous: ${}_{\al-1}Q(\ka X;p)=\ka\,{}_{\al-1}Q(X;p)$ for all $\ka\in[0,\infty)$. 
	\item\label{ql tr-inv} $\ql$ is translation-invariant: ${}_{\al-1}Q(X+c;p)={}_{\al-1}Q(X;p)+c$ for all $c\in\R$. 
	\item\label{ql part mono} $\ql$ is partially monotonic: $\ql\le 
	{}_{\al-1}Q(c;p)$ if $X\le c$ for some $c\in\R$; also, $\ql\le{}_{\al-1}Q(X+c;p)$ for any $c\in[0,\infty)$. 
	\item\label{ql not mono} However, for any $\al\in(1,\infty)$ and any $p\in(0,1)$, $\ql$ is not monotonic in all $X\in\XX_\al$. 
	\item\label{ql not subadd} Consequently, for any $\al\in[1,\infty)$ and any $p\in(0,1)$, $\ql$ is not subadditive or convex in all $X\in\XX_\al$. 
\end{enumerate}
\end{proposition} 

By \eqref{eq:ql<Q0}, $\ql$ is a lower bound on the true $(1-p)$-quantile $Q(X;p)$ of $X$. Therefore and in view of part~\eqref{ql comp} of Proposition~\ref{prop:ql}, $\ql$ may be referred to as the lower $(\al-1,1-p)$-quantile of the r.v.\ $X$.


\begin{ex}\label{ex:Qleft,X_ab}\ 

\vspace*{6pt}
\noindent
\begin{parbox}{.38\textwidth}
	{ \includegraphics[width=.36\textwidth]{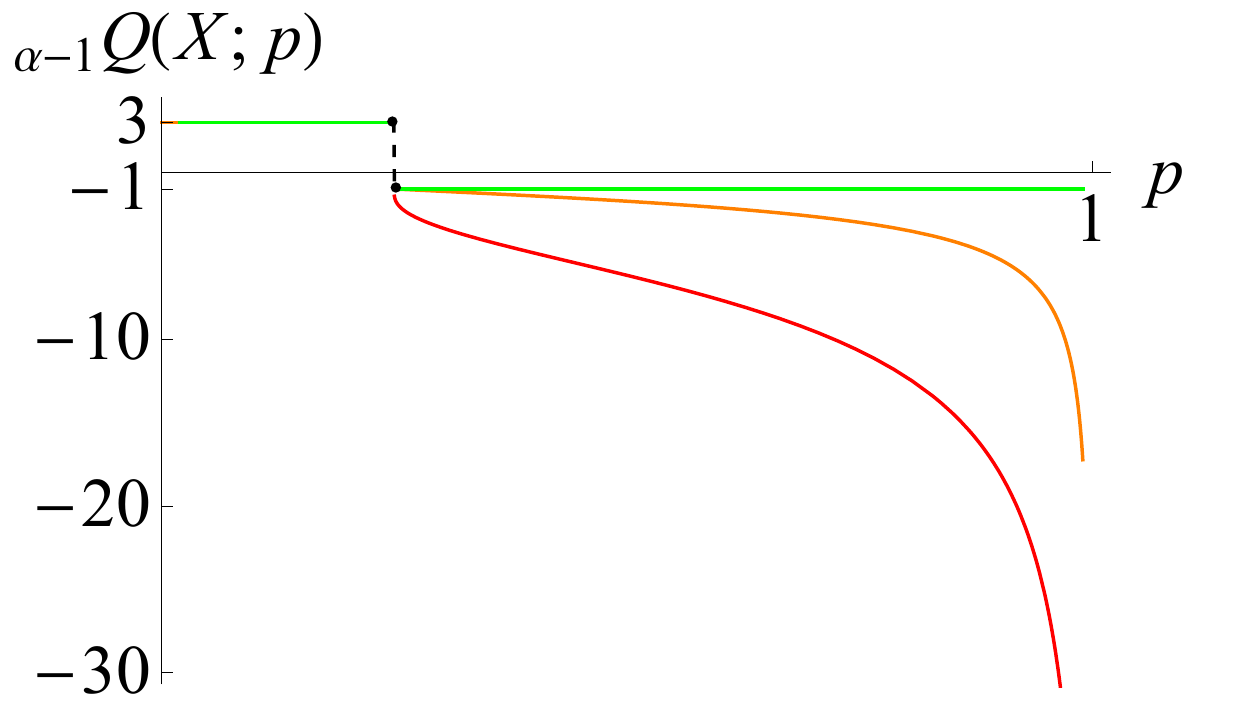} }
\end{parbox}
\begin{parbox}{.6\textwidth}
{Parts \eqref{ql decr in p} and \eqref{ql decr in al} of Proposition~\ref{prop:ql} are illustrated in the picture here on the left, with graphs $\big\{\big(p,\,\ql\big)\colon \break 0<p<0.99\big\}$ for a r.v.\ $X=X_{a,b}$ as in Examples~\ref{ex:P,X_ab} and \ref{ex:Q,X_ab}, with the same $a=1$ and $b=3$, and the values of $\al$ equal $1$ (green), $2$ (orange), and $5 
$ (red). 
Note that here $x_*=b=3$ and $x_{**}=-a=-1$. 
}
\end{parbox}
\end{ex}


\vspace*{6pt}

\begin{center}***\end{center}

One may conclude this section by an obvious but oftentimes rather useful observation 
that -- even when a minimizing value of $\la$ or $t$ in formulas \eqref{eq:P new}, \eqref{eq:P old}, or \eqref{eq:Q=inf} 
is not identified quite perfectly -- one still obtains, by those formulas, an upper bound on $P_\al(X;x)$ or $Q_\al(X;p)$ and hence on the true tail probability $\P(X\ge x)$ or the true quantile $Q(X;p)$, respectively.

\section{Implications for risk control/inequality modeling in finance/economics}\label{risk}

In financial literature -- see e.g.\ \cite{pflug00,rocka-ur02,kibzun-chern13}, the quantile bounds $Q_0(X;p)$  and $Q_1(X;p)$ are known as the \emph{value-at-risk} and \emph{conditional value-at-risk}, denoted as $\var_p(X)$ and $\cvar_p(X)$, respectively: 
\begin{equation}\label{eq:Q_1=cvar}
	Q_0(X;p)=\var_p(X)\quad\text{and}\quad Q_1(X;p)=\cvar_p(X);  
\end{equation} 
here, $X$ is interpreted as \emph{a priori} uncertain potential loss. 
The value of $Q_1(X;p)$ is also known as the expected shortfall (ES) \cite{acerbi-tasche02}, average value-at-risk (AVaR) \cite{rachev-etal}, and expected tail loss (ETL) \cite{litz-modest08}. 
As indicated in \cite{rocka-ur02}, at least in the case when there is no atom at the quantile point $Q(X;p)$, the quantile bound $Q_1(X;p)$ is also called the ``mean shortfall'' \cite{mausser-rosen}, whereas 
the difference $Q_1(X;p)-Q(X;p)$ is referred to as ``mean excess loss'' \cite{embrechts_etal,bassi_etal}. 


Greater values of $\al$ correspond to greater sensitivity to risk; cf.\ e.g.\ \cite{fishburn77}. 
For instance, let $X$ and $Y$ denote the potential losses corresponding to two different investments portfolios. 
Suppose that there are mutually exclusive events $E_1$ and $E_2$ and real numbers $p_*\in(0,1)$ and $\de\in(0,1)$ such that 
(i) $\P(E_1)=\P(E_2)=p_*/2$, (ii) the loss of either portfolio is $0$ if the event $E_1\cup E_2$ does not occur, (iii) the loss of the $X$-portfolio is $1$ if the event $E_1\cup E_2$ occurs, and (iv) the loss of the $Y$-portfolio is $1-\de$ if the event $E_1$ occurs, and it is $1+\de$ if the event $E_2$ occurs. 
Thus, the r.v.\ $X$ takes values $0$ and $1$ with probabilities $1-p_*$ and $p_*$, and the r.v.\ $Y$ takes values $0$, $1-\de$, and $1+\de$ with probabilities $1-p_*$, $p_*/2$, and $p_*/2$, respectively. 
So, $\E X=\E Y$, that is, the expected losses of the two portfolios are the same. 
Clearly, the distribution of $X$ is less dispersed than that of $Y$, both intuitively and also in the formal sense that $X\alle Y$ for all $\al\in[1,\infty]$. 
Therefore, everyone will probably say that the $Y$-portfolio is riskier than the $X$-portfolio. 
However, for any $p\in(p_*,1)$ it is easy to see, by \eqref{eq:Q_0}, that $Q_0(X;p)=0=Q_0(Y;p)$ and hence, in view of \eqref{eq:al=1}, $Q_1(Y;p)=\frac1p\,\E Y=\frac{p_*}p=\frac1p\,\E X=Q_1(X;p)$. 
Using also the continuity of $Q_\al(\cdot;p)$ in $p$, as stated in Theorem~\ref{th:coher}, one concludes that the $Q_1(\cdot;p)=\cvar_p(\cdot)$ risk value of the riskier $Y$-portfolio is the same as that of the less risky $X$-portfolio for all $p\in[p_*,1)$.  
Such indifference (which may also be referred to as insensitivity to risk) 
may generally be considered ``an unwanted characteristic'' \cite[pages~36, 48]{groot-haller04}. 

Let us now show that, in contrast with the risk measure $Q_1(\cdot;p)=\cvar_p(\cdot)$, the value of $Q_\al(\cdot;p)$ is sensitive to risk for all $\al\in(1,\infty)$ and all $p\in(0,1)$; that is, for all such $\al$ and $p$ and for the losses $X$ and $Y$ as above, $Q_\al(Y;p)>Q_\al(X;p)$. 
Indeed, take any $\al\in(1,\infty)$. 
By \eqref{eq:x_*,p_*} and \eqref{eq:x_**}, $x_{*,X}=1$, $p_{*,X}=p_*$, $x_{*,Y}=1+\de$, $x_{**,Y}=1-\de$, and $p_{*,Y}=p_*/2$. 
If $p\in(0,p_*/2]$ then, by part \eqref{it:p le p_*} of Proposition~\ref{prop:inverse}, 
$Q_\al(Y;p)=x_{*,Y}=1+\de>1=x_{*,X}=Q_\al(X;p)$. 
If now $p\in(p_*/2,1)$ then, by \eqref{eq:ql<x**}, 
$t_Y:={}_{\al-1}Q(Y;p)\in(-\infty,x_{**,Y})=(-\infty,1-\de)$. 
Also, by  strict version of Jensen's inequality
and the strict convexity of $u^\al$ in $u\in[0,\infty)$, $B_\al(X;p)(t)=t+p^{-1/\al}\|X-t\|_\al<t+p^{-1/\al}\|Y-t\|_\al=B_\al(Y;p)(t)$ for all $t\in(-\infty,1-\de]$. 
So, by \eqref{eq:al>1} and \eqref{eq:Q=inf}, 
$Q_\al(Y;p)=B_\al(Y;p)(t_Y)>B_\al(X;p)(t_Y)\ge Q_\al(X;p)$. 
Thus, it is checked that $Q_\al(Y;p)>Q_\al(X;p)$ for all $\al\in(1,\infty)$ and all $p\in(0,1)$.

The above example is illustrated in Figure~\ref{fig:sensitivity}, for $p_*=0.1$ and $\de=0.6$. It is seen that the sensitivity of the measure $Q_\al(\cdot;p)$ to risk \big(reflected especially by the gap between the red and blue lines for $p\in[p_*,1)=[0.1,\,1)$\big) increases from the zero sensitivity 
when $\al=1$ to an everywhere positive sensitivity when $\al=2$ to an everywhere greater positive sensitivity when $\al=5$. 

\begin{figure}[ht]
	\centering
		\includegraphics[width=1\textwidth]{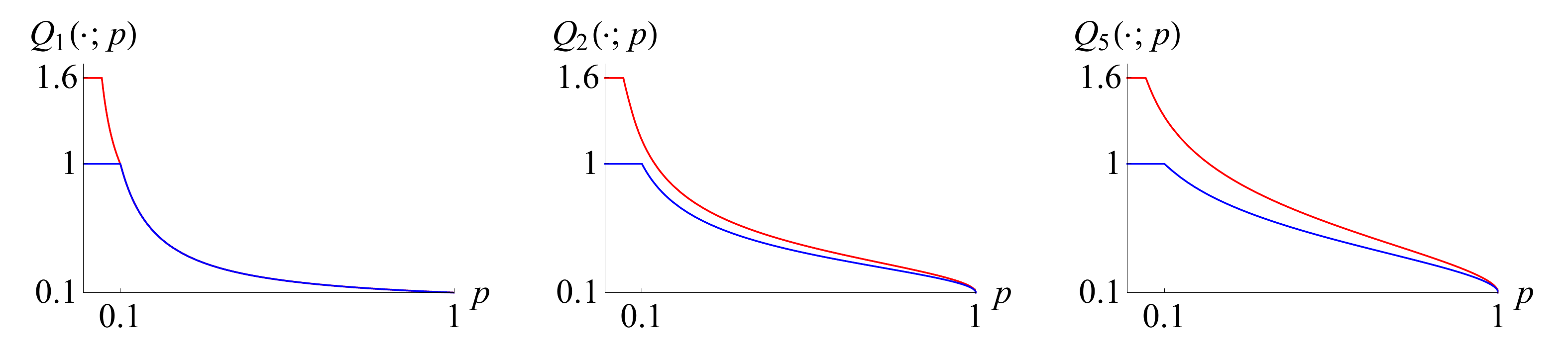}
	\caption{Sensitivity of $Q_\al(\cdot;p)$ to risk, depending on the value of $\al$: 
	graphs $\big\{\big(p,Q_\al(X;p)\big)\colon 0<p<1\big\}$ (blue) and $\big\{\big(p,Q_\al(Y;p)\big)\colon 0<p<1\big\}$ (red) for $\al=1$ (left panel), $\al=2$ (middle panel), and $\al=5$ (right panel).
	}
	\label{fig:sensitivity}
\end{figure}

\begin{center}***\end{center}

Based on an extensive and penetrating discussion of methods of measurement of market and nonmarket risks, Artzner \emph{et al} \cite{artzner.etal.99} 
concluded that, for a risk measure to be effective in risk regulation and management, it has to be 
\emph{coherent}, in the sense that it possess the translation
invariance, subadditivity, positive homogeneity, and monotonicity properties. In general, a risk measure, say $\hat\rho$, is a mapping of a linear space of 
real-valued r.v.'s on a given probability space into $\R$. The probability space (say $\Omega$) was assumed to be finite in \cite{artzner.etal.99}. More generally, one could allow $\Omega$ to be infinite, and then it is natural to allow $\hat\rho$ to take values $\pm\infty$ as well. In \cite{artzner.etal.99}, the r.v.'s (say $Y$) in the argument of the risk measure were called risks but at the same time interpreted as ``the investor's \emph{future net worth}''. Then the translation invariance was defined in \cite{artzner.etal.99} as the identity $\hat\rho(Y+rt)=\hat\rho(Y)-t$ for all r.v.'s $Y$ and real numbers $t$, where $r$ is a positive real number, interpreted as the rate of return. We shall, however, follow Pflug \cite{pflug00}, who considers a risk measure (say $\rho$) as a function of the potential cost/loss, say $X$, and then defines the translation invariance of $\rho$, quite conventionally, as the identity $\rho(X+c)=\rho(X)+c$ for all r.v.'s $X$ and real numbers $c$. 
The approaches in \cite{artzner.etal.99} and \cite{pflug00} are equivalent to each other, and the 
correspondence between them can be given by the formulas $\rho(X)=r\hat\rho(Y)=r\hat\rho(-X)$, $X=-Y$, and $c=-rt$. 
The positive homogeneity as defined in \cite{artzner.etal.99} can be stated as the identity $\rho(\la X)=\la\rho(X)$ for all r.v.'s $X$ and real numbers $\la\ge0$. 

\begin{corollary}\label{cor:coher}
For each $\al\in[1,\infty]$, the quantile bound $Q_\al(\cdot;p)$ is a coherent risk measure. 
\end{corollary}

This follows immediately from Theorem~\ref{th:coher}. 

The usually least trivial of the four properties characterizing the coherence is the subadditivity  of a risk measure -- which, in the presence of the positive homogeneity, is equivalent to the convexity, as was pointed out earlier in this paper. As is well known and also discussed above, the value-at-risk measure $\var_p(X)$ is translation
invariant, positive homogeneous, and monotone (in $X$), but it fails to be subadditive. Quoting \cite[page~1458]{rocka-ur02}: ``The coherence of [$\cvar_p(X)$] is a formidable advantage not shared by any other widely
applicable measure of risk yet proposed.'' 
Corollary~\ref{cor:coher} above addresses this problem by providing an entire infinite family of coherent risk measures, indexed by $\al\in[1,\infty]$, including $\cvar_p(X)=Q_1(X;p)$ just as one member of the family. 

Theorem~\ref{th:coher} also provides additional monotonicity and other useful properties of the spectrum of risk measures $Q_\al(\cdot;p)$. 
The terminology we use to name some of these properties differs from the corresponding terminology used in \cite{artzner.etal.99}. Namely, what we referred to as the ``sensitivity'' in Theorem~\ref{th:coher} corresponds to the ``relevance'' in \cite{artzner.etal.99}. 
Also, in the present paper the ``model-independence'' means that the risk measure depends on the potential loss only through the distribution of the loss, rather than on the way to model the ``states of nature'', on which the loss may depend. In contrast, in \cite{artzner.etal.99} a measure of risk is considered ``model-free'' if it does not depend, not only on modeling the ``states of nature'', but, to a possibly large extent, on the distribution of the loss. 
The ``model-independence'' property is called ``law-invariance'' in \cite[Section~12.1.2]{frit-gian04}, where the consistency property was referred to as ``constancy''. 
An example of such a ``model-free'' risk measure is given by the Securities and Exchange Commission (SEC) rules, 
described e.g.\ in \cite[Subsection~3.2]{artzner.etal.99}; this measure of risk depends only on the set of all possible representations of the investment portfolio in question as a portfolio of long call spreads, that is, pairs of the form (a long call, a short call). 
If a measure of risk is not ``model-free'', then it is called ``model-dependent'' in \cite{artzner.etal.99}. 

\begin{center}***\end{center}

Yitzhaki \cite{yitzhaki82} utilized the Gini mean difference -- which had prior to that been mainly used as a measure of economic inequality -- 
to construct, somewhat implicitly, a measure of risk; this approach was further developed in \cite{delq-cillo06,cillo-delq11}. 
If (say) a r.v.\ $X$ is thought of as the income of randomly selected person in a certain state, then the Gini mean difference can be defined by the formula 
\begin{equation*}
G_H(X):=\E H(|X-\tX|), 	
\end{equation*}
where $\tX$ is an independent copy of $X$ and $H\colon[0,\infty)\to\R$ is a measurable function, usually assumed to be nonnegative and such that $H(0)=0$; clearly, given the function $H$, the Gini mean difference $G_H(X)$ depends only on the distribution of the r.v.\ $X$. So, if $H(u)$ is considered, for any $u\in[0,\infty)$, as the measure of inequality between two individuals with incomes $x$ and $y$ such that $|x-y|=u$, then \emph{the Gini mean difference $\E H(|X-\tX|)$ is the mean $H$-inequality in income between two individuals selected at random} (and with replacement, thus independently of each other). The most standard choice for $H$ is the identity function $\id$, so that $H(u)=\id(u)=u$ for all $u\in[0,\infty)$. 
Based on the measure-of-inequality $G_H$, one can define the risk measure 
\begin{equation}\label{eq:R_H}
	R_H(X):=\E X+G_H(X)=\E X+\E H(|X-\tX|), 
\end{equation}
where now the r.v.\ $X$ is interpreted as the uncertain loss on a given investment, with the term $G_H(X)=\E H(|X-\tX|)$ then possibly interpreted as a measure of the uncertainty. 
Clearly, when there is no uncertainty, so that the loss $X$ is in fact a nonrandom real constant, then the measure $G_H(X)$ of the uncertainty is $0$, assuming that $H(0)=0$. 
If $X\sim N(\mu,\si^2)$ (that is, $X$ is normally distributed with mean $\mu$ and standard deviation $\si>0$) and $H=\ka\id$ for some positive constant $\ka$, then $R_H(X)=\mu+\frac{2\ka}{\sqrt\pi}\,\si$, a linear combination of the mean and the standard deviation, so that in such a case we find ourselves in the realm of the Markowitz mean-variance risk-assessment framework. 

It is assumed that $R_H(X)$ is defined when both expected values in the last expression in \eqref{eq:R_H} are defined and are not infinite values of opposite signs -- so that these two expected values could be added, as needed in \eqref{eq:R_H}. 

It is clear that $R_H(X)$ is translation-invariant. 
Moreover, 
$R_H(X)$ is convex in $X$ if the function $H$ is convex and nondecreasing. 
Further, if $H=\ka\id$ for some positive constant $\ka$, then $R_H(X)$ is also positive-homogeneous. 

It was shown in \cite{yitzhaki82}, under an additional technical condition, that $R_H(X)$ is nondecreasing in $X$ with respect to the stochastic dominance of order $1$ if $H=\frac12\id$. Namely, the result obtained in \cite{yitzhaki82} is that if $X\st Y$ and the distribution functions $F$ and $G$ of $X$ and $Y$ are such that $F-G$ changes sign only finitely many times on $\R$, then $R_{\frac12\id}(X)\le R_{\frac12\id}(Y)$. 
A more general result was obtained in \cite{delq-cillo06}, which can be stated as follows: in the case when the function $H$ is differentiable, $R_H(X)$ is nondecreasing in $X$ with respect to the stochastic dominance of order $1$ if and only if $|H'|\le\frac12$. Cf.\ also \cite{cillo-delq11}. 
The proof in \cite{delq-cillo06} was rather long and involved; in addition, it used a previously obtained result of \cite{machina82}. 
Here we are going to give (in Appendix~\ref{proofs}) a very short, direct, and simple proof of the more general 

\begin{proposition}\label{prop:gini}
The risk measure $R_H(X)$ is nondecreasing in $X$ with respect to the stochastic dominance of order $1$ if and only if the function $H$ is $\frac12$-Lipschitz: $|H(x)-H(y)|\le\frac12\,|x-y|$ for all $x$ and $y$ in $[0,\infty)$. 
\end{proposition}

In Proposition~\ref{prop:gini}, it is not assumed that $H\ge0$ or that $H(0)=0$. 
Of course, if $H$ is differentiable, then the $\frac12$-Lipschitz condition is equivalent to the  condition $|H'|\le\frac12$ in \cite{delq-cillo06}. 

The risk measure $R_H(X)$ was called mean-risk (M-R) in \cite{cillo-delq11}. 

It follows from \cite{delq-cillo06} or Proposition~\ref{prop:gini} above that the risk measure $R_{\ka\id}(X)$ is coherent for any $\ka\in[0,\frac12]$. In fact, based on Proposition~\ref{prop:gini}, one can rather easily show more: 


\begin{proposition}\label{prop:gini coher}
The risk measure $R_H(X)$ is coherent if and only if $H=\ka\id$ for some $\ka\in[0,\frac12]$. 
\end{proposition}

It is possible to indicate a relation -- albeit rather indirect -- of the risk measure $R_H(X)$, defined in \eqref{eq:R_H}, with the quantile bounds $Q_\al(X;p)$. 
Indeed, introduce 
\begin{equation}\label{eq:hat Q}
	\hat Q_\al(X;p)=\E X+p^{-1/\al}\,\big\|(X-\E X)_+\big\|_\al,   
\end{equation}
assuming $\E X$ exists in $\R$. 
By \eqref{eq:Q=inf}--\eqref{eq:B}, $\hat Q_\al(X;p)$ is another majorant of $Q_\al(X;p)$, obtained by using $t=\E X$ in \eqref{eq:Q=inf} as a surrogate of the minimizing value of $t$. 

The term $p^{-1/\al}\,\big\|(X-\E X)_+\big\|_\al$ in \eqref{eq:hat Q} is somewhat similar to the Gini mean-difference term $\E H(|X-\tX|)$, at least when $\al=1$ and (the distribution of) the r.v.\ $X$ is symmetric about its mean. 

Moreover, if the distribution of $X-\E X$ is symmetric and stable with index $\ga\in(1,2]$, then $\hat Q_1(X;p)=R_{\ka\id}(X)$ with 
$\ka=2^{-1-1/\ga}/p$. 

One may want to compare the two considered 
kinds of coherent measures of risk/inequality, $R_{\ka\id}(X)$ for $\ka\in[0,\frac12]$ and $Q_\al(X;p)$ for $\al\in[1,\infty]$ and $p\in(0,1)$. 
It appears that the latter measure is more flexible, as it depends on two parameters ($\al$ and $p$) rather than one just one parameter ($\ka$). 
Moreover, as Proposition~\ref{prop:Q close} shows, rather generally $Q_\al(X;p)$ retains a more or less close relation with the quantile $Q_0(X;p)$ -- which, recall, is the widely used value-at-risk (VaR). On the other hand, recall here that, in contrast with the VaR, $Q_\al(X;p)$ is coherent for $\al\in[1,\infty]$. 
However, both of these kinds of coherent measures appear useful, each in its own manner, representing two different ways to express risk/inequality. 


Formulas \eqref{eq:R_H} and \eqref{eq:hat Q} can be considered special instances of the general relation between risk measures and measures of inequality established in 
\cite{rocka-ur-zab}. 
Let $\XX_{\E}$ be a convex cone of real-valued r.v.\ $X\in\XX$ with a finite mean $\E X$ such that $\XX_{\E}$ contains all real constants. 
 
Largely following \cite{rocka-ur-zab}, let us say a coherent real-valued risk measure $R\colon\XX_{\E}\to(-\infty,\infty]$ 
is \emph{strictly expectation-bounded} if $R(X)>\E X$ for all $X\in\XX_{\E}$. 
\big(Note that here the r.v.\ $X$ represents the loss, whereas in \cite{rocka-ur-zab} it represents the gain; accordingly, $X$ in this paper corresponds to $-X$ in \cite{rocka-ur-zab}; also, in \cite{rocka-ur-zab} the cone $\XX_{\E}$ was taken to be the space $\LL^2$.\big) 
In view of Theorem~\ref{th:coher} and part~\eqref{it:>EX} of Proposition~\ref{prop:inverse}, it follows that $Q_\al(X;p)$ is a coherent and strictly expectation-bounded risk measure if $\al\in[1,\infty]$. 
Also (cf.\ \cite[Definition~1 and Proposition~1]{rocka-ur-zab}), let us say that a mapping $D\colon\XX_{\E}\to[0,\infty]$ is a \emph{deviation measure} if $D$ is subadditive, positive-homogeneous, and nonnegative with $D(X)=0$ if and only if $\P(X=c)=1$ for some real constant $c$;  
here $X$ is any r.v.\ in $\XX_{\E}$. 
Next (cf.\ \cite[Definition~2]{rocka-ur-zab}), let us say that a deviation measure $D\colon\XX_{\E}\to[0,\infty]$ is \emph{upper-range dominated} if $D(X)\le\sup\supp X-\E X$ for all $X\in\XX_{\E}$.  
Then (cf.\ \cite[Theorem~2]{rocka-ur-zab}), the formulas 
\begin{equation}\label{eq:R,D}
	D(X)=R(X-\E X)\quad\text{and}\quad R(X)=\E X+D(X)
\end{equation}
provide a one-to-one correspondence between all coherent strictly expectation-bounded risk measures $R\colon\XX_{\E}\to(-\infty,\infty]$ and all upper-range dominated deviation measures $D\colon\XX_{\E}\to[0,\infty]$. 

In particular, it follows that the risk measure $\hat Q_\al(\cdot;p)$, defined by formula \eqref{eq:hat Q}, is coherent for all $\al\in[1,\infty]$ and all $p\in(0,\infty)$. 
It also follows that $X\mapsto Q_\al(X-\E X;p)$ is a deviation measure. 
As was noted, $\hat Q_\al(X;p)$ is a majorant of $Q_\al(X;p)$. 
In contrast with $Q_\al(X;p)$, in general $\hat Q_\al(X;p)$ will not have such a close hereditary relation with the true quantile $Q_0(X;p)$ as e.g.\ the ones given in 
Proposition~\ref{prop:Q close}. For instance, if $\P(X\ge x)$ is like $x^{-\infty}$ then, by \eqref{eq:Q close}--\eqref{eq:K}, $Q_\al(X;p)\underset{p\downarrow0}\sim Q_0(X;p)$ for each $\al\in[0,\infty]$, whereas $\hat Q_\infty(X;p)=\infty$ for all real $p>0$. On the other hand, in distinction with the definition \eqref{eq:hat Q} of $\hat Q_\al(X;p)$, the expression \eqref{eq:Q=inf} for $Q_\al(X;p)$ requires minimization in $t$; however, that minimization will add comparatively little to the complexity of the problem of minimizing $Q_\al(X;p)$ subject to a usually large number of restrictions on the distribution of $X$; cf.\ again e.g.\ \cite[Theorem~2]{rocka-ur00}. 
%


One may also consider the following modification of $\hat Q_\al(X;p)$, which is still a majorant of $Q_\al(X;p)$, but is closer \big(than $\hat Q_\al(X;p)$ is\big) to the true quantile $Q_0(X;p)$ -- at least when $p\in(0,1)$ is close enough to $1$ \big(recall part~\eqref{ql decr in p} of 
Proposition~\ref{prop:
ql}\big): 
\begin{equation}\label{eq:hat Q^>}
\hat Q^<_\al(X;p):=\inf_{t\in(-\infty,\E X]}B_\al(X;p)(t); 
\end{equation}
cf.\ \eqref{eq:Q=inf}. 

The risk measure $\hat Q^<_\al(\cdot;p)$ is coherent and strictly expectation-bounded given the condition $\al\in[1,\infty]$, which will be assumed in this paragraph. The proof of the translation invariance, positive homogeneity, and subadditivity properties of $\hat Q^<_\al(\cdot;p)$ is almost the same as the proof of these properties of $Q_\al(\cdot;p)$, listed in Theorem~\ref{th:coher}. 
Then, to prove the monotonicity of $Q_\al(\cdot;p)$ (with respect to the order $\st$), it is enough, in view of the subadditivity of $\hat Q^<_\al(\cdot;p)$ and as in the proof of \cite[Theorem~2]{rocka-ur-zab}, to show that $X\le0$ implies $\hat Q^<_\al(X;p)\le0$, which  obtains indeed, because $\hat Q^<_\al(X;p)\le\hat Q_\al(X;p)$  and $\hat Q_\al(X;p)\le\hat Q_\al(0;p)=0$ if $X\le0$. 
Finally, $\hat Q^<_\al(\cdot;p)$ is strictly expectation-bounded, because it majorizes $Q_\al(\cdot;p)$, which is strictly expectation-bounded, as noted. 


\begin{center}***\end{center}

Recalling \eqref{eq:check Q} and following \cite{ortobelli_etal06,ortobelli_etal07,ortobelli_etal09}, one may also consider $
-F_{-X}^{(-\al)}(p)$ as a measure of risk. Here one will need the following semigroup identity, given in \cite[(8a)]{ortobelli_etal06} (cf.\ e.g.\ \cite[Remark~3.7]{pin98}):  
\begin{equation}\label{eq:check Q=}
	F_X^{(-\al)}(p)=\frac1{\Ga(\al-\nu)}\int_0^p(p-u)^{\al-\nu-1}F_X^{(-\nu)}(u)\dd u    
\end{equation}
whenever $0<\nu<\al<\infty$. 
%
%
The following proposition is well known. 
 
\begin{proposition}\label{prop:lorentz} If the r.v.\ $X$ is nonnegative then 
\begin{equation}\label{eq:L=}
	F_X^{(-2)}(p)=L_X(p)=-p\cvar_p(-X), 
\end{equation}
where $L_X$ is the Lorenz curve function, given by the formula 
\begin{equation}\label{eq:lorenz}
	L_X(p):=\int_0^p F_X^{-1}(u)
	\dd u.  
\end{equation}
\end{proposition}

Indeed, the first equality in \eqref{eq:L=} is the special case of the identity \eqref{eq:check Q=} with $\al=2$ and $\nu=1$, and the second equality in \eqref{eq:L=} follows by \cite[part~(i) of Theorem~3.1]{ogr-rusz04_SIAM}, identity \eqref{eq:Q=inf} for $\al=1$, and the second identity in \eqref{eq:Q_1=cvar}. 
Cf.\ \cite[Theorem~2]{kibzun-kuzn} and \cite{atkinson70,muliere-scarsini}. 


Using 
\eqref{eq:check Q=} 
with $\nu=2$, $\al+1$ in place of $\al$, and $-X$ in place of $X$ together with  Proposition~\ref{prop:lorentz}, one has  
\begin{equation}\label{eq:mix}
	-F_{-X}^{(-\al-1)}(p)=\frac1{\Ga(\al-1)}\int_0^p(p-u)^{\al-2}\,u\cvar_u(X)\dd u      
\end{equation} 
for any $\al\in(1,\infty)$. 
Since $\cvar_u(X)$ is a coherent risk measure, it now follows that, as noted in \cite{ortobelli_etal07}, 
$-F_{-X}^{(-\al-1)}(p)$ is a coherent risk measure as well, again for $\al\in(1,\infty)$; by \eqref{eq:L=}, this conclusion will hold for $\al=1$. 
%
However, one should remember that the expression $F_X^{(-\al)}(p)$ was defined 
only when the r.v.\ $X$ is nonnegative (and otherwise some of the crucial considerations above will not hold). 
Thus, the risk measure $-F_{-X}^{(-\al-1)}(p)$ is defined only if $X\le0$ almost surely.  

%

In view of \eqref{eq:mix}, this risk measure is a mixture of the coherent risk measures $\cvar_u(X)$ and thus a member of the general class of  
%
the so-called spectral risk measures \cite{acerbi02}, which are precisely the mixtures, over the values $u\in(0,1)$, of the risk measures $\cvar_u(X)$; thus, all spectral risk measures are automatically coherent. 
However, in general such measures will lack such an important variational representation as the one given by formula \eqref{eq:Q=inf} for the risk measure $Q_\al(X;p)$. 
Of course, for any ``mixing'' nonnegative Borel measure $\mu$ on the interval $(0,1)$ and the corresponding spectral risk measure 
\begin{equation*}
	\cvar_\mu(X):=\int_{(0,1)}\cvar_u(X)\,\mu(\dd u), 
\end{equation*}
one can write 
\begin{equation}\label{eq:many mins} 
	\cvar_\mu(X)=\int_{(0,1)}\inf_{t\in\R}\big(t+\tfrac1u\,\|(X-t)_+\|_1\big)\,\mu(\dd u), 
\end{equation}
in view of \eqref{eq:Q_1=cvar} and \eqref{eq:Q=inf}--\eqref{eq:B}. 
However, in contrast with \eqref{eq:Q=inf}, the minimization (in $t\in\R$) in \eqref{eq:many mins} needs in general to be done for each of the infinitely many values of $u\in(0,1)$.  
If the r.v.\ $X$ takes only finitely many values, then the expression of $\cvar_\mu(X)$ in \eqref{eq:many mins} can be rewritten as a finite sum, so that the minimization in $t\in\R$ will be needed only for finitely many values of $u$; cf.\ e.g.\ the optimization problem on page~8 in \cite{ortobelli_etal07}. 

On the other hand, one can of course consider arbitrary mixtures in $p\in(0,1)$ and/or $\al\in[1,\infty)$ of the risk measures $Q_\al(X;p)$. 
Such mixtures will automatically be coherent. Also, all mixtures of the measures $Q_\al(X;p)$ in $p$ will be nondecreasing in $\al$, and all mixtures of $Q_\al(X;p)$ in $\al$ will be nonincreasing in $p$. 

Deviation measures such as the ones studied in \cite{rocka-ur-zab} and discussed in the paragraph containing \eqref{eq:R,D} can be used as measures of economic inequality if the r.v.\ $X$ models, say, the random income/wealth -- defined as the income/wealth of an (economic) unit chosen at random from a population of such units. Then, according to the one-to-one correspondence given by \eqref{eq:R,D}, coherent risk measures $R$ translate into deviation measures $D$, and vice versa. 


However, the risk measures $Q_\al(\cdot;p)$ themselves can be used to express certain aspects of economic inequality directly, without translation into deviation measures. 
For instance, if $X$ stands for the random wealth then the statement $Q_1(X;0.01)=30\E X$ formalizes the common kind of expression ``the wealthiest 1\% own 30\% of all wealth'',  provided that the wealthiest 1\% can be adequately defined, say as follows: there is a threshold wealth value $t$ such that the number of units with wealth greather than or equal to $t$ is $0.01N$, where $N$ is the number of units in the entire population. 
Then (cf.\ \eqref{eq:Q_1,simple}) 
$0.01\,N\,Q_1(X;0.01)=0.01\,N\,\E(X|X\ge t)=N\E X\ii{X\ge t}=0.30\,N\E X$, whence indeed $Q_1(X;0.01)=30\E X$. 
Similar in spirit expressions of economic inequality in terms of $Q_\al(X;p)$ can be provided for all $\al\in(0,\infty)$. For instance, suppose now that $X$ stands for the annual income of a randomly selected household, whereas $x$ is a particular annual household income level in question. Then, in view of \eqref{eq:Q=inf}--\eqref{eq:B}, the inequality 
$Q_\al(X;p)\ge x$ means that for any (potential) annual household income level $t$ 
less than the maximum annual household income level $x_{*,X}$ in the population, 
the conditional $\al$-mean $\E\big((X-t)^\al|X>t\big)^{1/\al}$ of the excess $(X-t)_+$ of the random income $X$ over $t$ is no less than $\big(\frac p{\P(X>t)}\big)^{1/\al}$ times the excess $(x-t)_+$ of the income level $x$ over $t$. 
Of course, the conditional $\al$-mean $\E\big((X-t)^\al|X>t\big)^{1/\al}$ is increasing in $\al$. Thus, using the measure $Q_\al(X;p)$ of economic inequality
with a greater value of $\al$ means 
treating high values of the economic variable $X$ in a more progressive/sensitive manner. One may also note here that the above interpretation of the inequality 
$Q_\al(X;p)\ge x$ is a ``synthetic'' statement in the sense that is provides information concerning all values of potential interest of the threshold annual household income level $t$. 

Not only the upper bounds $Q_\al(X;p)$ on the quantile $Q(X;p)$, but also the upper bounds $\P_\al(X;x)$ on the tail probability $\P(X\ge x)$ may be considered measures of risk/inequality.  
Indeed, if $X$ is interpreted as the potential loss, then the tail probability $\P(X\ge x)$ corresponds to the classical safety-first (SF) risk measure; see e.g.\ 
\cite{roy52,giacom-orto04}. 

Using variational formulas -- of which formulas \eqref{eq:P new}, \eqref{eq:P old}, and \eqref{eq:Q=inf} are examples -- to define or compute measures of risk is not peculiar to the present paper.  
%
Indeed, as mentioned previously, the special case of \eqref{eq:Q=inf} with $\al=1$ is the well-known variational representation \eqref{eq:al=1} of $\cvar$, obtained in \cite{rocka-ur00,pflug00,rocka-ur02}.
The risk measure given by the SEC rules \cite[Subsection~3.2]{artzner.etal.99}, also mentioned before, is another example where the calculations are done, in effect, according to a certain minimization formula, which is somewhat implicit and complicated in that case. 

One can now list some of the advantages of the risk/inequality measures $P_\al(X;x)$ and $Q_\al(X;p)$: 
\begin{itemize} 
	\item $P_\al(X;x)$ and $Q_\al(X;p)$ are three-way monotonic and three-way stable -- in $\al$, $p$, and $X$. 
	\item The monotonicity in $X$ is graded continuously in $\al$, resulting in various, controllable degrees of sensitivity of $P_\al(X;x)$ and $Q_\al(X;p)$ to financial risk/economic inequality. 
	\item $x\mapsto P_\al(X;x)$ is the tail-function of a certain probability distribution. 
	\item $Q_\al(X;p)$ is a $(1-p)$-percentile of that probability distribution.  
	\item For small enough values of $p$ the quantile bounds $Q_\al(X;p)$ are close enough to the corresponding true quantiles $Q(X;p)$ provided that the right tail of the distribution of $X$ is light enough and regular enough, depending on $\al$.
	\item $P_\al(X;x)$ and $Q_\al(X;p)$ are solutions to mutually dual optimizations problems, which can be comparatively easily incorporated into more specialized optimization problems, with additional restrictions, say on the distribution of the random variable $X$. 
	\item $P_\al(X;x)$ and $Q_\al(X;p)$ are effectively computable. 
	\item Even when the corresponding minimizer 
is not identified quite perfectly -- one still obtains an upper bound on the risk/inequality measures  $P_\al(X;x)$ or $Q_\al(X;p)$.  
	\item Optimal upper bounds on $P_\al(X;x)$ and hence on $Q_\al(X;p)$ over important classes of r.v.'s $X$ represented (say) as sums of independent r.v.'s $X_i$ with restrictions on moments of the $X_i$'s and/or sums of such moments can be given, as is done e.g.\ in \cite{pin99,bent-liet02,bent-ap,normal,asymm,pin-hoeff}. 
	\item The quantile bounds $Q_\al(X;p)$ with $\al\in[1,\infty]$ constitute a spectrum of coherent measures of financial risk and economic inequality. 
	\item The r.v.'s $X$ of which the measures $P_\al(X;x)$ and $Q_\al(X;p)$ are taken are allowed to take values of both signs. In particular, if, in a context of economic inequality, $X$ is interpreted as the net amount of assets belonging to a randomly chosen economic unit, then a negative value of $X$ corresponds to a unit with more liabilities than paid-for assets. Similarly, if $X$ denotes the loss on a financial investment, then a negative value of $X$ will obtain when there actually is a net gain. 
\end{itemize}
Some of these advantages, and especially their totality, appear to be unique to the bounds proposed here. 

\medskip

\textbf{Acknowledgment.}\ I am pleased to thank Emmanuel Rio for 
the mentioned communication \cite{rio06.17.13}, which also included a reference to \cite{kibzun-chern13} and in fact sparked the study presented here. 

\appendix
\section{Proofs}\label{proofs}


\begin{proof}[Proof of Proposition~\ref{prop:P}]
Part (i) of the proposition follows immediately from \eqref{eq:P new}, \eqref{eq:A new}, and \eqref{eq:f incr,cont}. 

Parts (ii) and (iii) follow by \eqref{eq:P old}--\eqref{eq:A old}. Indeed, 
$\E(X-t)_+^\al\ge\E X_+^\al/2^{(\al-1)_+}-|t|^\al=\infty$ for all real $x$ and $t$ if $\E X_+^\al=\infty$, 
and $\E e^{(X-x)/t}=e^{-x/t}\E e^{X/t}=\infty$ for all real $x$ and $t>0$ if $\E e^{\la X}=\infty$ for all real $\la>0$. 

Concerning part (iv) of Proposition~\ref{prop:P}, assume indeed that $\al\in(0,\infty)$ and $\E X_+^\al<\infty$. Then, for any $x>0$, \eqref{eq:P old}--\eqref{eq:A old} imply $P_\al(X;x)\le\frac{\E X_+^\al}{x_+^\al}\to0$ as $x\to\infty$. On the other hand, obviously $P_\al(X;x)\ge0$ for all real $x$. So, indeed, $P_\al(X;x)\to0$ as $x\to\infty$. 

By \eqref{eq:P,tA}, \eqref{eq:A new}, and \eqref{eq:f}, $P_\al(X;x)\le A_\al(X;x)(0)=h_\al(0)=1$. On the other hand, by \eqref{eq:0<al},  $P_\al(X;(-\infty)+)\ge P_0(X;(-\infty)+)=1$. So, indeed $P_\al(X;(-\infty)+)=1$.  

Thus, part (iv) of Proposition~\ref{prop:P} is proved. 

The proof of part (v) is rather similar to that of part (iv). Assume indeed that $\al=\infty$ and $\E e^{\la_0 X}<\infty$ for some real $\la_0>0$. Then 
$P_\infty(X;x)\le\E e^{\la_0(X-x)}=e^{-\la_0 x}\E e^{\la_0 X}\to0$ as $x\to\infty$. Since $P_\infty(X;x)\ge0$ for all real $x$, one indeed has $P_\infty(X;x)\to0$ as $x\to\infty$. 

As for the proof of the statement that $P_\al(X;x)\to1$ as $x\to-\infty$ for $\al=\infty$, it is the same as the corresponding proof for $\al\in(0,\infty)$. 
  
Thus, part (v) of Proposition~\ref{prop:P} is proved as well. 
\end{proof}

\begin{proof}[Proof of Proposition~\ref{prop:P,+}]\ 

\eqref{it:x ge x_*}  Let us first verify part \eqref{it:x ge x_*}. 
For $\al=0$, this follows immediately from the equality in \eqref{eq:0<al} and the definitions in \eqref{eq:x_*,p_*}. 

Take then any $\al\in(0,\infty]$. 
Take indeed any $x\in[x_*,\infty)$. 

If $\al\in(0,\infty)$ and $u\in(-\infty,x_*]$, then $h_\al\big(\la(u-x)\big)=\big(1+\la(u-x)/\al\big)_+^\al\underset{\la\to\infty}\longrightarrow\ii{u=x}$;  
so, by \eqref{eq:A new}, 
the condition $X\in\XX_\al$, and dominated convergence, 
$A_\al(X;x)(\la)
=\E h_\al\big(\la(X-x)\big)\underset{\la\to\infty}\longrightarrow\P(X=x)$. 
The case $\al=\infty$ is similar: $A_\infty(X;x)(\la)
=\E e^{\la(X-x)}\ii{X\le x}\underset{\la\to\infty}\longrightarrow\P(X=x)$.   

So, by \eqref{eq:P new}, $P_\al(X;x)\le\P(X=x)=\P(X\ge x)$. Now part \eqref{it:x ge x_*}  of Proposition~\ref{prop:P,+} follows in view of the inequality in \eqref{eq:0<al}.  

\eqref{it:x<x_*} Part \eqref{it:x<x_*} of Proposition~\ref{prop:P,+} follows because, by \eqref{eq:0<al} and \eqref{eq:x_*,p_*},   
$P_\al(X;x)\ge\P(X\ge x)>0$ for all $x\in(-\infty,x_*)$. 

\eqref{it:conc} Concerning part~\eqref{it:conc} of Proposition~\ref{prop:P,+}, 
consider first the case $\al\in(0,\infty)$. Then, by 
\eqref{eq:P old}, the function 
\begin{equation}\label{eq:^-1/al}
(-\infty,x_*]\cap\R\ni x\mapsto P_\al(X;x)^{-1/\al}\in[0,\infty]	
\end{equation}
is the pointwise supremum, in $t\in(-\infty,x_*)$, of the family of continuous convex functions 
$(-\infty,x_*]\cap\R\ni x\mapsto\frac{(x-t)_+}{\|(X-t)_+\|_\al}\in[0,\infty]$. 
So, the function \eqref{eq:^-1/al} is convex. 
It is also finite on the interval $(-\infty,x_*)$, by part \eqref{it:x<x_*} of Proposition~\ref{prop:P,+}. 
So, the function \eqref{eq:^-1/al} is continuous on $(-\infty,x_*)$. 
Moreover, this function is 
lower-semicontinuous and nondecreasing, and hence continuous at the point $x_*$, in the case when $x_*\in\R$. 
Thus, the function \eqref{eq:^-1/al} is continuous on the entire set $(-\infty,x_*]\cap\R$,    
with respect to the natural topologies on $(-\infty,x_*]\cap\R$ and $[0,\infty]$. 

The case $\al=\infty$ is considered quite similarly. Here, instead of \eqref{eq:^-1/al}, one works with the function 
\begin{equation}\label{eq:-ln}
(-\infty,x_*]\cap\R\ni x\mapsto-\ln P_\al(X;x)\in(-\infty,\infty], 	
\end{equation}
which is 
the pointwise supremum, in real $\la>0$ such that $\E e^{\la X}<\infty$, of the family of continuous convex (in fact, affine) functions 
$\R\ni x\mapsto-\ln\E e^{\la(X-x)}=\la x-\ln\E e^{\la X}\in(-\infty,\infty]$. 
\big(Actually, the values of the latter family of functions are all real, for $\la>0$ such that $\E e^{\la X}<\infty$, whereas all the values of the function \eqref{eq:-ln} are in $[0,\infty]$; however, here we take the union, $(-\infty,\infty]$, of the sets $\R$ and $[0,\infty]$ as an interval which is guaranteed to contain all possible values of all the convex functions under consideration.\big) 
Here we also use the standard conventions $\ln0:=-\infty$ and $e^{-\infty}:=0$; concerning the continuity of functions with values in the set $(-\infty,\infty]$, we use the natural topology on this set. 

\eqref{it:cont} Let us now turn to part \eqref{it:cont} of Proposition~\ref{prop:P,+}. Consider first the case $\al\in(0,\infty)$. 
Then, since the map $[0,\infty]\ni r\mapsto r^{-\al}$ is continuous, it follows by part \eqref{it:conc} of Proposition~\ref{prop:P,+} that $P_\al(X;x)$ is indeed 
continuous in $x\in(-\infty,x_*)$ and left-continuous in $x$ at $x_*$ if $x_*\in\R$. 
The case $\al=\infty$ is quite similar; here, instead of the map $[0,\infty]\ni r\mapsto r^{-\al}$, one should use the continuous map $(-\infty,\infty]\ni r\mapsto e^{-r}$.

\eqref{it:left-cont} That the function  
$\R\ni x\mapsto P_\al(X;x)$ is left-continuous follows immediately from parts~\eqref{it:cont} and \eqref{it:x ge x_*} if $\al\in(0,\infty]$, and from the equality in \eqref{eq:0<al} if $\al=0$. 

\eqref{it:x_al incr} That $x_\al$ is nondecreasing in $\al\in[0,\infty]$ follows immediately from the definition of $x_\al$ in \eqref{eq:x_al} and \eqref{eq:mono in al}. 
That $x_\al<\infty$ follows because, by 
\eqref{eq:P decr in x},  
$P_\al(X;x)\to0<1$ as $x\to\infty$. 

\eqref{it:EX} 
If 
$x\in(-\infty,\E X]$ then, by Jensen's inequality, 
$$A_1(X;x)(\la)=\E\big(1+\la(X-x)\big)_+\ge\big(1+\la(\E X-x)\big)_+\ge1$$ 
for all $\la\in(0,\infty)$, whence, by \eqref{eq:P new}, $P_1(X;x)\ge1$, and so, by 
\eqref{eq:mono in al}, $P_\al(X;x)\ge1$ for all $\al\in[1,\infty]$. 

On the other hand, if 
$x\in(\E X,\infty)$ then $A_\infty(X;x)(\la)=\E e^{\la(X-x)}<\E e^{0(X-x)}=1$ for all $\la$ in a right neighborhood of $0$ -- because the right derivative of $\E e^{\la(X-x)}$ in $\la$ at $\la=0$ is $\E(X-x)<0$;  
therefore, $P_\infty(X;x)<1$, and so, again by \eqref{eq:mono in al}, 
$P_\al(X;x)<1$ for all $\al\in[1,\infty]$. 

This completes the proof of part \eqref{it:EX} of Proposition~\ref{prop:P,+}. 

\eqref{it:x_al<x_*} By part \eqref{it:x ge x_*} of Proposition~\ref{prop:P,+}, $P_\al(X;x)=0<1$ for all $x\in(x_*,\infty)$, so that $(x_*,\infty)\subseteq E_\al(1)$, which implies $x_\al\le x_*$. 

Let us show that $x_\al=x_*$ if and only if $p_*=1$. 
By the definition of $x_\al$ in \eqref{eq:x_al} and \eqref{eq:P decr in x}, $P_\al(X;x)=1$ for all $x\in(-\infty,x_\al)$. 
So, by part \eqref{it:left-cont} of Proposition~\ref{prop:P,+} and 
the inequality $x_\al<\infty$ in part~\eqref{it:x_al incr} of Proposition~\ref{prop:P,+}, 
\begin{equation}\label{eq:P=1}
	x_\al>-\infty\implies P_\al(X;x_\al)=\lim_{x\uparrow x_\al}P_\al(X;x)=1.   
\end{equation}

If now $x_\al=x_*$, then $x_\al>-\infty$ by the definition of $x_*$ in \eqref{eq:x_*,p_*};  
so, by part~\eqref{it:x ge x_*} of Proposition~\ref{prop:P,+} and \eqref{eq:P=1}, $p_*=P_\al(X;x_*)=P_\al(X;x_\al)=1$, which proves the implication $x_\al=x_*\implies p_*=1$. 
Vice versa, suppose that $p_*=1$. Then necessarily $x_*\in\R$. Moreover, by part~(i) of Proposition~\ref{prop:P} and part~\eqref{it:x ge x_*} of Proposition~\ref{prop:P,+}, for all $x\in(-\infty,x_*]$ one has $P_\al(X;x)\ge P_\al(X;x_*)=p_*=1$, so that $E_\al(1)\subseteq(x_*,\infty)$ and hence $x_\al\ge x_*$. Now the conclusion $x_\al=x_*$ follows by the already established inequality $x_\al\le x_*$. 

\eqref{it:J_al} By the definition of $x_\al$ in \eqref{eq:x_al} and part (i) of Proposition~\ref{prop:P}, the set $E_\al(1)$ is an interval with endpoints $x_\al$ and $\infty$. 
So, by the 
inequality $x_\al<\infty$ in part~\eqref{it:x_al incr} of Proposition~\ref{prop:P,+}, 
 $E_\al(1)\ne\emptyset$.  
Thus, to verify part \eqref{it:J_al} of Proposition~\ref{prop:P,+}, it is enough to show that $x_\al\notin E_\al(1)$. If $x_\al=-\infty$ then this follows immediately from the definition of $E_\al(p)$ in \eqref{eq:J_al} as a subset of $\R$, and if $x_\al>-\infty$ then the same conclusion follows by \eqref{eq:P=1}. 

\eqref{it:P_al,x<x_al} Part~\eqref{it:P_al,x<x_al} of Proposition~\ref{prop:P,+} follows immediately from part~\eqref{it:J_al} of Proposition~\ref{prop:P,+} and the inequality $P_\al(X;x)\le1$, which latter in turn follows by \eqref{eq:P decr in x}. 

\eqref{it:str decr} 
Consider first the case $\al\in(0,\infty)$. 
By part\eqref{it:x ge x_*} of Proposition~\ref{prop:P} and \eqref{eq:P decr in x}, the function \eqref{eq:^-1/al} is nondecreasing, from the value $1$ 
at $-\infty$. 
It is easy to see that these conditions, together with the convexity of the function \eqref{eq:^-1/al}, 
imply that this function 
is strictly increasing on the set $\{x\in(-\infty,x_*]\cap\R\colon 1<P_\al(X;x)^{-1/\al}<\infty\}$. 
In view of part~\eqref{it:x<x_*} of Proposition~\ref{prop:P,+}, this 
implies that the function $x\mapsto P_\al(X;x)$ is strictly decreasing on the set 
$(-\infty,x_*)\cap\{x\in\R\colon P_\al(X;x)<1\}$, which is the same as $(x_\al,x_*)$, by part \eqref{it:J_al} of Proposition~\ref{prop:P,+}. 
The conclusion in part \eqref{it:str decr} of Proposition~\ref{prop:P,+} for $\al\in(0,\infty)$ now follows by its part~\eqref{it:cont}. 
The case $\al=\infty$ is quite similar, where one uses, instead of \eqref{eq:^-1/al}, the function \eqref{eq:-ln}, whose limit at $-\infty$ is $0$.  
\end{proof}

%

%
%

\begin{proof}[Proof of Proposition~\ref{prop:P cont in al}]
Let $\al$ and a sequence $(\al_n)$ be indeed as in Proposition~\ref{prop:P cont in al}. 
If $x\in[x_*,-\infty)$ then the desired conclusion $P_{\al_n}(X;x)\to P_\al(X;x)$ follows immediately from part~\eqref{it:x ge x_*} of Proposition~\ref{prop:P,+}. 
Therefore, assume in the rest of the proof of Proposition~\ref{prop:P cont in al} that 
\begin{equation}\label{eq:x<x_*}
	x\in(-\infty,x_*). 
\end{equation}
Then \eqref{eq:P,tA,lamax} takes place and, 
by \eqref{eq:lamax}, $\la_{\max,\al}$ is continuous in $\al\in(0,\infty]$. 
So, 
\begin{equation}\label{eq:la_*}
\la^*:=\sup_n\la_{\max,\al_n}\in[0,\infty) 	
\end{equation}
and 
\begin{equation}\label{eq:P,tA,la_*}
P_\ga(X;x)=\inf_{\la\in[0,\la_*]}A_\ga(X;x)(\la)\quad\text{for all}\quad
\ga\in\{\al\}\cup\{\al_n\colon n\in\N\}. 
\end{equation} 
Also, by \eqref{eq:A new}, \eqref{eq:f incr,cont}, the inequality \eqref{eq:domin} for $\al\in(0,\infty)$, the condition $X\in\XX_\be$, 
and dominated convergence, 
\begin{equation}\label{eq:A converg}
	A_{\al_n}(X;x)(\la)\to A_\al(X;x)(\la).   
\end{equation}
Hence, by \eqref{eq:P new}, $\limsup_n P_{\al_n}(X;x)\le\limsup_n A_{\al_n}(X;x)(\la)=A_\al(X;x)(\la)$ for all $\la\in[0,\infty)$, whence, again by \eqref{eq:P new}, 
\begin{equation}\label{eq:limsup,al}
	\limsup_n P_{\al_n}(X;x)\le P_\al(X;x). 
\end{equation}
So, the case $\al=0$ of Proposition~\ref{prop:P cont in al} follows by \eqref{eq:0<al}. 

If $\al\in(0,1]$ then for any $\ka$ and $\la$ such that $0\le\ka<\la<\infty$ one has 
\begin{multline}\label{eq:equi}
	|A_\al(X;x)(\la)-A_\al(X;x)(\ka)| 
\le(\la-\ka)^\al \E(X-x)_+^\al/\al^\al
	+(\la-\ka)^{\al/2}/\al^\al
	+	\P\big((x-X)_+>\tfrac1{\sqrt{\la-\ka}}\big); 
\end{multline}
this follows because  
\begin{align*}
	0\le(1+\la u/\al)_+^\al-(1+\ka u/\al)_+^\al
	&\le(\la-\ka)^\al u^\al/\al^\al\quad\text{if}\quad u\ge0, \\ 
	0\le(1+\ka u/\al)_+^\al-(1+\la u/\al)_+^\al
	&\le\min\big(1,(\la-\ka)^\al|u|^\al/\al^\al\big) \\ 	
	&\le(\la-\ka)^{\al/2}/\al^\al+\ii{|u|>\tfrac1{\sqrt{\la-\ka}}}\quad\text{if}\quad u<0.  
\end{align*}

If now $\al\in(0,1)$ 
then \big(say, by cutting off an initial segment of the sequence $(\al_n)$\big) one may assume that $\be\in(0,1)$, and then, by \eqref{eq:equi} with $\al_n$ in place of $\al$, the sequence $\big(A_{\al_n}(X;x)(\la)\big)$ is equicontinuous in $\la\in[0,\infty)$, uniformly in $n$. 
Therefore, by 
\eqref{eq:la_*} and the Arzel\`a--Ascoli theorem, the convergence in \eqref{eq:A converg} is uniform in $\la\in[0,\la^*]$ and hence the conclusion $P_{\al_n}(X;x)\to P_\al(X;x)$ follows by \eqref{eq:P,tA,la_*} -- in the case when $\al\in(0,1)$. 

Quite similarly, the same conclusion holds if $\al=1=\be$; 
that is, $P_\al(X;x)$ is left-continuous in $\al$ at the point $\al=1$ provided that $\E X_+<\infty$. 

It remains to consider the case when $\al\in[1,\infty]$ and 
$\al_n\ge1$ for all $n$. Then, by the definition in \eqref{eq:f}, the functions $h_\al$ and $h_{\al_n}$ are convex and hence, by \eqref{eq:A new}, $A_\al(X;x)(\la)$ and $A_{\al_n}(X;x)(\la)$ are convex in $\la\in[0,\infty)$. 
Then 
the conclusion $P_{\al_n}(X;x)\to P_\al(X;x)$ follows by \cite[Corollary 3]{inf-stable}, 
the condition $X\in\XX_\be$, \eqref{eq:P,tA,la_*}, and \eqref{eq:la_*}. 
\end{proof}

\begin{proof}[Proof of Proposition~\ref{prop:P cont in X}] 
This is somewhat similar to the proof of Proposition~\ref{prop:P cont in al}. 
One difference here is the use of the uniform integrability condition, which, in view of \eqref{eq:A new}, \eqref{eq:domin}, and 
the condition $X\in\XX_\al$, implies (see e.g.\ \cite[Theorem~5.4]{billingsley}) that for all $\la\in[0,\infty)$ 
\begin{equation}\label{eq:A->A}
	\lim_{n\to\infty}A_\al(X_n;x)(\la)=A_\al(X;x)(\la);   	
\end{equation}
here, in the case when $\al=\infty$ and $\la\notin\La_X$, one should also use  
the Fatou lemma for the convergence in distribution \cite[Theorem~5.3]{billingsley}, according to which one always has $\liminf_{n\to\infty}A_\al(X_n;x)(\la)\ge A_\al(X;x)(\la)$, even without the uniform integrability condition. 
In this entire proof, it is indeed assumed that $\al\in(0,\infty]$. 

It follows from \eqref{eq:A->A} and the nonnegativity of $P_\al(\cdot;\cdot)$ that 
\begin{equation}\label{eq:limsup P}
	0\le\liminf_{n\to\infty}P_\al(X_n;x)\le\limsup_{n\to\infty}P_\al(X_n;x)\le P_\al(X;x)
\end{equation}
for all real $x$; cf.\ \eqref{eq:A converg} and \eqref{eq:limsup,al}. 

The convergence \eqref{eq:P, n->infty} for $x\in(x_*,\infty)$ follows immediately from \eqref{eq:limsup P} and part~\eqref{it:x ge x_*} of Proposition~\ref{prop:P,+}. 

Using the same ingredients, it is easy to check part~\eqref{it:P,x_*} of Proposition~\ref{prop:P cont in X} as well. Indeed, 
assuming that $\P(X_n=x_*)\underset{n\to\infty}\longrightarrow\P(X=x_*)$ and using also  
\eqref{eq:0<al}, 
one has 
\begin{multline*}
	\P(X=x_*)=\liminf_{n\to\infty}\P(X_n=x_*)\le\liminf_{n\to\infty}\P(X_n\ge x_*)
	\le\liminf_{n\to\infty}P_\al(X_n;x_*) \\ 
	\le\limsup_{n\to\infty}P_\al(X_n;x_*)
	\le P_\al(X;x_*)=\P(X=x_*), 
\end{multline*}
which yields \eqref{eq:P, n->infty} for $x=x_*$. 
Also, $X_n\D{n\to\infty} X$ implies $\limsup_{n\to\infty}\P(X_n=x_*)\le\P(X=x_*)$; see e.g.\ \cite[Theorem~2.1]{billingsley}.
So, if $\P(X=x_*)=0$, then $\P(X_n=x_*)\to\P(X=x_*)$ and hence \eqref{eq:P, n->infty} holds for $x=x_*$, by the first sentence of part~\eqref{it:P,x_*} of Proposition~\ref{prop:P cont in X}. 

It remains to prove part~\eqref{it:P, n->infty} of Proposition~\ref{prop:P cont in X} assuming \eqref{eq:x<x_*}. 
The reasoning here is quite similar to the corresponding reasoning in the proof of Proposition~\ref{prop:P cont in al}, starting with \eqref{eq:x<x_*}. 
Here, instead of the continuity of $\la_{\max,\al}=\la_{\max,\al,X}$ in $\al$, one should use the convergence $\la_{\max,\al,X_n}\to\la_{\max,\al,X}$, which holds provided that $y\in(x,x_*)$ is chosen to be such that $\P(X=y)=0$. 
Concerning the use of inequality \eqref{eq:equi}, note that (i) the uniform integrability condition implies that $\E(X_n-x)_+^\al$ is bounded in $n$ and (ii) the convergence in distribution $X_n\D{n\to\infty} X$ implies that $\sup_n\P\big((x-X_n)_+>\tfrac1{\sqrt{\la-\ka}}\big)\longrightarrow0$ as $0<\la-\ka\to0$. 
Proposition~\ref{prop:P cont in X} is now completely proved. 
\end{proof}


\begin{proof}[Proof of Theorem~\ref{th:P}]
The \textbf{model-independence} 
is obvious from the definition \eqref{eq:P new}. 
The \textbf{monotonicity in $X$} follows immediately from \eqref{eq:alle}, \eqref{eq:g_al,t}, and \eqref{eq:P old}--\eqref{eq:A old}.  
The \textbf{monotonicity in $\al$} was already given in \eqref{eq:mono in al}. 
The \textbf{monotonicity in $x$} is part (i) of Proposition~\ref{prop:P}. 
That $P_\al(X;x)$ takes on only \textbf{values} in the interval $[0,1]$ follows immediately from \eqref{eq:P decr in x}. 
The \textbf{$\al$-concavity in $x$} and \textbf{stability in $x$} follow immediately from parts~\eqref{it:conc} and \eqref{it:x ge x_*} of Proposition~\ref{prop:P,+}. 
The \textbf{stability in $\al$} and the \textbf{stability in $X$} are Propositions~\ref{prop:P cont in al} and \ref{prop:P cont in X}, respectively. 
The \textbf{translation invariance}, \textbf{consistency}, and \textbf{positive homogeneity} follow immediately from the definition \eqref{eq:P new}. 
\end{proof}

\begin{proof}[Proof of Proposition~\ref{prop:inverse}] \ 

\eqref{it:Q in R} Part~\eqref{it:Q in R} of this proposition follows immediately from \eqref{eq:Q} 
and \eqref{eq:P decr in x}. 
%

\eqref{it:p le p_*} Suppose here indeed that $p\in(0,p_*]\cap(0,1)$. 
Then for any $x\in(x_*,\infty)$ one has $P_\al(X;x)=0<p$, by part~\eqref{it:x ge x_*} of Proposition~\ref{prop:P,+}, whence, by \eqref{eq:J_al}, $x\in E_\al(p)$. 
On the other hand, for any $x\in(-\infty,x_*]$ one has $P_\al(X;x)\ge P_\al(X;x_*)=p_*\ge p$, by part~(i) of Proposition~\ref{prop:P} and part~\eqref{it:x ge x_*} of Proposition~\ref{prop:P,+}, whence $x\notin E_\al(p)$. 
So, 
$E_\al(p)=(x_*,\infty)$, and the conclusion $Q_\al(X;p)=x_*$ now follows by the definition of $Q_\al(X;p)$ in \eqref{eq:Q}. 

\eqref{it:Q le x_*} 
If $x_*=\infty$ then the inequality $Q_\al(X;p)\le x_*$ in part~\eqref{it:Q le x_*} of Proposition~\ref{prop:inverse} is trivial. If $x_*<\infty$ and $p\in(p_*,1)$, then $x_*\in E_\al(p)$ and hence $Q_\al(X;p)\le x_*$ by \eqref{eq:Q}. 
%
Now part~\eqref{it:Q le x_*} of Proposition~\ref{prop:inverse} follows from its part~\eqref{it:p le p_*}. 

\eqref{it:Q to x_*} Take any $x\in(-\infty,x_*)$. Then $P_0(X;x)=\P(X\ge x)>0$. 
Moreover, for all $p\in(0,P_0(X;x))$ one has $x\notin E_{0,X}(p)$. 
Therefore and because the set $E_{0,X}(p)$ is an interval with endpoints $Q_0(X;p)$ and $\infty$, it follows that $x\le Q_0(X;p)$. 
Thus, for any given $x\in(-\infty,x_*)$ and for all small enough $p>0$ one has $Q_0(X;p)\ge x$ and hence, by the already established part~\eqref{it:Q le x_*} of Proposition~\ref{prop:inverse}, $Q_0(X;p)\in[x,x_*]$. This means that part~\eqref{it:Q to x_*} of Proposition~\ref{prop:inverse} is proved for $\al=0$. To complete the proof of this part, it remains to refer to 
the monotonicity of $Q_\al(X;p)$ in $\al$ stated in \eqref{eq:Q incr} and, again, 
to part~\eqref{it:Q le x_*} of Proposition~\ref{prop:inverse}. 

\eqref{it:p>p_*} 
Assume indeed that $\al\in(0,\infty]$. By part~\eqref{it:x_al<x_*} of Proposition~\ref{prop:P,+}, the case $p_*=1$ is equivalent to $x_\al=x_*$, and in that case both mappings \eqref{eq:inv} and \eqref{eq:dir} are empty, so that part~\eqref{it:p>p_*} of Proposition~\ref{prop:inverse} is trivial. 
So, assume that $p_*<1$ and, equivalently, $x_\al<x_*$. 
The function $(x_\al,x_*)\ni x\mapsto P_\al(X;x)$ 
is continuous and strictly decreasing, by parts~\eqref{it:cont} and \eqref{it:str decr} of Proposition~\ref{prop:P,+}.  
At that, $P_\al(X;x_*-)=
P_\al(X;x_*)=p_*$ by parts~\eqref{it:cont} and \eqref{it:x ge x_*} of Proposition~\ref{prop:P,+} if $x_*<\infty$, and $P_\al(X;x_*-)=
0=p_*$ by 
\eqref{eq:P decr in x} and \eqref{eq:x_*,p_*} 
if $x_*=\infty$. 
Also, $P_\al(X;x_\al+)=
P_\al(X;x_\al)=1$ by 
the condition $x_\al<x_*$ and 
parts~\eqref{it:cont} and \eqref{it:P_al,x<x_al} of Proposition~\ref{prop:P,+} if $x_\al>-\infty$, and $P_\al(X;x_\al+)=
1$ by 
\eqref{eq:P decr in x} if $x_\al=-\infty$. 
Therefore, the continuous and strictly decreasing function $(x_\al,x_*)\ni x\mapsto P_\al(X;x)$ maps $(x_\al,x_*)$ onto $(p_*,1)$, and so, formula \eqref{eq:dir} is correct, and there is a unique inverse function, say $(p_*,1)\ni p\mapsto x_{\al,p}\in(x_\al,x_*)$, to the function \eqref{eq:dir}; moreover, this inverse function is continuous and strictly decreasing. 
It remains to show that $Q_\al(X;p)=x_{\al,p}$ for all $p\in(p_*,1)$. 
Take indeed any $p\in(p_*,1)$. 
Since the function $(p_*,1)\ni p\mapsto x_{\al,p}\in(x_\al,x_*)$ is inverse to \eqref{eq:dir} and strictly decreasing, $P_\al(X;x_{\al,p})=p$, $P_\al(X;x)>p$ for $x\in(x_\al,x_{\al,p})$, and $P_\al(X;x)<p$ for $x\in(x_{\al,p},x_*)$. 
So, by part~(i) of Proposition~\ref{prop:P}, $P_\al(X;x)>p$ for $x\in(-\infty,x_{\al,p})$ and $P_\al(X;x)<p$ for $x\in(x_{\al,p},\infty)$. 
Now the conclusion that $Q_\al(X;p)=x_{\al,p}$ for all $p\in(p_*,1)$ follows by \eqref{eq:Q}. 

\eqref{it:y} Assume indeed that $\al\in(0,\infty]$ and take indeed any $y\in\big(-\infty,Q_\al(X;p)\big)$. 
If $P_\al(X;y)=1$ then the conclusion $P_\al(X;y)>p$ in part~\eqref{it:y} of Proposition~\ref{prop:inverse} is trivial, in view of \eqref{eq:p}. 
So, w.l.o.g.\ $P_\al(X;y)<1$ and hence $y\in E_\al(1)=(x_\al,\infty)$, by \eqref{eq:J_al} 
and part~\eqref{it:J_al} of Proposition~\ref{prop:P,+}. 
Let now $y_p:=Q_\al(X;p)$ for brevity, so that $y\in(-\infty,y_p)$ and, by the already verified part~\eqref{it:Q le x_*} of Proposition~\ref{prop:inverse}, $y_p\le x_*$. 
Therefore, 
$x_\al<y<y_p\le x_*$. 
So, by part~\eqref{it:p>p_*} of Proposition~\ref{prop:inverse} and parts~\eqref{it:cont} and \eqref{it:x ge x_*} of Proposition~\ref{prop:P,+}, 
\begin{equation}\label{eq:P>}
P_\al(X;y)>\lim_{x\uparrow y_p}P_\al(X;x)=P_\al(X;y_p)\ge P_\al(X;x_*)=p_*,  	
\end{equation}
which yields the conclusion $P_\al(X;y)>p$ in the case when $p\le p_*$. 
If now $p>p_*$ then $p\in(p_*,1)$ and, by part~\eqref{it:p>p_*} of Proposition~\ref{prop:inverse}, $y_p=Q_\al(X;p)\in(x_\al,x_*)$ and $P_\al(X;y_p)=p$, so that the conclusion $P_\al(X;y)>p$ follows by \eqref{eq:P>} in this case as well. 

\eqref{it:>EX} Part~\eqref{it:>EX} of Proposition~\ref{prop:inverse} follows immediately from \eqref{eq:inv}, \eqref{eq:Q decr in p}, and part \eqref{it:EX} of Proposition~\ref{prop:P,+}. 
\end{proof}

\begin{proof}[Proof of Theorem~\ref{th:coher}]
The \textbf{model-independence}, \textbf{monotonicity in $X$}, \textbf{monotonicity in $\al$}, \textbf{translation invariance}, \textbf{consistency}, and \textbf{positive homogeneity} properties of $Q_\al(X;p)$ 
follow immediately from \eqref{eq:Q} and the corresponding properties of 
$P_\al(X;x)$ stated in Theorem~\ref{th:P}. 

Concerning the \textbf{monotonicity of $Q_\al(X;p)$ in $p$}: that $Q_\al(X;p)$ is nondecreasing in $p\in(0,1)$ follows immediately from \eqref{eq:Q_0} for $\al=0$ and from 
\eqref{eq:Q=inf} and \eqref{eq:B} for $\al\in(0,\infty]$. 
That $Q_\al(X;p)$ is strictly decreasing in $p\in[p_*,1)\cap(0,1)$ if $\al\in(0,\infty]$ follows immediately from part~\eqref{it:p>p_*} of 
Proposition~\ref{prop:inverse} and the verified below statement on the stability in $p$:  $Q_\al(X;p)$ is continuous in $p\in(0,1)$ if $\al\in(0,\infty]$. 

The \textbf{monotonicity of $Q_\al(X;p)$ in $\al$} follows immediately from \eqref{eq:mono in al} and \eqref{eq:Q}. 

The \textbf{finiteness of $Q_\al(X;p)$} was already stated in part~\eqref{it:Q in R} of Proposition~\ref{prop:inverse}. 

The \textbf{concavity of $Q_\al(X;p)$ in $p^{-1/\al}$} in the case when $\al\in(0,\infty)$ follows by \eqref{eq:Q=inf}, since $B_\al(X;p)(t)$ is affine (and hence concave) in $p^{-1/\al}$. 
Similarly, the \textbf{concavity of $Q_\infty(X;p)$ in $\ln\frac1p$} follows by \eqref{eq:Q=inf}, since $B_\infty(X;p)(t)$ is affine in $\ln\frac1p$. 

The \textbf{stability 
of $Q_\al(X;p)$ in $p$} 
can be deduced from Proposition~\ref{prop:inverse}. 
Alternatively, the same follows from 
the already established finiteness and concavity of $Q_\al(X;p)$ in $p^{-1/\al}$ or $\ln\frac1p$ (cf.\ the proof of \cite[Proposition~13]{rocka-ur02}), because any finite concave function on an open interval of the real line is continuous, whereas the mappings $(0,1)\ni p\mapsto p^{-1/\al}\in(0,\infty)$ and $(0,1)\ni p\mapsto\ln\frac1p\in(0,\infty)$ are 
homeomorphisms. 

Concerning the \textbf{stability 
of $Q_\al(X;p)$ in $X$}, take any real $x\ne x_*$. Then the convergence $P_\al(X_n;x)\to P_\al(X;x)$ holds, by 
Proposition~\ref{prop:P cont in X}. 
So, in view of \eqref{eq:J_al}, if $x\in E_{\al,X}(p)$ then eventually (that is, for all large enough $n$) $x\in E_{\al,X_n}(p)$. Hence, by \eqref{eq:Q}, for each real $x\ne x_*$ such that $x>Q_\al(X;p)$ eventually one has $x\ge Q_\al(X_n;p)$. It follows that 
$
	\limsup_n Q_\al(X_n;p)\le Q_\al(X;p). 
$ 
On the other hand, by part~\eqref{it:y} of Proposition~\ref{prop:inverse}, for any $y\in\big(-\infty,Q_\al(X;p)\big)$ one has $P_\al(X;y)>p$ and hence eventually $P_\al(X_n;y)>p$, which yields $y\notin E_{\al,X_n}(p)$ and hence $y\le Q_\al(X_n;p)$.  
It follows that $\liminf_n Q_\al(X_n;p)\ge Q_\al(X;p)$. Recalling now the established inequality $\limsup_n Q_\al(X_n;p)\le Q_\al(X;p)$, one completes the verification of the stability 
of $Q_\al(X;p)$ in $X$. 

The \textbf{stability 
of $Q_\al(X;p)$ in $\al$} is proved quite similarly, only using Proposition~\ref{prop:P cont in al} in place of Proposition~\ref{prop:P cont in X}. 
Here the stipulation $x\ne x_*$ is not needed. 

%

Consider now the \textbf{sensitivity} property. 
First, suppose that $\al\in(0,1)$. 
Then, for all real $t<0$, the derivative of $B_\al(X;p)(t)$ in $t$ is less than 
$D:=1-(\E Y^\al)^{-1+1/\al}\E Y^{\al-1}$, where $Y:=(X-t)_+=X-t>0$. 
The inequality $D\le0$ can be rewritten as the true inequality  $\frac\tau{\tau+1}L(-1)+\frac1{\tau+1}L(\tau)\ge L(0)$ for the convex function $s\mapsto L(s):=\ln\E\exp\{(1-\al)s\ln Y\}$, where $\tau:=\frac\al{1-\al}$. 
So, the derivative is negative and hence $B_\al(X;p)(t)$ decreases in $t\le0$ (here, to include $t=0$, we also used the continuity of $B_\al(X;p)(t)$ in $t$, which follows by 
the condition $X\in\XX_\al$ and 
dominated convergence
). On the other hand, if $t>0$ then $B_\al(X;p)(t)\ge t>0$. Also, $B_\al(X;p)(0)>0$ by \eqref{eq:B} if the condition $\P(X>0)>0$ holds. Recalling again the continuity of $B_\al(X;p)(t)$ in $t$, one completes the verification of the sensitivity property -- in the case $\al\in(0,1)$. \\ 
The sensitivity property in the case $\al=1$ follows by \eqref{eq:al=1}. Indeed, \eqref{eq:al=1} yields $Q_1(X;p)\ge Q(X;p)>0$ if $Q(X;p)>0$, and 
$Q_1(X;p)=\tfrac1p\,\E X\ge0$ by the condition $X\ge0$ if $Q(X;p)=0$; 
moreover, one has $\E X>0$ and hence $Q_1(X;p)=\tfrac1p\,\E X>0$ if $Q(X;p)=0$ and $\P(X>0)>0$. 
On the other hand, by \eqref{eq:Q_0}, $X\ge0$ implies $Q(X;p)\ge0$. 
Thus, the sensitivity property in the case $\al=1$ is verified is well. 
This and the already established monotonicity of $Q_\al(X;p)$ in $\al$ implies the sensitivity property whenever $\al\in[1,\infty]$. \\ 
As far as this property is concerned, it remains to verify it when $\al=0$ -- assuming that $\P(X>0)>p$. The sets $E:=\big\{x\in\R\colon\P(X>x)\le p\big\}$ and $E^\circ:=\big\{x\in\R\colon\P(X>x)<p\big\}$ are intervals with the right endpoint $\infty$. 
The condition $\P(X>0)>p$ means that $0\notin E$. 
By the right continuity of $\P(X>x)$ in $x$, 
the set $E$ contains the closure $\overline{E^\circ}$ of the set $E^\circ$. 
So, $0\notin\overline{E^\circ}$ and hence $0<\inf E^\circ=Q_0(X;p)$, by \eqref{eq:Q_0}. 
Thus, the sensitivity property is fully verified.  


In the presence of the positive homogeneity, the \textbf{subadditivity} property is easy to see to be equivalent to the convexity; cf.\ e.g. \cite[Theorem~4.7]{rocka}. 

Therefore, it remains to verify the \textbf{convexity} property. 
Assume indeed that $\al\in[1,\infty]$. 
If at that $\al<\infty$, then the function $\|\cdot\|_\al$ is a norm and hence convex; moreover, this function is nondecreasing on the set of all nonnegative r.v.'s. On the other hand, the function $\R\ni x\mapsto x_+$ is nonnegative and convex. It follows by \eqref{eq:B} that 
$B_\al(X;p)(t)$ is convex in the pair $(X,t)$. 
So, to complete the verification of the convexity property of $Q_\al(X;p)$ in the case $\al\in[1,\infty)$, it remains to refer to the well-known and easily established fact that, if $f(x,y)$ is convex in $(x,y)$, then $\inf_y \,f(x,y)$ is convex in $x$; cf.\ e.g.\ \cite[Theorem~5.7]{rocka}. \\ 
The subadditivity and hence 
convexity of $Q_\al(X;p)$ in $X$ in the remaining case $\al=\infty$ 
can now be obtained 
by the already established stability in $\al$. 
It can also be deduced 
from \cite[Lemma~B.2]{rio.transl} (cf.\ \cite[Lemma~2.1]{rio}) or 
from by the main result in \cite{pin-yt}, in view of the inequality  
$\li{\big(L_{X_1+\dots+X_n}\big)}\le\li{\big(L_{X_1}\fJ\cdots\fJ L_{X_n}\big)}$ 
given in the course of the discussion following 
\cite[Corollary~2.2]{pin-yt} therein. 
However, a direct proof, similar to the one above for $\al\in[1,\infty)$, can be based on the observation that $B_\infty(X;p)(t)$ is convex in the pair $(X,t)$. 
Since $t\,\ln\frac1p$ is obviously linear in $(X,t)$, the convexity of $B_\infty(X;p)(t)$ in $(X,t)$ means precisely that 
for any natural number $n$, any r.v.'s $X_1,\dots,X_n$, any positive real numbers $t_1,\dots,t_n$, and any positive real numbers $\al_1,\dots,\al_n$ with $\sum_i\al_i=1$, one has the inequality 
$
	t\ln\E e^{X/t}\le\sum_i\al_i t_i\ln\E e^{X_i/t_i}
$
, where 
$X:=\sum_i\al_i X_i$ and $t:=\sum_i\al_i t_i$; but the latter inequality can be rewritten as an instance of H\"older's inequality: 
$\E\prod_i Z_i\le\prod_i\|Z_i\|_{p_i}$, where $Z_i:=e^{\al_i X_i/t}$ and $p_i:=t/(\al_i t_i)$ (so that $\sum_i\frac1{p_i}=1$). 
\big(In particular, it follows that $B_\infty(X;p)(t)$ is convex in $t$, which is useful when $Q_\infty(X;p)$ is computed by formula \eqref{eq:Q=inf}.\big) 
 
The proof of Theorem~\ref{th:coher} is now complete. 
\end{proof}

\begin{proof}[Proof of Proposition~\ref{cor:str mono X}] 
%
Consider first the case $\al\in(0,\infty)$. 
Let r.v.'s $X$ and $Y$ be in the default domain of definition, $\XX_\al$, of the functional $Q_\al(\cdot;p)$. 
The condition $X\sst Y$ and the left continuity of the function $\P(X\ge\cdot)$ imply that 
for any $v\in\R$ there are some $u\in(v,\infty)$ and $w\in(v,u)$ such that $\P(X\ge z)<\P(Y\ge z)$ for all $z\in[w,u]$. On the other hand, by the Fubini theorem, $\E(X-t)_+^\al=\int_\R\al(z-t)_+^{\al-1}\P(X\ge z)\dd z$ for all $t\in\R$. 
Recalling also 
that $X$ and $Y$ are in $\XX_\al$, one has $B_\al(X;p)(t)<B_\al(Y;p)(t)$ for all $t\in\R$. 
By Proposition~\ref{prop:att}, $Q_\al(Y;p)=B_\al(Y;p)(t_\opt)$ for some $t_\opt\in\R$. 
So, $Q_\al(X;p)\le B_\al(X;p)(t_\opt)<B_\al(Y;p)(t_\opt)=Q_\al(Y;p)$. 
\big(Note that the proof of Proposition~\ref{prop:att}, given later in this appendix, does not use Proposition~\ref{cor:str mono X} -- so that there is no vicious circle here.\big)

Concerning the case $\al=\infty$, recall \eqref{eq:x_*,p_*} and \eqref{eq:La_X}, and then note that the condition $X\sst Y$ implies that $x_{*,Y}=\infty$, $\La_X\supseteq\La_Y$, and $B_\infty(X;p)(t)<B_\infty(Y;p)(t)$ for all $t\in(0,\infty)$ such that $\frac1t\in\La_X$ and hence for all $t\in(0,\infty)$ such that $\frac1t\in\La_Y$. 
Here, instead of the formula $\E(X-t)_+^\al=\int_\R\al(z-t)_+^{\al-1}\P(X\ge z)\dd z$ for all $t\in\R$, one uses the formula $\E e^{(X-x)/t}=\int_\R \frac1t\,e^{(z-x)/t}\P(X\ge z)\dd z$ for all $t\in(0,\infty)$. 
Using now Proposition~\ref{prop:att}, one sees that $Q_\infty(Y;p)=B_\infty(Y;p)(t_\opt)$ for some $t_\opt\in(0,\infty)$ such that $\frac1t\in\La_Y$. 
So, $Q_\infty(X;p)\le B_\infty(X;p)(t_\opt)<B_\infty(Y;p)(t_\opt)=Q_\infty(Y;p)$. 
%
\end{proof}

\begin{proof}[Proof of Proposition~\ref{prop:al<1}]
Suppose that indeed $\al\in
[0,1)$. 
Let $X$ and $Y$ be independent r.v.'s, each with the Pareto density function given by the formula $f(u)=(1+u)^{-2}\ii{u>0}$, so that $\P(X\ge x)=\P(Y\ge x)=(1+x_+)^{-1}$ for all $x\in\R$. 
Then, by the condition $\al\in
[0,1)$, 
the condition $X\in\XX_\al$ (assumed by default in this paper and, in particular, in Proposition~\ref{cor:str mono X}) holds; this is the only place in the proof of Proposition~\ref{prop:al<1} where the condition $\al<1$ is used. 
Also, then it is not hard to see that for all $x\in(0,\infty)$ one has $\P(X+Y\ge x)-\P(2X\ge x)=2(2 + x)^{-2}\ln(1 + x)>0$ 
and hence, by the definition of the relation $\sst$ given in Proposition~\ref{cor:str mono X},  
\begin{equation*}
	2X\sst X+Y.  
\end{equation*}
Using now Proposition~\ref{cor:str mono X} together with the positive homogeneity property stated in Theorem~\ref{th:coher}, 
one concludes that $Q_\al(X+Y;p)>Q_\al(2X;p)=2Q_\al(X;p)=Q_\al(X;p)+Q_\al(Y;p)$ if $\al\in(0,1)$. 

It remains to consider the case $\al=0$. 
Note that the function $(0,\infty)\ni x\mapsto\P(X+Y\ge x)\in(0,1)$ is decreasing strictly and continuously from $1$ to $0$. 
So, in view of \eqref{eq:Q_0}, the function $(0,1)\ni p\mapsto Q(X+Y;p)\in(0,\infty)$ is the inverse to the function $(0,\infty)\ni x\mapsto\P(X+Y\ge x)\in(0,1)$. 
Similarly, the function $(0,1)\ni p\mapsto Q(2X;p)\in(0,\infty)$ is the inverse to the strictly decreasing continuous function $(0,\infty)\ni x\mapsto\P(2X\ge x)\in(0,1)$. 
Since $\P(X+Y\ge x)>\P(2X\ge x)$ for all $x\in(0,\infty)$, it follows that $Q(X+Y;p)>Q(2X;p)$ and thus the inequality $Q_\al(X+Y;p)>Q_\al(X;p)+Q_\al(Y;p)$ holds for $\al=0$ as well. 
\end{proof}

\begin{proof}[Proof of Proposition~\ref{prop:Q close}]\ 

(i) The equalities in \eqref{eq:x_*<infty} follow immediately from part \eqref{it:Q to x_*} of Proposition~\ref{prop:inverse}. The condition $x_*\in\R$ in \eqref{eq:x_*<infty} follows from the condition $x_*<\infty$ -- because, by the definition of $x_*$ in \eqref{eq:x_*,p_*}, one always has $x_*\in(-\infty,\infty]$. Thus, part (i) of Proposition~\ref{prop:Q close} is verified. 

(ii) Take any $r\in(\al,\infty]$ and suppose that indeed $\P(X\ge x)$ is like  $x^{-r}$. 
Then, in view of \eqref{like-x-r-eq} and because the function $q_0$ was supposed to be positive on $\R$, one observes that $\P(X\ge x)>0$ for all large enough real $x$. Therefore and because $\P(X\ge x)$ is nondecreasing in $x\in\R$, in fact $\P(X\ge x)>0$ for all real $x$.  
In particular, it now follows that indeed $x_*=\infty$. Moreover, recalling the definition 
\eqref{eq:J_al} of $E_{\al,X}(p)$ and the equality in \eqref{eq:0<al}, one sees that for any real $x$ and all $p$ in the (nonempty) right neighborhood $\big(0,P_0(X;x)\big)=\big(0,\P(X\ge x)\big)$ of $0$, one has $x\notin E_{0,X}(p)$; therefore and because, by the definition \eqref{eq:Q} of $Q_\al(X;p)$, the set $E_{0,X}(p)$ is an interval with endpoints $Q_0(X;p)$ and $\infty$, one concludes that $Q_0(X;p)\ge x$ for all $p\in\big(0,P_0(X;x)\big)$. 
Thus, $Q_\al(X;p)\underset{p\downarrow0}\longrightarrow\infty$ for $\al=0$; that the same limit relation holds for any $\al\in[0,\infty]$ now follows immediately by the monotonicity of $Q_\al(X;p)$ in $\al$, as stated in \eqref{eq:Q incr}. 

To complete the proof of Proposition~\ref{prop:Q close}, it remains to verify \eqref{eq:Q close}. 
First here, 
consider the case $r<\infty$, so that $r\in(\al,\infty)$. 
For brevity, let 
\begin{equation*}
q(x):=\P(X\ge x). 	
\end{equation*}
Then 
\begin{equation}\label{eq:rat asymp}
	\tfrac{q(x)}{q(y)}\sim\big(\tfrac yx\big)^{r+o(1)}\quad\text{as}\quad x,y\to\infty;  
\end{equation}
the latter asymptotic relation is an extension of, and proved quite similarly to, the asymptotic relation (\ref{like-x-r-eq}b). 
Introduce also 
\begin{equation*}
	x_{\al,p}^\pm:=Q_\al(X;p)\pm1. 
\end{equation*} 
Let indeed $p\downarrow0$, as in \eqref{eq:Q close}. 
Then 
\begin{equation}\label{eq:x sim Q}
	x_{\al,p}^\pm\sim Q_\al(X;p)\to\infty. 
\end{equation}
Because the set $E_{\al,X}(p)$ is an interval with endpoints $Q_\al(X;p)$ and $\infty$, one has $x_{\al,p}^+\in E_{\al,X}(p)$ and $x_{\al,p}^-\notin E_{\al,X}(p)$, whence 
\begin{equation}\label{eq:<p<}
P_\al(X;x_{\al,p}^+)<p\le P_\al(X;x_{\al,p}^-). 	
\end{equation}
On the other hand, by \cite{pin98} \big(see Corollary~2.3, duality relation (4), Theorem~4.2, and Remark~4.3 there\big), 
\begin{equation}\label{eq:P sim cq}
	P_\al(X;x)\sim c_{r,\al}\, q(x) \quad\text{as}\quad x\to\infty;  
\end{equation}
note that the condition ``$\P(X\ge x)$ is like  $x^{-r}$'' in part~(ii) of Proposition~\ref{prop:Q close} corresponds to the condition ``$q(x)/q_0(x)\to1$ as $x\to\infty$ for some $q_0$ which is like $x^{-r}$'' in \cite[Remark~4.3]{pin98}, because the notion ``like $x^{-r}$'' is defined in the present paper slightly differently from \cite{pin98}. 
Combining \eqref{eq:<p<} and \eqref{eq:P sim cq}, one has  
\begin{equation}\label{eq:<p<,sim}
c_{r,\al} q(x_{\al,p}^+)\lesssim p\lesssim c_{r,\al} q(x_{\al,p}^-);  	
\end{equation}
here and elsewhere, $a(p)\lesssim b(p)$ or, equivalently, $b(p)\gtrsim a(p)$ means, by definition, that $b(p)\sim a(p)(1+d(p))>0$ for some nonnegative function $d$. 
Also, \eqref{eq:<p<} with $\al=0$ can be written as 
\begin{equation*}
q(x_{0,p}^+)<p\le q(x_{0,p}^-).   	
\end{equation*} 
Comparing this with \eqref{eq:<p<,sim} and recalling \eqref{eq:rat asymp}, one sees that 
\begin{equation*}
	c_{r,\al}\lesssim\frac{q(x_{0,p}^-)}{q(x_{\al,p}^+)}
	\sim\Big(\frac{x_{\al,p}^+}{x_{0,p}^-}\Big)^{r+o(1)}.  
\end{equation*}
Therefore and because of \eqref{eq:x sim Q},  
\begin{equation*}
	\frac{Q_\al(X;p)}{Q_0(X;p)}\sim\frac{x_{\al,p}^+}{x_{0,p}^-}\gtrsim c_{r,\al}^{1/r}=K(r,\al), 
\end{equation*}
so that $\frac{Q_\al(X;p)}{Q_0(X;p)}\gtrsim K(r,\al)$. 
Quite similarly, $\frac{Q_\al(X;p)}{Q_0(X;p)}\lesssim K(r,\al)$, which shows that indeed \eqref{eq:Q close} holds -- in the case $r<\infty$. 

The case $r=\infty$ is similar. The main differences here are that (a) instead of \eqref{eq:rat asymp}, one should now use the asymptotic relation $\tfrac{q(x)}{q(y)}\sim\big(\tfrac yx\big)^\rho$ as $x,y\to\infty$, with some $\rho=\rho(x,y)\to\infty$, and (b) \eqref{eq:P sim cq} holds for $\al=\infty$ with $c_{r,\infty}:=\Ga(\al+1)(e/\al)^\al$.
\end{proof}

\begin{proof}[Proof of Proposition~\ref{prop:two min,P}]
Take indeed any $\al\in(0,1)$ and $p\in(0,1)$. 
Note that there are real numbers $q$, $r$, and $b$ such that 
\begin{equation}
	\begin{gathered}\label{eq:q,r,b}
	q>0,\ r>0,\ q+r<1, \\ 
	0<b<1, \\ 
	q(1-b)^\al+r(1+b)^\al=2^\al r=p. 
	\end{gathered}
\end{equation}
Indeed, if $0<b<1$, $r=\frac p{2^\al}$, and $q=k(b)r$, where $k(b):=\frac{2^\al-(1+b)^\al}{(1-b)^\al}$, then all of the conditions in \eqref{eq:q,r,b} will be satisfied, possibly except the condition $q+r<1$, which latter will be then equivalent to the condition $h(b):=\frac p{2^\al}\,(1+k(b))<1$. However, this condition can be satisfied by letting $b\in(0,1)$ be small enough -- because $h(0+)= 
p\in(0,1)$. 

If now $q$, $r$, and $b$ satisfy \eqref{eq:q,r,b}, then there is a r.v.\ $X$ taking values $-1$, $-b$, and $b$ with probabilities $1-q-r$, $q$, and $r$, respectively. Let indeed $X$ be such a r.v. Then for all $s\in(0,\infty)$ 
\begin{equation}\label{eq:A=g}
	A_\al(X;0)(\al s)=g(s):=(1-q-r)(1-s)_+^\al+q(1-bs)_+^\al+r(1+bs)^\al. 
\end{equation}
In view of \eqref{eq:A=g} and \eqref{eq:q,r,b}, 
\begin{equation*}
	g(0+)=1>p=g(\tfrac1b)=g(1)<\infty=g(\infty-). 
\end{equation*}
Moreover, by the condition $\al\in(0,1)$, the function $g$ is strictly concave on each of the intervals $(0,1]$, $[1,\tfrac1b]$, and $[\tfrac1b,\infty)$. 
So, the minimum of $g(s)$ in $s\in(0,\infty)$ equals $p$ and is attained precisely at two distinct positive values of $s$. 
Thus, in the case $x=0$, Proposition~\ref{prop:two min,P} follows by \eqref{eq:A=g}. 
The case of a general $x\in\R$ immediately reduces to that of $x=0$ by using the shifted r.v.\ $X+x$ in place of $X$.  
\end{proof}

\begin{proof}[Proof of Proposition~\ref{prop:att}]
Consider first part (i) of the proposition. 
For any real $t>t_{\max}$ one has 
$B_\al(X;p)(t)\ge t>B_\al(X;p)(s)\ge\inf_{t\in\R}B_\al(X;p)(t)$. 
On the other hand, by \eqref{eq:t_0,t_1}, 
for all real $t\le t_0:=t_{0,\min}$ one has $\|(X-t)_+\|_\al^\al\ge\E(X-t)^\al\ii{X\ge t_0}
\ge(t_0-t)^\al\P(X\ge t_0)\ge(t_0-t)^\al\tp$, whence 
$B_\al(X;p)(t)\ge t+(t_0-t)(\tp/p)^{1/\al}>t_{\max}=B_\al(X;p)(s)\ge\inf_{t\in\R}B_\al(X;p)(t)$ provided that also $t<t_{1,\min}$. 
Thus, $B_\al(X;p)(t)>\inf_{t\in\R}B_\al(X;p)(t)$ if either $t>t_{\max}$ or $t<t_{0,\min}\wedge t_{1,\min}=t_{\min}$. 
This, together with the continuity of $B_\al(X;p)(t)$ in $t$, 
completes the proof of part (i) of Proposition~\ref{prop:att}. 

Concerning part (ii) of the proposition, consider first 
 
\emph{Case 1: $x_*=\infty$.} Take then any real $t_1>0$ such that $\E e^{X/t_1}<\infty$ and then any real $x>x_1:=B_\infty(X;p)(t_1)$ such that 
$q:=\P(X\ge x)<p$; note that $q>0$, since $x_*=\infty$. 
Then for any real $t>0$ one has $\E e^{X/t}\ge qe^{x/t}$ and hence 
\begin{equation}\label{eq:cs1}
	B_\infty(X;p)(t)=t\,\ln\frac{\E e^{X/t}}p
	\ge t\,\ln\frac{qe^{x/t}}p
	=x-t\ln\frac pq>x_1=B_\infty(X;p)(t_1)\ge \inf_{t>0}B_\al(X;p)(t)
\end{equation}
provided that 
\begin{equation*}
	t<t_{\min}:=\frac{x-x_1}{\ln(p/q)};
\end{equation*}
the latter inequality is in fact equivalent to the strict inequality in \eqref{eq:cs1}; recall here also that $x>x_1$ and $0<q<p$, whence $t_{\min}\in(0,\infty)$. 
Taking now into account that $B_\infty(X;p)(t)$ is lower semi-continuous in $t$ (by Fatou's lemma) and $B_\infty(X;p)(t)=t\,\ln\frac{\E e^{X/t}}p\sim t\,\ln\frac1p\to\infty$ as $t\to\infty$, one concludes that  
\begin{equation*}
	\inf_{t>0}B_\infty(X;p)(t)=\inf_{t\ge t_{\min}}B_\infty(X;p)(t)=\min_{t\ge t_{\min}}B_\infty(X;p)(t),  
\end{equation*}
which completes the consideration of Case 1 for part (ii) of the proposition. 
It remains to consider 
 
\emph{Case 2: $x_*<\infty$.} Note that $B_\infty(\cdot;p)(t)$ is translation invariant in the sense that 
$B_\infty(X+c;p)(t)=B_\infty(X;p)(t)+c$ for all $c\in\R$ and $t\in(0,\infty)$.  
Therefore, 
%
without loss of generality $x_*=0$, so that $X\le0$ a.s.\ and $\P(X\ge-\vp)>0$ for all real $\vp>0$. 
Now, by dominated convergence, 
$
\E e^{X/t}\underset{t\downarrow0}\longrightarrow\P(X=0)=p_*     	
$ 
and $\E e^{X/t}\underset{t\to\infty}\longrightarrow1$, 
whence 
\begin{equation}\label{eq:ln->}
	\ln\frac{\E e^{X/t}}p\longrightarrow
	\left\{
	\begin{aligned}
	\ln\frac{p_*}p&\text{ as }t\downarrow0, \\
	\ln\frac1p&\text{ as }t\to\infty.    	
	\end{aligned}
	\right.
\end{equation}
Moreover,  
\begin{equation}\label{eq:B->}
	B_\infty(X;p)(t)=t\,\ln\frac{\E e^{X/t}}p\longrightarrow
	\left\{
	\begin{aligned}
	0&\text{ as }t\downarrow0, \\
	\infty&\text{ as }t\to\infty.  	
	\end{aligned}
	\right.
\end{equation}
Indeed, if $p_*=0$ then 
for each real $\vp>0$ and all small enough real $t>0$, one has $\E e^{X/t}<p$ and hence 
$0>t\,\ln\frac{\E e^{X/t}}p
\ge t\,\ln\big(\frac1p\,\E e^{X/t}\ii{X\ge-\vp}\big)
\ge-\vp+t\ln\P(X\ge-\vp)\underset{t\downarrow0}\longrightarrow-\vp$, which yields \eqref{eq:B->} for $t\downarrow0$, in the case when $p_*=0$. As for the cases when $t\to\infty$, or $t\downarrow0$ and $p_*>0$, then \eqref{eq:B->} follows from \eqref{eq:ln->} because $0<p<1$. 

To proceed further with the consideration of Case~2, one needs to distinguish the following three subcases. 

\emph{Subcase 2.1: $p_*\in[0,p)$.} Then, by \eqref{eq:B->}, for all large enough real $t>0$ 
\begin{equation*}
	B_\infty(X;p)(t)>0=\lim_{t\downarrow0}B_\infty(X;p)(t)\ge\inf_{t>0}B_\infty(X;p)(t)
\end{equation*}
and, by \eqref{eq:B->} and \eqref{eq:ln->}, for all small enough real $s>0$ 
\begin{equation*}
	\lim_{t\downarrow0}B_\infty(X;p)(t)=0	
	>s\,\ln\frac{\E e^{X/s}}p
	=B_\infty(X;p)(s)
	\ge\inf_{t>0}B_\infty(X;p)(t). 
\end{equation*}
It follows that for some positive real $t_{\min}$ and $t_{\max}$
\begin{equation*}
	\inf_{t>0}B_\infty(X;p)(t)=\inf_{t_{\min}\le t\le t_{\max}}B_\infty(X;p)(t)
	=\min_{t_{\min}\le t\le t_{\max}}B_\infty(X;p)(t);  
\end{equation*}
the latter equality here follows by the continuity of $B_\infty(X;p)(t)$ in $t\in(0,\infty)$, which in turn takes place by the Case~2 condition $x_*<\infty$. 
This completes the consideration of Subcase 2.1 for part (ii) of the proposition. 

\emph{Subcase 2.2: $p_*\in[p,1)$.} Here, note that $\P(X<0)>0$ (since $p_*<1$) and 
$\E e^{X/t}=p_*+\E e^{X/t}\ii{X<0}$. 
So, if $t$ is decreasing from $\infty$ to $0$, then $\E e^{X/t}$ is strictly decreasing and hence $\ln\frac{\E e^{X/t}}p$ is strictly decreasing -- to $\ln\frac{p_*}p\ge0$, by \eqref{eq:ln->} and the case condition $p_*\in[p,1)$. Therefore, $\ln\frac{\E e^{X/t}}p>0$ for all $t>0$ and hence $B_\infty(X;p)(t)=t\ln\frac{\E e^{X/t}}p$ is strictly decreasing if $t$ is decreasing from $\infty$ to $0$. 
It follows that, in Subcase 2.2, 
$\inf_{t\in T_\al}=\inf_{t\in(0,\infty)}$ in \eqref{eq:Q=inf} is not attained; rather, 
$
	\inf_{t>0}B_\infty(X;p)(t)=\lim_{t\downarrow0}B_\infty(X;p)(t)=0=x_* 
$, 
in view of \eqref{eq:B->} and the assumption $x_*=0$. 
It remains to consider 

\emph{Subcase 2.3: $p_*=1$.} Then $\P(X=0)=1$ and hence $B_\infty(X;p)(t)=t\ln\frac1p$, so that, as in Subcase 2.2, $\inf_{t\in T_\al}=\inf_{t\in(0,\infty)}$ in \eqref{eq:Q=inf} is not attained, and  
$
	\inf_{t>0}B_\infty(X;p)(t)=\lim_{t\downarrow0}B_\infty(X;p)(t)=0=x_* 
$. 

Now Proposition~\ref{prop:att} is completely proved. 
\end{proof}

\begin{proof}[Proof of Proposition~\ref{prop:B str conv}]\ 
 
\eqref{B conv} Part~\eqref{B conv} of Proposition~\ref{prop:B str conv} follows because, as shown in the proof of Theorem~\ref{th:coher}, $B_\al(X;p)(t)$ is convex in the pair $(X,t)$. 

\eqref{B str conv} Assume indeed that $\al\in(1,\infty)$. It is then well known that the norm $\|\cdot\|_\al$ is strictly convex, in the sense that $\|(1-s)Y+sZ\|_\al<(1-s)\|Y\|_\al+s\|Z\|_\al$ for all $s\in(0,1)$ and all r.v.'s $Y$ and $Z$ such that $\|Y\|_\al+\|Z\|_\al<\infty$ and $\P(yY+zZ\ne0)>0$ for all nonzero real $y$ and $z$. 
The strict convexity of the norm $\|\cdot\|_\al$ is of course equivalent to its strict subadditivity -- see e.g.\ 
\cite[Corollary, page~405]{clarkson}. 
Alternatively, the strict subadditivity of the norm $\|\cdot\|_\al$ can be easily discerned from a proof of Minkowski's inequality, say the classical proof based on H\"older's inequality, or the one given in \cite
{pin-yt}. 
Since for any $t\in(-\infty,x_{**})$ the set $\supp\big((X-t)_+\big)$ contains at least two distinct points, it follows that $B_\al(X;p)(t)$ is strictly convex in $t\in(-\infty,x_{**})$ and hence, by continuity, in $t\in(-\infty,x_{**}]\cap\R$. 

\eqref{B str conv,infty} Part~\eqref{B str conv,infty} of Proposition~\ref{prop:B str conv} can be verified by invoking, in the proof of the subadditivity/convexity in Theorem~\ref{th:coher}, the well-known strictness condition for H\"older's inequality. 
\end{proof}

\begin{proof}[Proof of Proposition~\ref{prop:ql}]\ 

\eqref{ql,p<p*} In the case $\al=1$, part~\eqref{ql,p<p*} of Proposition~\ref{prop:ql} follows immediately from \eqref{eq:qlz=Q0} and part~\eqref{it:p le p_*} of Proposition~\ref{prop:inverse}. 
So, assume that $\al\in(1,\infty)$. 
Also, indeed assume that $p\in(0,p_*]\cap(0,1)$. Then necessarily $p_*>0$, $x_*<\infty$, and, again by part~\eqref{it:p le p_*} of Proposition~\ref{prop:inverse}, $Q_\al(X;p)=x_*$. 
On the other hand, $B_\al(X;p)(x_*)=x_*$ by \eqref{eq:B}. So, by \eqref{eq:amin}, $x_*\in\amin$. Also, again by \eqref{eq:B}, $B_\al(X;p)(t)=t>x_*=Q_\al(X;p)$ for all $t\in(x_*,\infty)$. Thus, by \eqref{eq:ql}, indeed $\ql=x_*$. 

\eqref{ql,p>p*} Let us now verify part~\eqref{ql,p>p*} of Proposition~\ref{prop:ql}. 
Toward that end, assume indeed that $\al\in(1,\infty)$ and $p\in(p_*,1)$. Then, by \eqref{eq:inv}, $Q_\al(X;p)\in(-\infty,x_*)$. By \eqref{eq:B}, $t\le B_\al(X;p)(t)$ for all real $t$. 
So, in view of \eqref{eq:ql} and \eqref{eq:amin}, $\ql\le Q_\al(X;p)<x_*$. 
By \eqref{eq:ql in R}, one now has 
$\ql\in(-\infty,x_*)$. 
Since $\al\in(1,\infty)$, $B(t):=B_\al(X;p)(t)$ is differentiable in $t\in(-\infty,x_*)$, 
with the derivative
\begin{equation}\label{eq:B'}
B'(t)=1-\Big(\frac{\pl\al}p\Big)^{1/\al}\quad\text{for}\quad t\in(-\infty,x_*). 
\end{equation}
It follows by \eqref{eq:ql} that $\ql$ is a root $t\in(-\infty,x_*)$ of the equation $B'(t)=0$, which can be rewritten as $\pl\al=p$. 
Let us show that such a root $t$ is unique and \eqref{eq:ql<x**} holds.

If $x_{**}<x_*$, it follows by \eqref{eq:pl} that $\pl\al=p_*<p$ for all $t\in[x_{**},x_*)$ and hence, by \eqref{eq:B'}, $B'(t)>0$ for all such $t$. 
So, all roots $t\in(-\infty,x_*)$ of the equation $B'(t)=0$ are in fact in the interval $(-\infty,x_{**})$. 
On the other hand, by part~\eqref{B str conv} of Proposition~\ref{prop:B str conv},  there is at most one root $t\in(-\infty,x_{**})$ of the equation $B'(t)=0$ or, equivalently, of the equation $\pl\al=p$. 
Since $\ql$ is such a root, one obtains \eqref{eq:ql<x**} -- in the case when $x_{**}<x_*$. 
%
Relations \eqref{eq:ql<x**} hold in the remaining case when $x_{**}=x_*$, since, as established above, $\ql\in(-\infty,x_*)$. 
Thus, the first two sentences of part~\eqref{ql,p>p*} of Proposition~\ref{prop:ql} are verified. 

Concerning the third sentence there, assume that indeed $t\in(-\infty,x_*)$. 
Then the function $g$ defined by the formula $g(\ga):=\ln\E\big((X-t)^\ga|X>t\big)$ is convex on $[0,\infty)$, with $g(0)=0$.
So, by \eqref{eq:pl}, 
\begin{equation}\label{eq:g ineq} 
	\frac1\al\ln\,\frac{\pl\al}{\P(X>t)}
	=g(\al-1)-\Big(\frac{\al-1}\al\,g(\al)+\frac1\al\,g(0)\Big)\le0, 
\end{equation}
which shows that indeed $\pl\al\le\P(X>t)$. 
If $t\in(-\infty,x_{**})$, then the interval $(t,\infty)$ contains at least two distinct points of $\supp X$, whence  
the function $g$ is strictly convex on $[0,\infty)$, which makes the inequality in \eqref{eq:g ineq} strict, so that the strict inequality $\pl\al<\P(X>t)$ holds. 
%
%
The inequality $\P(X>t)\le P_0(X;t)$ obviously follows from the equality in \eqref{eq:0<al}. 
The relation $\pl\al\to\P(X>t)$ as $\al\downarrow1$ easily follows from \eqref{eq:pl} by dominated convergence. 

Thus, part~\eqref{ql,p>p*} of Proposition~\ref{prop:ql} is completely proved. 

\eqref{ql decr in p} Let us turn to part~\eqref{ql decr in p} of the proposition. 
That $\ql=x_*$ for $p\in(0,p_*]\cap(0,1)$ follows immediately from part~\eqref{ql,p<p*} of Proposition~\ref{prop:ql}. 
Also, by the first equality in \eqref{eq:ql<Q0} and \eqref{eq:Q decr in p}, $\ql$ is nonincreasing in $p\in(0,1)$ if $\al=1$. 

Assume now that $\al\in(1,\infty)$. 
Note that $\pl\al$ is continuous in $t\in(-\infty,x_*)$. 
So, if $x_{**}<x_*$, then $\pl\al\underset{t\uparrow x_{**}}\longrightarrow{{}_{\al-1}P(X;x_{**})}=p_*$. 
On the other hand, by the inequality $\pl\al\le\P(X>t)$ 
for all $t\in(-\infty,x_*)$, one has 
$\pl\al\underset{t\uparrow x_*}\longrightarrow0=p_*$ if $p_*=0$. 
In the remaining case, when $x_{**}=x_*$ and $p_*>0$, for all $\ga\in(0,\infty)$, all $X\in\XX_\ga$, all $t\in(-\infty,x_*)$, and some $\th_{\ga,t}\in[0,1]$ one can write 
$$\E(X-t)_+^\ga=p_*(x_*-t)^\ga+\th_{\ga,t}\P(t<X<x_*)(x_*-t)^\ga
\underset{t\uparrow x_*}\sim p_*(x_*-t)^\ga,$$whence $\pl\al\sim p_*$ as $t\uparrow x_*=x_{**}$. 
Thus, in any case $\pl\al\underset{t\uparrow x_{**}}\longrightarrow p_*$. 

Next, for all $\ga\in(0,\infty)$, all $X\in\XX_\ga$, and all $t\in(-\infty,0)$, by the monotone convergence  
$\E(X-t)_+^\ga=|t|^\ga\E(1+X/|t|)_+^\ga\underset{t\downarrow-\infty}\sim|t|^\ga$, whence 
$\pl\al\underset{t\downarrow-\infty}\longrightarrow1$. 
So, the function $t\longmapsto\,\pl\al$ maps the interval $(-\infty,x_{**})$ continuously onto the interval $(p_*,1)$. Moreover, by part~\eqref{ql,p>p*} of Proposition~\ref{prop:ql}, this function is one-to-one. Furthermore, by \eqref{eq:B'}, this function is nonincreasing, because $B$ is convex and hence $B'$ is nondecreasing. 
It then follows 
that this function is strictly decreasing on the interval $(-\infty,x_{**})$. 
One concludes that the function $(p_*,1)\ni p\longmapsto\,\ql=t_{\al,p}\in(-\infty,x_{**})$ is a bijection, which is the strictly decreasing continuous inverse to the strictly decreasing continuous bijection $(-\infty,x_{**})\ni t\longmapsto\,\pl\al\in(p_*,1)$. 
This completes the proof of part~\eqref{ql decr in p} of Proposition~\ref{prop:ql}. 

%
%

\eqref{ql decr in al} Consider now part~\eqref{ql decr in al} of Proposition~\ref{prop:ql}. 
The first equality in \eqref{eq:ql<Q0} is \eqref{eq:qlz=Q0}, and the second equality there follows by the definition of $Q(X;p)$ in \eqref{eq:Q_0}. 

The non-strict inequality in \eqref{eq:ql<Q0} is a trivial equality if $\al=1$. 
Take now any $\al\in(1,\infty)$ and any $t\in\big(Q_0(X;p),x_*\big)$. Then, again by \eqref{eq:Q_0}, $\P(X\ge t)<p$. Combining this with the inequality $\pl\al\le\P(X>t)$ for $t\in(-\infty,x_*)$, established in part~\eqref{ql,p>p*} of Proposition~\ref{prop:ql}, one has $\pl\al<p$ and hence, by \eqref{eq:B'}, $B'(t)>0$, for all $t\in\big(Q_0(X;p),x_*\big)$. 
Also, $B(t)=t$ 
for all $t\in[x_*,\infty)$. 
So, $B(t)$ is strictly increasing in 
$t\in\big(Q_0(X;p),\infty)$ and therefore,  
in view of \eqref{eq:ql}, $\ql\le Q_0(X;p)$, which completes the proof of \eqref{eq:ql<Q0}. 

In the case when $p\in(0,p_*]\cap(0,1)$, 
$1\le\al<\infty$, and $X\in\XX_\al$, one has $\ql=x_*$ by part~\eqref{ql,p<p*} of Proposition~\ref{prop:ql}. 

Assume now that $p\in(p_*,1)$, 
$1<\al<\be<\infty$, and $X\in\XX_\be$. 
Take first any $t\in(-\infty,x_*)$. The condition $p\in(p_*,1)$ implies $p_*<1$, and so, $\P\big((X-t)_+=c\big)\ne1$ for any $c\in\R$. 
Hence, the function $h$ defined by the formula $h(\ga):=h_t(\ga):=\ln\E(X-t)_+^\ga$ is strictly convex on the 
interval $(0,\be]$. 
Noting also that $\al-1<\be-1<\be$ and $\al-1<\al<\be$, and recalling \eqref{eq:pl}, one can now write 
\begin{align*}
	&\ln\,\frac{\pl\be}{\pl\al} \\
	&=\be h(\be-1)+(\al-1)h(\al)-(\be-1)h(\be)-\al h(\al-1) \\ 
	&<\be\,\frac{h(\al-1)+(\be-\al)h(\be)}{\be-(\al-1)}
	+(\al-1)\,\frac{(\be-\al)h(\al-1)+h(\be)}{\be-(\al-1)}-(\be-1)h(\be)-\al h(\al-1)=0, 
\end{align*}
so that $\pl\be<\pl\al$, for each $t\in(-\infty,x_*)$. 
Using now part~\eqref{ql,p>p*} of Proposition~\ref{prop:ql}, one sees that 
\begin{equation}\label{eq:pl>pl}
	{}_{\al-1}P(X;t_{\be,p})>{}_{\be-1}P(X;t_{\be,p})=p. 
\end{equation}
On the other hand, by the convexity of $B(t)$ in $t$, 
$B'(t)$ is 
nondecreasing in $t\in(-\infty,X_*)$ and hence, by \eqref{eq:B'}, $\pl\al$ is 
nonincreasing in $t\in(-\infty,
x_*)$. 
Therefore, one would have $p={}_{\al-1}P(X;t_{\al,p})\ge{}_{\al-1}P(X;t_{\be,p})$ if it were true that $t_{\al,p}\le t_{\be,p}$, which would then contradict \eqref{eq:pl>pl}. 
Thus, $t_{\be,p}<t_{\al,p}$. 
In view of part~\eqref{ql,p>p*} of Proposition~\ref{prop:ql}, the latter inequality means that $\qlbe<\ql$ -- assuming that $p\in(p_*,1)$, $1<\al<\be<\infty$, and $X\in\XX_\be$. Therefore and by the inequality in \eqref{eq:ql<Q0}, 
one can write $\qlbe<{}_{\tilde\be-1}Q(X;p)\le\qlz$ for any $\tilde\be\in(1,\be)$, so that 
the inequality $\qlbe<\ql$ still holds if $p\in(p_*,1)$, $1=\al<\be<\infty$, and $X\in\XX_\be$. 

Concerning part~\eqref{ql decr in al} of Proposition~\ref{prop:ql}, it remains to show that, if $p\in(p_*,1)$ and $X\in\XX_\infty$, then $\ql\underset{\al\uparrow\infty}\longrightarrow-\infty$. 
To obtain a contradiction, assume the contrary: $p\in(p_*,1)$ and $X\in\XX_\infty$ but $\ql$ does not converge to $-\infty$ as $\al\uparrow\infty$. 
Because $\ql$ is nonincreasing in $\al\in[1,\infty)$ and in view of \eqref{eq:ql<x**}, 
$t_{\al,p}=\ql\underset{\al\uparrow\infty}\longrightarrow t_*$ for some $t_*\in(-\infty,x_{**})$ and hence $Q_\al(X;p)=B_\al(X;p)(t_{\al,p})
=t_{\al,p}+p^{-1/\al}\|(X-t_{\al,p})_+\|_\al
\underset{\al\uparrow\infty}\longrightarrow t_*+\|(X-t_*)_+\|_\infty
=t_*+(x_*-t_*)_+=x_*$. 
On the other hand, by 
the stability of $Q_\al(X;p)$ in $\al$ stated in Theorem~\ref{th:coher} and \eqref{eq:inv}, 
$Q_\al(X;p)
\underset{\al\uparrow\infty}\longrightarrow Q_\infty(X;p)<x_*$. 
This contradiction 
completes the proof of part~\eqref{ql decr in al} of Proposition~\ref{prop:ql}. 

\eqref{ql consist} If $X=c$ for some $c\in\R$ then $x_*=c$ and $p_*=1$, so that part~\eqref{ql consist} of Proposition~\ref{prop:ql} follows immediately from its part~\eqref{ql,p<p*}. 

\eqref{ql pos-hom} Concerning the positive homogeneity of $\ql$ stated in part~\eqref{ql pos-hom} of Proposition~\ref{prop:ql}, the case $\ka=0$ follows immediately by the consistency of $\ql$ and \eqref{eq:ql in R}.  
%
The 
case of any real $\ka>0$ follows by \eqref{eq:ql} and the identity $B_\al(\ka X;p)(\ka t)=\ka B_\al(X;p)(t)$ for all $t\in\R$.   

\eqref{ql tr-inv} The translation invariance of $\ql$ stated in part~\eqref{ql tr-inv} of Proposition~\ref{prop:ql}, follows immediately by \eqref{eq:ql} and the identity $B_\al(X+c;p)(t+c)=B_\al(X;p)(t)+c$ for all $t\in\R$. 

\eqref{ql part mono} If $X\le c$ for some $c\in\R$, then $x_*\le c$, and so, by part~\eqref{ql comp} of Proposition~\ref{prop:ql}, $\ql\le x_*\le c$. On the other hand, by the consistency property of $\ql$ stated in part~\eqref{ql consist} of Proposition~\ref{prop:ql}, ${}_{\al-1}Q(c;p)=c$. 
So, indeed $X\le c$ implies $\ql\le{}_{\al-1}Q(c;p)$. 
The inequality $\ql\le{}_{\al-1}Q(X+c;p)$ for any $c\in[0,\infty)$ follows immediately from the translation invariance of $\ql$ stated in part~\eqref{ql tr-inv} of Proposition~\ref{prop:ql}. 
Thus, part~\eqref{ql part mono} of Proposition~\ref{prop:ql} is checked. 

\eqref{ql not mono} To verify part~\eqref{ql not mono} of Proposition~\ref{prop:ql}, take indeed any $\al\in(1,\infty)$ and any $p\in(0,1)$. 
Take then any $p_*\in(0,p)$, so that $p\in(p_*,1)$. 
Let $Y=0$ and let $X$ be a r.v.\ taking values $0$ and $1$ with probabilities $1-p_*$ and $p_*$, respectively, so that $Y\le X$ and, of course, $X$ and $Y$ are in $\XX_\al$ for any $\al\in[0,\infty]$. Also, by \eqref{eq:x_**}, $x_{**,X}=0$. 
So, in view of relations \eqref{eq:ql<x**} and part~\eqref{ql consist} of Proposition~\ref{prop:ql}, 
$\ql<x_{**,X}=0={}_{\al-1}Q(0;p)={}_{\al-1}Q(Y;p)$. 
If now $Z=1$, then $Y\le Z$ and ${}_{\al-1}Q(Y;p)=0<1={}_{\al-1}Q(Z;p)$. 
Thus, $Y\le X$ and $Y\le Z$, whereas $\ql<{}_{\al-1}Q(Y;p)<{}_{\al-1}Q(Z;p)$, which means that 
${}_{\al-1}Q(\cdot;p)$ is not monotonic. 

\eqref{ql not subadd} The case $\al=1$ of part~\eqref{ql not subadd} of Proposition~\ref{prop:ql} follows, in view of \eqref{eq:qlz=Q0}, because $Q_0(X;p)$ is well known not to be subadditive or convex. 
Take now any $\al\in(1,\infty)$ and any $p\in(0,1)$. 
To verify part~\eqref{ql not subadd} of Proposition~\ref{prop:ql} for such $\al$ and $p$, let us use an idea from \cite{rocka-ur-zab}, which allows one to show that the non-subadditivity follows from the non-monotonicity and partial monotonicity.  
Thus, 
let $Y$ and $X$ be as in the above proof of part~\eqref{ql not mono} of Proposition~\ref{prop:ql}. Let $V:=Y-X$, so that $V\le0$ and hence, by  part~\eqref{ql part mono} of Proposition~\ref{prop:ql}, ${}_{\al-1}Q(V;p)\le0$. 
It follows that ${}_{\al-1}Q(V;p)+\ql\le\ql<{}_{\al-1}Q(Y;p)={}_{\al-1}Q(V+X;p)$. 
So, ${}_{\al-1}Q(\cdot;p)$ is not subadditive. Since${}_{\al-1}Q(\cdot;p)$ is positive homogeneous, it is not convex either. 

The proof of Proposition~\ref{prop:ql} is now quite complete. 
\end{proof}

\begin{proof}[Proof of Proposition~\ref{prop:gini}]
To prove the ``if'' part of the proposition, suppose that $H$ is $\frac12$-Lipschitz and take any r.v.'s $X$ and $Y$ such that $X\st Y$. We have to show that then $R_H(X)\le R_H(Y)$.  By \eqref{eq:st iff} and because $R_H(X)$ depends only on the distribution of $X$, w.l.o.g.\ $X\le Y$. Let $(\tX,\tY)$ be an independent copy of the pair $(X,Y)$. 
Then, by \eqref{eq:R_H}, the $\frac12$-Lipschitz condition, the triangle inequality, and the condition $X\le Y$, 
\begin{align*}
	R_H(X)-R_H(Y)&=\E(X-Y)+\E H(|X-\tX|)-\E H(|Y-\tY|) \\ 
	&\le\E(X-Y)+\tfrac12\,\E(|X-\tX|-|Y-\tY|) \\ 
	&\le\E(X-Y)+\tfrac12\,\E|X-\tX-Y+\tY| \\ 
	&\le\E(X-Y)+\tfrac12\,\E(|X-Y|+|\tX-\tY|) \\ 
	&=\E(X-Y)+\E|X-Y|=\E(X-Y)+\E(Y-X)=0, 	
\end{align*}
so that the ``if'' part of Proposition~\ref{prop:gini} is verified. 

To prove the ``only if'' part of the proposition, suppose that $R_H(X)$ is nondecreasing in $X$ with respect to the stochastic dominance of order $1$ and take any $x$ and $y$ in $[0,\infty)$ such that $x<y$. It is enough to show that then $|H(x)-H(y)|\le\frac12\,(y-x)$. 
Take also an arbitrary $p\in(0,1)$. 
Let 
$X$ and $Y$ be such r.v.'s that $\P(X=0)=1$ if $x=0$, 
$\P(X=x)=p=1-\P(X=0)$ if $x\in(0,\infty)$, and $\P(Y=y)=p=1-\P(Y=0)$. Then $X\st Y$, whence, by \eqref{eq:R_H},  
$
	0\ge\tfrac1p\,[R_H(X)-R_H(Y)]=x-y+2(1-p)[H(x)-H(y)], 
$ 
which yields $H(x)-H(y)\le\frac1{2(1-p)}\,(y-x)$ for an arbitrary $p\in(0,1)$ and hence 
\begin{equation}\label{eq:H(x)-H(y)}
	H(x)-H(y)\le\tfrac12\,(y-x). 
\end{equation}
Similarly, letting now $X$ and $Y$ be such r.v.'s that $\P(X=-y)=p=1-\P(X=0)$,  $\P(Y=0)=1$ if $x=0$, and $\P(Y=-x)=p=1-\P(Y=0)$ if $x\in(0,\infty)$, one has 
$X\st Y$ and hence   
$
	0\ge\tfrac1p\,[R_H(X)-R_H(Y)]=-y+x+2(1-p)[H(y)-H(x)], 
$ 
which yields 
$H(y)-H(x)\le\tfrac12\,(y-x)$. 
Thus, by \eqref{eq:H(x)-H(y)}, $|H(x)-H(y)|\le\frac12\,(y-x)$. 
\end{proof}

\begin{proof}[Proof of Proposition~\ref{prop:gini coher}]
To prove the ``if'' part of the proposition, suppose that $H=\ka\id$ for some $\ka\in[0,\frac12]$. 
We have to check that then $R_H(X)$ has the translation
invariance, subadditivity, positive homogeneity, and monotonicity properties and thus is coherent. 
As noted in the discussion in Section~\ref{risk}, $R_H(X)$ is translation
invariant for any function $H$. It is also obvious that $R_{\ka\id}(X)$ is positive homogeneneous for any $\ka\in[0,\infty)$. 
Next, as also noted in the discussion in Section~\ref{risk}, $R_H(X)$ is convex in $X$ whenever the function $H$ is convex and nondecreasing. Indeed, let then $(\tX_0,\tX_1)$ be an independent copy in distribution of a pair $(X_0,X_1)$ of r.v.'s, and introduce $X_\la:=(1-\la)X_0+\la X_1$ and $\tX_\la:=(1-\la)\tX_0+\la\tX_1$, for an arbitrary $\la\in(0,1)$. Then 
\begin{align*}
	R_H(X_\la)&=\E X_\la+\E H(|X_\la-\tX_\la|) \\ 
	&=(1-\la)\E X_0+\la\E X_1+\E H\big(|(1-\la)(X_0-\tX_0)+\la(X_1-\tX_1)|\big) \\
	&\le(1-\la)\E X_0+\la\E X_1+\E H\big((1-\la)|X_0-\tX_0|+\la|X_1-\tX_1|\big) \\ 
	&\le(1-\la)\E X_0+\la\E X_1+(1-\la)\E H(|X_0-\tX_0|)+\la\E H(|X_1-\tX_1|) \\ 
	&=(1-\la)R_H(X_0)+\la R_H(X_1). 
\end{align*}
So, the convexity property of $R_H(X)$ is verified, which, as noted earlier, is equivalent to the subadditivity given the positive homogeneity. 
Now, to finish the proof of ``if'' part of Proposition~\ref{prop:gini coher}, it remains to notice that the monotonicity property of $R_{\ka\id}(X)$ for $\ka\in[0,\frac12]$ follows immediately from Proposition~\ref{prop:gini}. 

To prove the ``only if'' part of the proposition, suppose that the function $H$ is such that $R_H(X)$ is coherent and thus positive homogeneous, monotonic, and subadditive (as noted before, $R_H(X)$ is translation
invariant for any $H$). 
Take any $p\in(0,1)$ and let $X$ here be a r.v.\ such that $\P(X=1)=p=1-\P(X=0)$. 
Then, by the positive homogeneity, for any real $u>0$ one has 
\begin{equation*}
	0=R_H(u X)-u R_H(X)=aA+B,  
\end{equation*}
where $B:=(1-u)H(0)$, $A:=H(u)-uH(1)-B$, and $a:=2p(1-p)$, so that 
the range of values of $a$ is the entire interval $(0,\frac12)$ as $p$ varies in the interval $(0,1)$. Thus, $aA+B=0$ for all $a\in(0,\frac12)$. 
On the other hand, $aA+B$ is a polynomial in $a$, with coefficients $A$ and $B$ not depending on $a$. It follows that $A=B=0$, which yields $H(u)=uH(1)$ for all $u\in(0,\infty)
$ and $H(0)=0$. Hence, $H(u)=uH(1)$ for all real $u\ge0$. 
In other words, $H=\ka\id$, with $\ka:=H(1)$. 
Then the monotonicity property and Proposition~\ref{prop:gini} imply that $|\ka|\le\frac12$. 
It remains to show that necessarily $\ka\ge0$. Take here $X$ and $Y$ to be independent standard normal r.v.'s. Then, by the subadditivity, 
\begin{equation*}
	2\ka\E|X|=R_{\ka\id}(X+Y)\le R_{\ka\id}(X)+R_{\ka\id}(Y)=2\sqrt2\,\ka\E|X|, 
\end{equation*}
whence indeed $\ka\ge0$. 
\end{proof}

\bibliographystyle{abbrv}


\bibliography{C:/Users/Iosif/Dropbox/mtu/bib_files/citations12.13.12}

\def\cprime{$'$} \def\polhk#1{\setbox0=\hbox{#1}{\ooalign{\hidewidth
  \lower1.5ex\hbox{`}\hidewidth\crcr\unhbox0}}}
  \def\polhk#1{\setbox0=\hbox{#1}{\ooalign{\hidewidth
  \lower1.5ex\hbox{`}\hidewidth\crcr\unhbox0}}}
  \def\polhk#1{\setbox0=\hbox{#1}{\ooalign{\hidewidth
  \lower1.5ex\hbox{`}\hidewidth\crcr\unhbox0}}} \def\cprime{$'$}
  \def\polhk#1{\setbox0=\hbox{#1}{\ooalign{\hidewidth
  \lower1.5ex\hbox{`}\hidewidth\crcr\unhbox0}}}
  \def\polhk#1{\setbox0=\hbox{#1}{\ooalign{\hidewidth
  \lower1.5ex\hbox{`}\hidewidth\crcr\unhbox0}}}
\begin{thebibliography}{10}

\bibitem{acerbi02}
C.~Acerbi.
\newblock Spectral measures of risk: A coherent representation of subjective
  risk aversion.
\newblock {\em Journal of Banking \& Finance}, 26:1505--1518, 2002.

\bibitem{acerbi-tasche02}
C.~Acerbi and D.~Tasche.
\newblock Expected shortfall: a natural coherent alternative to value at risk.
\newblock {\em Economic Notes}, 31:379--388, 2002.

\bibitem{artzner.etal.99}
P.~Artzner, F.~Delbaen, J.-M. Eber, and D.~Heath.
\newblock Coherent measures of risk.
\newblock {\em Math. Finance}, 9(3):203--228, 1999.

\bibitem{atkinson70}
A.~B. Atkinson.
\newblock On the measurement of inequality.
\newblock {\em J. Econom. Theory}, 2:244--263, 1970.

\bibitem{atkinson08}
A.~B. Atkinson.
\newblock More on the measurement of inequality.
\newblock {\em J. Econ. Inequal.}, 6:277--283, 2008.

\bibitem{bassi_etal}
F.~Bassi, P.~Embrechts, and M.~Kafetzaki.
\newblock Risk management and quantile estimation.
\newblock In {\em A practical guide to heavy tails ({S}anta {B}arbara, {CA},
  1995)}, pages 111--130. Birkh\"auser Boston, Boston, MA, 1998.

\bibitem{bent-liet02}
V.~Bentkus.
\newblock A remark on the inequalities of {B}ernstein, {P}rokhorov, {B}ennett,
  {H}oeffding, and {T}alagrand.
\newblock {\em Liet. Mat. Rink.}, 42(3):332--342, 2002.

\bibitem{bent-ap}
V.~Bentkus.
\newblock On {H}oeffding's inequalities.
\newblock {\em Ann. Probab.}, 32(2):1650--1673, 2004.

\bibitem{bent-64pp}
V.~Bentkus, N.~Kalosha, and M.~van Zuijlen.
\newblock On domination of tail probabilities of (super)martingales: explicit
  bounds.
\newblock {\em Liet. Mat. Rink.}, 46(1):3--54, 2006.

\bibitem{billingsley}
P.~Billingsley.
\newblock {\em Convergence of probability measures}.
\newblock John Wiley \& Sons Inc., New York, 1968.

\bibitem{cillo-delq11}
A.~Cillo and P.~Delquie.
\newblock Mean-risk analysis with enhanced behavioral content.
\newblock Technical Report TR-2011-11, Institute for Integrating Statistics in
  Decision Sciences, George Washington University, May 2011.

\bibitem{clarkson}
J.~A. Clarkson.
\newblock Uniformly convex spaces.
\newblock {\em Trans. Amer. Math. Soc.}, 40(3):396--414, 1936.

\bibitem{degiorgi}
E.~De~Giorgi.
\newblock Reward-risk portfolio selection and stochastic dominance.
\newblock {\em Journal of Banking and Finance}, 29:895--926, 2005.

\bibitem{delq-cillo06}
P.~Delquié and A.~Cillo.
\newblock Disappointment without prior expectation: A unifying perspective on
  decision under risk.
\newblock {\em Journal of Risk and Uncertainty}, 33:197--215, 2006.

\bibitem{dufour-hallin}
J.-M. Dufour and M.~Hallin.
\newblock Improved {E}aton bounds for linear combinations of bounded random
  variables, with statistical applications.
\newblock {\em J. Amer. Statist. Assoc.}, 88:1026--1033, 1993.

\bibitem{eaton2}
M.~L. Eaton.
\newblock A probability inequality for linear combinations of bounded random
  variables.
\newblock {\em Ann. Statist.}, 2:609--613, 1974.

\bibitem{embrechts_etal}
P.~Embrechts, C.~Kl\"uppelberg, and T.~Mikosch.
\newblock {\em Modelling Extremal Events for Insurance and Finance}.
\newblock Springer, New York, 1997.

\bibitem{fishburn76}
P.~C. Fishburn.
\newblock Continua of stochastic dominance relations for bounded probability
  distributions.
\newblock {\em J. Math. Econom.}, 3(3):295--311, 1976.

\bibitem{fishburn77}
P.~C. Fishburn.
\newblock Mean-risk analysis with risk associated with below-target returns.
\newblock {\em The American Economic Review,}, 67(2):116--126, 1977.

\bibitem{fishburn80}
P.~C. Fishburn.
\newblock Continua of stochastic dominance relations for unbounded probability
  distributions.
\newblock {\em J. Math. Econom.}, 7(3):271--285, 1980.

\bibitem{frit-gian04}
M.~Frittelli and E.~Rosazza~Gianin.
\newblock Dynamic convex risk measures.
\newblock In {\em Risk Measures in the 21st century}, pages 227--248. Wiley,
  2004.

\bibitem{giacom-orto04}
R.~Giacometti and S.~Ortobelli.
\newblock Risk measures for asset allocation models.
\newblock In {\em Risk Measures in the 21st century}, pages 69--86. Wiley,
  2004.

\bibitem{groot-haller04}
H.~Grootveld and W.~G. Hallerbach.
\newblock Upgrading value-at-risk from diagnostic metric to decision variable:
  a wise thing to do?
\newblock In {\em Risk Measures in the 21st century}, pages 33--50. Wiley,
  2004.

\bibitem{kibzun-chern13}
A.~Kibzun and A.~Chernobrovov.
\newblock Equivalence of the problems with quantile and integral quantile
  criteria.
\newblock {\em Automation and Remote Control}, 74(2):225--239, 2013.

\bibitem{kibzun-kuzn}
A.~Kibzun and E.~A. Kuznetsov.
\newblock Comparison of {VaR} and {CVaR} criteria.
\newblock {\em Automation and Remote Control}, 64(7):1154--1164, 2003.

\bibitem{litz-modest08}
R.~H. Litzenberger and D.~M. Modest.
\newblock Crisis and non-crisis risk in financial markets: A unified approach
  to risk management.
\newblock In {\em The known, the unknown, and unknowable in financial risk
  management}, pages 74--102. Princeton University Press, 2008.

\bibitem{machina82}
M.~J. Machina.
\newblock Expected utility analysis without the independence axiom.
\newblock {\em Econometrica}, 50:277--323, 1982.

\bibitem{mausser-rosen}
R.~D. Mausser, H.
\newblock Efficient risk/return frontiers for credit risk.
\newblock {\em Algo Research Quarterly}, 2(4):35--47, 1999.

\bibitem{muliere-scarsini}
P.~Muliere and M.~Scarsini.
\newblock A note on stochastic dominance and inequality measures.
\newblock {\em J. Econom. Theory}, 49(2):314--323, 1989.

\bibitem{ogr-rusz04_SIAM}
W.~Ogryczak and A.~Ruszczy\'nski.
\newblock Dual stochastic dominance and related mean-risk models.
\newblock {\em SIAM Journal of Optimization}, 13(1):60--78, 2002.

\bibitem{ortobelli_etal06}
S.~Ortobelli, S.~T. Rachev, H.~Shalit, and F.~J. Fabozzi.
\newblock The theory of orderings and risk probability functionals.
\newblock
  \url{http://www.google.com/url?sa=t&rct=j&q=&esrc=s&source=web&cd=1&cad=rja&ved=0CDIQFjAA&url=http%3A%2F%2Fwww.pstat.ucsb.edu%2Fresearch%2Fpapers%2FRisk_Measures_and_Orderings.pdf&ei=3SU2Usf9BePL2gXe6ICwCg&usg=AFQjCNHNvCyKtIFmf1hYlJvgY5ie9e-M_w&sig2=7LnFKo_fRBfeA18-GIHJ1g&bvm=bv.52164340,d.b2I},
  2006.

\bibitem{ortobelli_etal07}
S.~Ortobelli, S.~T. Rachev, H.~Shalit, and F.~J. Fabozzi.
\newblock Risk probability functionals and probability metrics applied to
  portfolio theory.
\newblock \url{www.pstat.ucsb.edu/research/papers/2006mid/view.pdf}, 2007.

\bibitem{ortobelli_etal09}
S.~Ortobelli, S.~T. Rachev, H.~Shalit, and F.~J. Fabozzi.
\newblock Orderings and probability functionals consistent with preferences.
\newblock {\em Applied Mathematical Finance,}, 16:81--102, 2009.

\bibitem{pearson02}
N.~D. Pearson.
\newblock {\em Risk Budgeting: Portfolio Problem Solving with Value-at-Risk}.
\newblock Wiley, New York, 2002.

\bibitem{pflug00}
G.~C. Pflug.
\newblock Some remarks on the value-at-risk and the conditional value-at-risk.
\newblock In {\em Probabilistic constrained optimization}, volume~49 of {\em
  Nonconvex Optim. Appl.}, pages 272--281. Kluwer Acad. Publ., Dordrecht, 2000.

\bibitem{T2original}
I.~Pinelis.
\newblock Extremal probabilistic problems and {H}otelling's ${T}^2$ test under
  symmetry condition.
\newblock \url{http://arxiv.org/abs/math/0701806}, 1991.

\bibitem{T2}
I.~Pinelis.
\newblock Extremal probabilistic problems and {H}otelling's {$T^2$} test under
  a symmetry condition.
\newblock {\em Ann. Statist.}, 22(1):357--368, 1994.

\bibitem{pin98}
I.~Pinelis.
\newblock Optimal tail comparison based on comparison of moments.
\newblock In {\em High dimensional probability ({O}berwolfach, 1996)},
  volume~43 of {\em Progr. Probab.}, pages 297--314. Birkh\"auser, Basel, 1998.

\bibitem{pin99}
I.~Pinelis.
\newblock Fractional sums and integrals of {$r$}-concave tails and applications
  to comparison probability inequalities.
\newblock In {\em Advances in stochastic inequalities ({A}tlanta, {GA}, 1997)},
  volume 234 of {\em Contemp. Math.}, pages 149--168. Amer. Math. Soc.,
  Providence, RI, 1999.

\bibitem{binom}
I.~Pinelis.
\newblock Binomial upper bounds on generalized moments and tail probabilities
  of (super)martingales with differences bounded from above.
\newblock In {\em High dimensional probability}, volume~51 of {\em IMS Lecture
  Notes Monogr. Ser.}, pages 33--52. Inst. Math. Statist., Beachwood, OH, 2006.

\bibitem{normal}
I.~Pinelis.
\newblock On normal domination of (super)martingales.
\newblock {\em Electron. J. Probab.}, 11:no. 39, 1049--1070, 2006.

\bibitem{asymm}
I.~Pinelis.
\newblock Exact inequalities for sums of asymmetric random variables, with
  applications.
\newblock {\em Probab. Theory Related Fields}, 139(3-4):605--635, 2007.

\bibitem{pin-hoeff}
I.~Pinelis.
\newblock On the {B}ennett-{H}oeffding inequality, \emph{a shorter version to
  appear in} \emph{{A}nnales de l'{I}nstitut {H}enri {P}oincar\'e}.
\newblock \url{http://arxiv.org/abs/0902.4058}, 2009.

\bibitem{positive}
I.~Pinelis.
\newblock Positive-part moments via the {F}ourier–-{L}aplace transform.
\newblock {\em J. Theor. Probab.}, 24:409--421, 2011.

\bibitem{inf-stable}
I.~Pinelis.
\newblock A necessary and sufficient condition on the stability of the infimum
  of convex functions.
\newblock \url{http://arxiv.org/abs/1307.3806}, 2013.

\bibitem{pin-yt}
I.~Pinelis.
\newblock ({Q}uasi)additivity properties of the {L}egendre--{F}enchel transform
  and its inverse, with applications in probability.
\newblock \url{http://arxiv.org/abs/1305.1860}, 2013.

\bibitem{rachev-etal}
S.~T. Rachev, S.~Stoyanov, and F.~J. Fabozzi.
\newblock {\em Advanced Stochastic Models, Risk Assessment, and Portfolio
  Optimization: The Ideal Risk, Uncertainty, and Performance Measures}.
\newblock John Wiley, 2007.

\bibitem{rio}
E.~Rio.
\newblock Local invariance principles and their application to density
  estimation.
\newblock {\em Probab. Theory Related Fields}, 98(1):21--45, 1994.

\bibitem{rio.transl}
E.~Rio.
\newblock English translation of the monograph \emph{Th\'eorie asymptotique des
  processus al\'eatoires faiblement d\'ependants} (2000) by {E}.\ {R}io.
\newblock Work in progress, 2012.

\bibitem{rio06.17.13}
E.~Rio.
\newblock Personal communication.
\newblock 2013.

\bibitem{rocka}
R.~T. Rockafellar.
\newblock {\em Convex analysis}.
\newblock Princeton Landmarks in Mathematics. Princeton University Press,
  Princeton, NJ, 1997.
\newblock Reprint of the 1970 original, Princeton Paperbacks.

\bibitem{rocka-ur00}
R.~T. Rockafellar and S.~Uryasev.
\newblock Optimization of conditional value-at-risk.
\newblock {\em Journal of Risk}, 2:21--41, 2000.

\bibitem{rocka-ur02}
R.~T. Rockafellar and S.~Uryasev.
\newblock Conditional value-at-risk for general loss distributions.
\newblock {\em Journal of Banking \& Finance}, 26:1443--1471, 2002.

\bibitem{rocka-ur-zab}
R.~T. Rockafellar, S.~Uryasev, and M.~Zabarankin.
\newblock Generalized deviations in risk analysis.
\newblock {\em Finance Stoch.}, 10(1):51--74, 2006.

\bibitem{roy52}
A.~D. Roy.
\newblock Safety first and the holding of assets.
\newblock {\em Econometrica}, 20:431--449, 1952.

\bibitem{shaked-shanti}
M.~Shaked and J.~G. Shanthikumar.
\newblock {\em Stochastic orders}.
\newblock Springer Series in Statistics. Springer, New York, 2007.

\bibitem{yaari87}
M.~E. Yaari.
\newblock The dual theory of choice under risk.
\newblock {\em Econometrica}, 55(1):95--115, 1987.

\bibitem{yitzhaki82}
S.~Yitzhaki.
\newblock Stochastic dominance, mean variance, and {G}ini's mean difference.
\newblock {\em American Economic Review}, 72:178--185, 1982.

\end{thebibliography}

\end{document}